\documentclass{siamart1116}
\usepackage{amsmath,amssymb,ifthen}
\usepackage{graphicx,psfrag,subfigure}
\usepackage{lipsum}
\usepackage{amsfonts}
\usepackage{graphicx}
\usepackage{epstopdf}
\usepackage{algorithmic}
\ifpdf
  \DeclareGraphicsExtensions{.eps,.pdf,.png,.jpg}
\else
  \DeclareGraphicsExtensions{.eps}
\fi

\newcommand{\TheTitle}{The Eddy Current--LLG Equations: FEM-BEM Coupling and A Priori Error Estimates} 
\newcommand{\TheAuthors}{M. Feischl and T. Tran}

\headers{The Eddy Current--LLG Equations}{\TheAuthors}

\title{{\TheTitle}\thanks{Supported by the Australian Research Council under
grant numbers DP120101886 and DP160101755}}

\author{
  Michael Feischl\thanks{School of Mathematics and Statistics,
         The University of New South Wales,
         Sydney 2052, Australia
    (\email{m.feischl@unsw.edu.au}).}
  \and
  Thanh Tran\thanks{School of Mathematics and Statistics,
         The University of New South Wales,
         Sydney 2052, Australia (\email{thanh.tran@unsw.edu.au}).}
}

\usepackage{amsopn}

\def\H{\widetilde{H}}
\def\P{{\mathbb P}}
\def\R{{\mathbb R}}
\def\N{{\mathbb N}}
\def\Z{{\mathbb Z}}
\def\ba{\boldsymbol{a}}
\def\bb{\boldsymbol{b}}
\def\bc{\boldsymbol{c}}
\def\bv{\boldsymbol{v}}
\def\bu{\boldsymbol{u}}
\def\bE{\boldsymbol{E}}
\def\be{\boldsymbol{e}}
\def\bw{\boldsymbol{w}}
\def\bm{\boldsymbol{m}}
\def\bxi{\boldsymbol{\xi}}
\def\bH{\boldsymbol{H}}
\def\bphi{\boldsymbol{\phi}}
\def\bvarphi{\boldsymbol{\varphi}}
\def\bpsi{\boldsymbol{\psi}}

\def\bn{\boldsymbol{n}}
\def\btau{\boldsymbol{\tau}}

\def\NN{{\mathcal N}}
\def\DD{{\mathcal D}}
\def\KK{{\mathcal K}}

\def\PP{{\mathcal P}}

\def\SS{{\mathcal S}}
\def\TT{{\mathcal T}}

\def\XX{{\mathcal X}}

\def\Ltwo#1{{\mathbb L}^{2}{(#1})}
\def\W{{\mathbb{W}}}
\def\L{{\mathbb{L}}}
\def\H{{\mathbb{H}}}

\def\Hone#1{{\mathbb H}^{1}{(#1})}
\def\Hcurl#1{{\mathbb H}({\rm curl},#1)}

\def\dive{{\rm div}}

\def\norm#1#2{\|#1\|_{#2}}

\def\set#1#2{\big\{#1\,:\,#2\big\}}

\def\slp{\mathfrak{V}} % simple-layer potential
\def\dlp{\mathfrak{K}} % double-layer potential
\def\hyp{\mathfrak{W}} % hypersingular integral operator
\def\dtn{\mathfrak{S}} % Dirichlet-to-Neumann map
\def\mfrac#1#2{\mbox{$\frac{#1}{#2}$}}

\def\normL2#1#2{\|#1\|_{L^2(#2)}}

\newcommand{\dual}[3][]{#1\langle#2\,,\,#3#1\rangle}

\newtheorem{remark}[theorem]{Remark}
\def\revision#1{#1}
\def\comment#1{#1}

\begin{document}%%%%%%%%%%%%%%%%%%%%%%%%%%%%%%%%%%%%%%%%%%%%%%%%%%%%%
\maketitle

\begin{abstract}
We analyze a numerical method for the coupled
system of the eddy current equations in $\R^3$ with the
Landau-Lifshitz-Gilbert equation in a bounded domain.
The unbounded domain is discretized by means of
finite-element/boundary-element coupling. Even though the
considered problem is strongly nonlinear, the numerical approach is
constructed such that only two linear systems per time step have to
be solved. We prove unconditional weak convergence (of a
subsequence) of the finite-element solutions
towards a weak solution. We establish a priori error estimates
if a sufficiently smooth strong solution exists.
Numerical experiments underlining the theoretical results are
presented.
\end{abstract}



\section{Introduction}
This paper deals with the coupling of finite element and boundary
element methods to solve the system of the eddy current
equations in
the whole 3D spatial space and the Landau-Lifshitz-Gilbert equation
(LLG), the so-called ELLG system or equations. The system is also called the
quasi-static Maxwell-LLG (MLLG) system.

The LLG is widely considered as a valid model of micromagnetic
phenomena occurring in, e.g., magnetic sensors, recording heads,
and magneto-resistive storage device~\cite{Gil,LL,prohl2001}. 
Classical results concerning existence and non-uniqueness of solutions can
be found in~\cite{as,vis}. In a ferro-magnetic material, magnetization is
created or affected by external electro-magnetic fields. It is therefore
necessary to augment the Maxwell system with the LLG, which describes the
influence of a ferromagnet; see e.g.~\cite{cim, KruzikProhl06, vis}.
Existence, regularity and local uniqueness for the MLLG equations are studied
in~\cite{CimExist}.

Throughout the literature, there are various works on numerical
approximation methods for the LLG, ELLG, and MLLG 
equations~\cite{alouges,alok,bako,bapr,cim,ellg,thanh} (the list is not
exhausted),
and even with the full Maxwell system on bounded
domains~\cite{BBP,MLLG}, and in the whole~$\R^3$~\cite{CarFab98}.
Originating from the seminal
work~\cite{alouges}, the recent works~\cite{ellg,thanh} consider a
similar numeric integrator for a bounded domain. While the numerical integrator of~\cite{thanh} treated LLG and eddy current simultaneously 
per time step,~\cite{ellg} adapted an idea of~\cite{MLLG} and decoupled the time-steps for LLG and
the eddy current equation. Our approach follows~\cite{ellg}.

This work studies the ELLG equations where we 
consider the electromagnetic field on the whole $\R^3$ and do not need to
introduce artificial boundaries. 
Differently from~\cite{CarFab98} where the Faedo-Galerkin method is
used to prove existence of weak solutions,
we extend the analysis for the integrator used in~\cite{alouges, ellg,
thanh} to a
finite-element/boundary-element (FEM/BEM) discretization of the
eddy current part on $\R^3$. This is inspired by the FEM/BEM coupling
approach designed for the pure eddy current problem
in~\cite{dual}, which
allows us to treat unbounded domains without introducing artificial
boundaries. Two approaches are proposed in~\cite{dual}: the
so-called ``magnetic (or $\bH$-based)
approach'' which eliminates the electric
field, retaining only the magnetic field as the unknown in the eddy-current
system, and the ``electric (or $\bE$-based)
approach'' which considers a
primitive of the electric field as the only unknown. The
coupling of the eddy-current system with the LLG dictates that
the first approach is more appropriate, \comment{because this coupling involves the
magnetic field in the LLG equation rather than the electric field}; see~\cref{eq:strong}.

The main results of this work are weak convergence of the 
discrete approximation towards a weak solution without any
condition on the space and time discretization as well as a~priori error estimates
under the condition that
the exact (strong) solution is sufficiently smooth. In particular, the first result implies the existence of weak solutions, whereas the latter shows that the \revision{smooth} strong solution is unique. 
To the best of our knowledge, no such results \revision{for the tangent plane scheme} have been proved for the LLG
equation. Therefore, we present the proof for this equation in a separate
section, before proving the result for the ELLG system.

\revision{As in~\cite{Abert_etal},} the proof is facilitated by use of an idea of~\cite{bartels} for the harmonic map heat (analyzed for LLG in~\cite{Abert_etal}), which avoids the
normalization of the solution in each time-step, and therefore allows us to use a
linear update formula. This also enables us to consider general quasi-uniform
triangulations for discretization and removes the
requirement for very shape-regular elements (all dihedral angles smaller
than $\pi/2$) present in previous works on this topic.
%This approach which does not use normalization has also been used
%in~\cite{Abert_etal}.

The remainder of this work is organized as follows.
\Cref{section:model} introduces the coupled problem and the
notation, presents the numerical algorithm, and states the main
results of this paper. \Cref{sec:pro} is devoted to the proofs of
these main results. Numerical results are presented in
\Cref{section:numerics}.
The final section, the Appendix, contains the
proofs of some rather elementary or well-known results.

 \section{Model Problem \& Main Results}\label{section:model}
 \subsection{The problem}\label{subsec:pro}
 Consider a bounded Lipschitz domain $D\subset \R^3$ with connected
boundary $\Gamma$ having the outward normal vector $\bn$.
We define $D^\ast:=\R^3\setminus\overline D$,
$D_T:=(0,T)\times D$, $\Gamma_T := (0,T)\times\Gamma$, 
$D_T^{\ast}:=(0,T)\times D^\ast$,
and $\R^3_T := (0,T)\times\R^3$ for $T>0$. \revision{For simplicity, we assume that $D^\ast$ is simply connected.}
 We start with the quasi-static approximation of the full Maxwell-LLG system from~\cite{vis} which reads as
\begin{subequations}\label{eq:strong}
 \begin{alignat}{2}
  \bm_t - \alpha\bm\times\bm_t &= -\bm \times \bH_{\rm eff}
 &&\quad\text{in }D_T,\label{eq:llg}\\
\sigma\bE -\nabla\times\bH&=0&&\quad\text{in }\R^3_T,\label{eq:MLLG1}\\
\mu_0\bH_t +\nabla \times\bE &=-\mu_0\widetilde\bm_t&&\quad\text{in }\R^3_T,\label{eq:MLLG2}\\
{\rm div}(\bH+\widetilde\bm)&=0 &&\quad\text{in }\R^3_T,\label{eq:MLLG3}\\
{\rm div}(\bE)&=0&&\quad\text{in }D^\ast_T,
\end{alignat}
\end{subequations}
where $\widetilde \bm$ is
the zero extension of $\bm$ to $\R^3$ 
and $\bH_{\rm eff}$ is the
effective field defined by $\bH_{\rm eff}= C_e\Delta\bm+\bH$ for
some constant $C_e>0$.
Here the parameter $\alpha>0$ and permeability $\mu_0\geq0$ are
constants, whereas the conductivity $\sigma$ takes a constant
positive value in
$D$ and the zero value in $D^\ast$.
\Cref{eq:MLLG3} is understood in the
distributional sense because there is a jump of~$\widetilde\bm$
across~$\Gamma$. Note that $\bH_{\rm eff}$ contains only the high order term for simplicity. A refined analysis might also allow us to include lower order terms (anisotropy, exterior applied field)
as done in~\cite{bsu}.

It follows from~\cref{eq:llg} that $|\bm|$ is constant. We
follow the usual practice to normalize~$|\bm|$ (and thus
the same condition is required for $|\bm^0|$).
The following conditions are imposed on the solutions
of~\cref{eq:strong}:
\begin{subequations}\label{eq:con}
\begin{alignat}{2}
\partial_n\bm&=0
&& \quad\text{on }\Gamma_T,\label{eq:con1} 
\\
|\bm| &=1
&& \quad\text{in } D_T, \label{eq:con2}
\\
\bm(0,\cdot) &= \bm^0
&& \quad\text{in } D, \label{eq:con3}
\\
\bH(0,\cdot) &= \bH^0
&& \quad\text{in } \R^3,
\\
|\bH(t,x)|&=\mathcal{O}(|x|^{-1})
&& \quad\text{as }|x|\to \infty,
\end{alignat}
\end{subequations}
where $\partial_n$ denotes the normal derivative.
The initial data~$\bm^0$ and~$\bH^0$ satisfy~$|\bm^0|=1$ in~$D$ and
\begin{align}\label{eq:ini}
\begin{split}
 {\rm div}(\bH^0 + \widetilde \bm^0)&=0\quad\text{in }\R^3.
%\\
%(\bH^0 + \bm^0)\cdot \bn&=0\quad\text{on }\Gamma.
\end{split}
\end{align}

The condition~\cref{eq:con2} together with basic
properties of the cross product leads to the following equivalent
formulation of~\cref{eq:llg}:
\begin{align}\label{eq:llg2}
\alpha\bm_t+\bm\times\bm_t= \bH_{\rm eff}-(\bm\cdot \bH_{\rm eff})\bm
\quad\text{in } D_T.
\end{align}

Below, we focus on an $\bH$-based formulation of the problem. 
%However, %as stated in~\cite[Equation~(15)]{dual}, 
It is possible to recover $\bE$ once $\bH$ and $\bm$ are known;
see~\cref{eq:el}

\subsection{Function spaces and notations}\label{subsec:fun spa}
Before introducing the concept of weak solutions
to problem~\cref{eq:strong}--\cref{eq:con} 
we need the following definitions of function
spaces. Let $\Ltwo{D}:=L^2(D;\R^3)$ and $\Hcurl{D}:=\set{\bw\in
\Ltwo{D}}{\nabla\times\bw\in\Ltwo{D}}$.
We define $H^{1/2}(\Gamma)$ as the usual trace space of $H^1(D)$ and define its dual space $H^{-1/2}(\Gamma)$ by extending
the $L^2$-inner product on $\Gamma$.
For convenience we denote
\[
\XX:=\set{(\bxi,\zeta)\in\Hcurl{D}\times
H^{1/2}(\Gamma)}{\bn\times\bxi|_\Gamma =\bn\times \nabla_\Gamma\zeta
\text{ in the sense of traces}}.
\]
Recall that~$\bn\times\bxi|_\Gamma$ is the tangential trace (or
twisted tangential trace) of~$\bxi$, and $\nabla_\Gamma\zeta$ is the
surface gradient of~$\zeta$. Their definitions and properties can
be found in~\cite{buffa, buffa2}.
%(Note that this is well-defined since $\nabla_\Gamma\colon H^{1/2}(\Gamma)\to H^{-1/2}(\Gamma)$ and $\nabla\times(\cdot)\colon\Hcurl{D}\to
%H^{-1/2}(\Gamma)$. Moreover, $\norm{\bn\times \nabla_\Gamma\lambda}{H^{-1/2}(\Gamma)}\lesssim \norm{ \nabla_\Gamma\lambda}{H^{-1/2}(\Gamma)}$ at least on polygonal surfaces.) 

Finally, if $X$ is a normed vector space then, for $m\geq 0$ and $p\in\N\cup\{\infty\}$, $L^2(0,T;X)$,
$H^m(0,T;X)$, and $W^{m,p}(0,T;X)$ 
denote the usual Lebesgue and Sobolev spaces of functions
defined on $(0,T)$ and taking values in $X$.

We finish this subsection with the clarification of the meaning of the
cross product between different mathematical objects.
For any vector functions $\bu, \bv, \bw$ %\in \Hone{D}$
we denote
\begin{gather*}
\bu\times\nabla\bv
%&
:=
\left(
\bu\times\frac{\partial\bv}{\partial x_1},
\bu\times\frac{\partial\bv}{\partial x_2},
\bu\times\frac{\partial\bv}{\partial x_3}
\right), 
%\\
\quad
\nabla\bu\times\nabla\bv
%&
:=
\sum_{i=1}^3
\frac{\partial\bu}{\partial x_i}
\times
\frac{\partial\bv}{\partial x_i} 
\\
\intertext{and}
%\dual{\bu\times\nabla\bv}{\nabla\bw}_{\Ltwo{D}}
(\bu\times\nabla\bv)\cdot \nabla\bw
%&
:=
\sum_{i=1}^3
\left(
\bu\times\frac{\partial\bv}{\partial x_i}
\right)
\cdot \frac{\partial\bw}{\partial x_i}.
%\dual{\bu\times\frac{\partial\bv}{\partial x_i}}%
%{\frac{\partial\bw}{\partial x_i}}_{\Ltwo{D}}.
\end{gather*}

\subsection{Weak solutions}\label{subsec:wea sol}
A weak formulation for~\cref{eq:llg} is well-known, see
e.g.~\cite{alouges, thanh}. 
Indeed, by multiplying~\cref{eq:llg2} by~$\bphi\in
C^\infty(D_T;\R^3)$, using integration by parts, we
deduce
\[
\alpha\dual{\bm_t}{\bm\times\bphi}_{D_T} 
+
\dual{\bm\times\bm_t}{\bm\times\bphi}_{D_T}+C_e 
\dual{\nabla\bm}{\nabla(\bm\times\bphi)}_{D_T}
=
\dual{\bH}{\bm\times\bphi}_{D_T}.
\]

To tackle the eddy current equations on $\R^3$, we aim to employ
FE/BE coupling methods.
To that end, we employ the \emph{magnetic} approach
from~\cite{dual}, which eventually results in a variant of the
\emph{Trifou}-discretization of the eddy-current Maxwell equations.

Multiplying~\cref{eq:MLLG2} by~$\bxi\in C^{\infty}(D,\R^3)$
satisfying~$\nabla\times\bxi=0$ in~$D^\ast$, integrating over~$\R^3$,
and using integration by parts, we obtain for almost
all~$t\in[0,T]$
\[
\mu_0
\dual{\bH_t(t)}{\bxi}_{\R^3}
+
\dual{\bE(t)}{\nabla\times\bxi}_{\R^3}
=
-\mu_0
\dual{\bm_t(t)}{\bxi}_{D}.
\]
Using~$\nabla\times\bxi=0$ in~$D^\ast$ and~\cref{eq:MLLG1} we deduce
\[
\mu_0
\dual{\bH_t(t)}{\bxi}_{\R^3}
+
\sigma^{-1}\dual{\nabla\times\bH(t)}{\nabla\times\bxi}_{D}
=
-\mu_0
\dual{\bm_t(t)}{\bxi}_{D}.
\]
Since~$\nabla\times\bH=\nabla\times\bxi=0$ in~$D^\ast$ \revision{and $D^\ast$ is simply connected by definition (a workaround for non-simply connected $D^\ast$ is presented in~\cite{notsimply})}, 
there exists~$\varphi$ and~$\zeta$ such
that~$\bH=\nabla\varphi$ and~$\bxi=\nabla\zeta$ in~$D^\ast$. 
Therefore, the above equation can be rewritten as
\[
\mu_0
\dual{\bH_t(t)}{\bxi}_{D}
+
\mu_0
\dual{\nabla\varphi_t(t)}{\nabla\zeta}_{D^\ast}
+
\sigma^{-1}\dual{\nabla\times\bH(t)}{\nabla\times\bxi}_{D}
=
-\mu_0
\dual{\bm_t(t)}{\bxi}_{D}.
\]
Since~\cref{eq:MLLG3} implies~$\dive(\bH)=0$ in~$D^\ast$,
we have~$\Delta\varphi=0$ in~$D^\ast$, so that
(formally)~$\Delta\varphi_t=0$ in~$D^\ast$. Hence
integration by parts yields
\begin{equation}\label{eq:Ht}
\mu_0
\dual{\bH_t(t)}{\bxi}_{D}
-
\mu_0
\dual{\partial_n^+\varphi_t(t)}{\zeta}_{\Gamma}
+
\sigma^{-1}\dual{\nabla\times\bH(t)}{\nabla\times\bxi}_{D}
=
-\mu_0
\dual{\bm_t(t)}{\bxi}_{D},
\end{equation}
where~$\partial_n^+$ is the exterior Neumann trace operator with
the limit taken from~$D^\ast$. The advantage of the above
formulation is that no integration over the unbounded domain~$D^\ast$ is
required.
The exterior Neumann trace~$\partial_n^+\varphi_t$ can be computed
from the exterior Dirichlet trace~$\lambda$ of~$\varphi$ by using the
Dirichlet-to-Neumann operator~$\dtn$, which is defined as follows.

Let $\gamma^-$ be the interior Dirichlet trace
operator and $\partial_n^-$ be the interior normal derivative
or Neumann trace operator.
(The $-$ sign indicates \revision{that} the trace is taken from $D$.)
%Analogously, we define $\gamma^+$ and $\partial_n^+$ as the
%exterior counterparts (with the same outward pointing normal
%vector).
Recalling the fundamental solution of the Laplacian
$G(x,y):=1/(4\pi|x-y|)$, we introduce the following integral operators
defined formally on $\Gamma$ as 
\begin{align*}
\slp(\lambda):=\gamma^-\overline\slp(\lambda),
\quad
\dlp(\lambda):=\gamma^-\overline\dlp(\lambda)+\mfrac12,
\quad\text{and}\quad
\hyp(\lambda):=-\partial_n^-\overline\dlp(\lambda),
\end{align*}
where, for $x\notin\Gamma$,
\begin{align*}
\overline\slp(\lambda)(x)
:=
\int_{\Gamma} G(x,y) \lambda(y)\,ds_y
\quad\text{and}\quad 
\overline\dlp(\lambda)(x)
:=\int_{\Gamma} \partial_{n(y)}G(x,y)\lambda(y)\,ds_y,
\end{align*}
\revision{see, e.g.,~\cite{mclean} for further details.}
Moreover, let $\dlp^\prime$ denote the adjoint operator of $\dlp$
with respect to the extended $L^2$-inner product.
Then the exterior Dirichlet-to-Neumann map $\dtn\colon
H^{1/2}(\Gamma)\to H^{-1/2}(\Gamma)$ can be represented as
\begin{equation}\label{eq:dtn}
\dtn 
= 
- \slp^{-1}(1/2-\dlp).
\end{equation}
\revision{Another representation is}
\begin{equation}\label{eq:dtn2}
\dtn 
=
-(1/2-\dlp^\prime) \slp^{-1}(1/2-\dlp)-\hyp.
\end{equation}

Recall that~$\varphi$ satisfies~$\bH=\nabla\varphi$ in~$D^\ast$.
We can choose~$\varphi$ satisfying~$\varphi(x)=O(|x|^{-1})$ as
$|x|\to\infty$.
Now if~$\lambda=\gamma^+\varphi$ 
then~$\lambda_t=\gamma^+\varphi_t$.
Since~$\Delta\varphi=\Delta\varphi_t=0$ in~$D^\ast$, and since the
exterior Laplace problem has a unique solution 
we
have~$\dtn\lambda=\partial_n^+\varphi$
and~$\dtn\lambda_t=\partial_n^+\varphi_t$.
Hence~\cref{eq:Ht} can be rewritten as
\begin{equation}\label{eq:H lam t}
\dual{\bH_t(t)}{\bxi}_{D}
-
\dual{\dtn\lambda_t(t)}{\zeta}_{\Gamma}
+
\mu_0^{-1}
\sigma^{-1}\dual{\nabla\times\bH(t)}{\nabla\times\bxi}_{D}
=
-
\dual{\bm_t(t)}{\bxi}_{D}.
\end{equation}
We remark that if~$\nabla_\Gamma$ denotes the surface gradient
operator on~$\Gamma$ then it is well-known that
$
\nabla_\Gamma\lambda
=
(\nabla\varphi)|_{\Gamma}
-
(\partial_n^+\varphi)\bn
=
\bH|_{\Gamma}
-
(\partial_n^+\varphi)\bn;
$
see e.g.~\cite[Section~3.4]{Monk03}. Hence~$\bn\times\nabla_\Gamma\lambda =
\bn\times\bH|_{\Gamma}$.

The above analysis prompts us to define the following
weak formulation.

\begin{definition}\label{def:fembemllg}
A triple $(\bm,\bH,\lambda)$ satisfying
\begin{align*}
\bm &\in \Hone{D_T}
\quad\text{and}\quad
\bm_t|_{\Gamma_T} \in L^2(0,T;H^{-1/2}(\Gamma)),
\\
\bH &\in L^2(0,T;\Hcurl{D})\cap H^1(0,T;\Ltwo{D}), 
\\
\lambda &\in H^1(0,T;H^{1/2}(\Gamma))
%\\
%\phi &\in L^2(0,T;H^{1/2}(\Gamma))
\end{align*}
is called a weak solution to~\cref{eq:strong}--\cref{eq:con}
if the following statements hold 
 \begin{enumerate}
  \item $|\bm|=1$ almost everywhere in $D_T$; \label{ite:1}

  \item $\bm(0,\cdot)=\bm^0$, $\bH(0,\cdot)=\bH^0$, and 
$\lambda(0,\cdot)=\gamma^+ \varphi^0$ where~$\varphi^0$ is a scalar
function satisfies $\bH^0=\nabla\varphi^0$ in~$D^\ast$ (the
assumption~\cref{eq:ini} ensures the existence
of~$\varphi^0$); \label{ite:2}

  \item For all $\bphi\in C^\infty(D_T;\R^3)$
\label{ite:3}
  \begin{subequations}\label{eq:wssymm}
  \begin{align}
\alpha\dual{\bm_t}{\bm\times\bphi}_{D_T} 
&+
\dual{\bm\times\bm_t}{\bm\times\bphi}_{D_T}+C_e 
\dual{\nabla\bm}{\nabla(\bm\times\bphi)}_{D_T}
\nonumber
\\
&=
\dual{\bH}{\bm\times\bphi}_{D_T};
  \label{eq:wssymm1}
\end{align}

\item There holds $\bn\times \nabla_\Gamma\lambda = \bn\times
\bH|_{\Gamma}$ in the sense of traces; \label{ite:4}

\item For $\bxi\in C^\infty(D;\R^3)$ and $\zeta\in
C^\infty(\Gamma)$ satisfying
$\bn\times\bxi|_{\Gamma}=\bn\times\nabla_\Gamma\zeta$ in the sense of
traces \label{ite:5}
\begin{align}\label{eq:wssymm2}
\dual{\bH_t}{\bxi}_{D_T}
-\dual{\dtn\lambda_t}{\zeta}_{\Gamma_T}
+
\sigma^{-1}\mu_0^{-1}\dual{\nabla\times\bH}{\nabla\times\bxi}_{D_T}
&=-\dual{\bm_t}{\bxi}_{D_T};
\end{align}
\end{subequations}
%where $\SS$ denotes the exterior Dirichlet-to-Neumann operator defined by $\dtn v=\partial_n u$ for
%\begin{align*}
%-\Delta u =0\text{ in }D^\ast,\quad |u(x)|=\mathcal{O}(|x|^{-1})\;\text{ as }x\to\infty,\quad\text{and}\quad u|_{\Gamma}=v
%\end{align*}
\item For almost all $t\in[0,T]$
\label{ite:6}
\begin{gather}
 \norm{\nabla \bm(t)}{\Ltwo{D}}^2 
+
 \norm{\bH(t)}{\Hcurl{D}}^2
+
\norm{\lambda(t)}{H^{1/2}(\Gamma)}^2
\nonumber
\\
+ 
 \norm{\bm_t}{\Ltwo{D_t}}^2
+
 \norm{\bH_t}{\Ltwo{D_t}}^2
+
 \norm{\lambda_t}{H^{1/2}(\Gamma_t)}^2 \leq C,
\label{eq:energybound2}
 \end{gather}
where the constant $C>0$ is independent of $t$.
\end{enumerate}

\comment{
A triple $(\bm,\bH,\lambda)$ is called a strong solution of the ELLG
system~\cref{eq:strong}--\cref{eq:con} if it is a weak solution and
additionally it is sufficiently smooth such that~\cref{eq:llg2} is satisfied in the strong sense.}
\end{definition}

\revision{
\begin{remark}
A refinement of the arguments in Theorem~\ref{thm:weakconv} would allow us to prove that the weak solutions which appear as limits of
the approximations from Algorithm~\ref{algorithm}, are energy dissipative, i.e.,
\begin{align*}
 \frac{C_e}{2}\norm{\nabla\bm(t)}{\Ltwo{D}}^2 &+\norm{\bH(t)}{\Ltwo{D}}^2 -\dual{\dtn\lambda(t)}{\lambda(t)}_\Gamma\\
 &\leq
 \frac{C_e}{2}\norm{\nabla\bm^0}{\Ltwo{D}}^2 +\norm{\bH^0}{\Ltwo{D}}^2 -\dual{\dtn\lambda(0)}{\lambda(0)}_\Gamma
\end{align*}
for all $t\in[0,T]$. The proof works along the lines of~\cite[Theorem~24]{Abert_etal} or~\cite[Appendix~A]{bsu} and is therefore omitted.
\end{remark}
}

%We will derive the weak form of \Cref{def:fembemllg} from
%the strong form~\cref{eq:strong} in Subsection~\ref{subsec:equ
%def} below. 
The reason we integrate over~$[0,T]$ in~\cref{eq:H lam t} to
have~\cref{eq:wssymm2} is to facilitate the passing to the
limit in the proof of the main theorem.
The following lemma justifies the above definition.
\begin{lemma}\label{lem:equidef}
Let $(\bm,\bH,\bE)$ be a strong solution 
of~\cref{eq:strong}--\cref{eq:con}. If $\varphi\in
H(0,T;H^1(D^\ast))$
satisfies $\nabla\varphi=\bH|_{D^\ast_T}$, and if
$\lambda:=\gamma^+\varphi$, then the triple
$(\bm,\bH|_{D_T},\lambda)$
is a weak solution in the sense of \Cref{def:fembemllg}. 
Conversely, given a weak solution $(\bm,\bH|_{D_T},\lambda)$ in the sense of \Cref{def:fembemllg}, let~$\varphi$ be the solution of
\begin{equation}\label{eq:var phi}
\Delta\varphi = 0
\text{ in } D^\ast,
\quad
\varphi = \lambda
\text{ on } \Gamma,
\quad
\varphi(x) = O(|x|^{-1})
\text{ as } |x|\to\infty
\end{equation}
%\begin{alignat}{2}
%\Delta\varphi &= 0
%\quad && \text{in } D^\ast,
%\notag
%\\
%\varphi &= \lambda
%\quad && \text{on } \Gamma,
%\label{eq:var phi}
%\\
%\varphi(x) &= O(|x|^{-1})
%\quad && \text{as } |x|\to\infty.
%\notag
%\end{alignat}
and define $\bH|_{D_T^\star}:=\nabla \phi$
% \begin{align}\label{eq:repform}
% \overline\bH:=\begin{cases}
% \bH &\quad\text{in }D_T,\\
% %\overline{\dlp}(\nabla_\Gamma\lambda)-\overline{\slp}\dtn
% %(\nabla_\Gamma\lambda) 
% \nabla\varphi
% &\quad\text{in } D_T^\ast,
% \end{cases}
% \end{align}
as well as~$\bE$ via
$\bE=\sigma^{-1}(\nabla\times\bH|_{D_T})$ in $D_T$ and outside of $D_T$ as the solution of
\begin{subequations}\label{eq:el}
\begin{alignat}{2}
%\label{eq:el}
%\begin{split}
\nabla\times\bE &= -\mu_0 \bH_t
&&\quad\text{in } D_T^\ast,
\\
{\rm div}(\bE) &=0
&&\quad\text{in }D_T^\ast,
\\
\bn\times \bE|_{D_T^\ast} &= \bn\times \bE|_{D_T}
&&\quad\text{on }\Gamma_T.
%\end{split}
\end{alignat}
\end{subequations}
If $\bm$, $\bH$, and $\bE$ are sufficiently smooth, $(\bm,\bH,\bE)$ is a strong solution of~\cref{eq:strong}--\cref{eq:con}.
 \end{lemma}
\begin{proof}%[Proof of \Cref{lem:equidef}]
We follow~\cite{dual}.
Assume that $(\bm,\bH,\bE)$
satisfies~\cref{eq:strong}--\cref{eq:con}.
Then,~\cref{ite:1}, \cref{ite:2} and~\cref{ite:6} in
\Cref{def:fembemllg} hold, noting~\cref{eq:ini}. \Cref{ite:3}, \cref{ite:4}
and~\cref{ite:5} also hold due to the analysis above
\Cref{def:fembemllg}.
The converse is also true due to the
well-posedness of~\cref{eq:el} as stated
in~\cite[Equation~(15)]{dual}. 
\end{proof}
\begin{remark}
The solution~$\varphi$ to~\cref{eq:var phi} can be represented
as
$
\varphi
=
(1/2+\dlp)\lambda - \slp\dtn\lambda.
$
\end{remark}

The next subsection defines the spaces and functions to be used
in the approximation of the weak solution
the sense of \Cref{def:fembemllg}.
\subsection{Discrete spaces and functions}\label{subsec:dis spa}
For time discretization, we use a uniform partition $0\leq t_i\leq
T$, $i=0,\ldots,N$ with $t_i:=ik$ and $k:=T/N$. The spatial
discretization is determined by a (shape) regular triangulation
$\TT_h$ of $D$ into compact tetrahedra \comment{$\tau\in\TT_h$} with diameter
\comment{$h_{\tau}/C\leq h\leq C|\tau|^{1/3}$}  for some uniform constant $C>0$.
Denoting by~$\NN_h$ the set of nodes of~$\TT_h$, 
we define the following spaces
\begin{align*}
 \SS^1(\TT_h)&:=\set{\phi_h\in C(D)}{\phi_h|\comment{\tau} \in
\PP^1(\comment{\tau})\text{ for
all } \comment{\tau}\in\TT_h},\\
 \KK_{\bphi_h}&:=\set{\bpsi_h\in\SS^1(\TT_h)^3}{\bpsi_h(z)\cdot\bphi_h(z)=0\text{
for all }z\in\NN_h},
\quad\bphi_h\in\SS^1(\TT_h)^3,
\end{align*}
where $\PP^1(\comment{\tau})$ is the space of polynomials of degree at most~1
on~$\comment{\tau}$.

For the discretization of~\cref{eq:wssymm2}, we employ the
space $\NN\DD^1(\TT_h)$ of
first order N\'ed\'elec (edge) elements for $\bH$ and 
and the space $\SS^1(\TT_h|_\Gamma)$ for $\lambda$.
Here $\TT_h|_{\Gamma}$ denotes the restriction of the
triangulation to the boundary~$\Gamma$.
It follows from~\cref{ite:4} in
\Cref{def:fembemllg} that for each~$t\in[0,T]$, the
pair~$(\bH(t),\lambda(t))\in\XX$. We approximate the space~$\XX$ by
\begin{align*}
\XX_h:=\set{(\bxi,\zeta)\in \NN\DD^1(\TT_h)\times
\SS^1(\TT_h|_\Gamma)}{\bn\times \nabla_\Gamma\zeta =
\bn\times\bxi|_\Gamma}.
\end{align*}
To ensure the condition $\bn\times \nabla_\Gamma\zeta = \bn\times
\bxi|_{\Gamma}$, we observe the following.
For any~$\zeta\in \SS^1(\TT_h|_\Gamma)$, if $e$ denotes an edge of
$\TT_h$
on $\Gamma$, then $\int_e \bxi\cdot \btau\,ds = 
\int_e \nabla \zeta \cdot \btau\,ds = \zeta(z_0)-\zeta(z_1)$, where
$\btau$ is the unit direction vector on~$e$, and
$z_0,z_1$ are the endpoints of $e$.
Thus, taking as degrees of freedom all interior edges of $\TT_h$
(i.e. $\int_{e_i}\bxi\cdot\btau \,ds$) as well as all nodes of
$\TT_h|_\Gamma$ (i.e.~$\zeta(z_i)$), we fully determine
a function pair $(\bxi,\zeta)\in\XX_h $.
Due to the considerations above, it is clear that the above
space can be implemented directly without use of Lagrange
multipliers or other extra equations. 

The density properties of the finite element
spaces~$\{\XX_h\}_{h>0}$
are shown in Subsection~\ref{subsec:som lem}; see 
\Cref{lem:den pro}.

Given functions 
$\bw_h^i\colon D\to \R^d$, $d\in\N$,
for all $i=0,\ldots,N$ we define for all $t\in [t_i,t_{i+1}]$
\begin{align*}
\bw_{hk}(t):=\frac{t_{i+1}-t}{k}\bw_h^i + \frac{t-t_i}{k}\bw_h^{i+1},
\quad
\bw_{hk}^-(t):=\bw_h^i,
\quad
\bw_{hk}^+(t):=\bw_h^{i+1} .
\end{align*}
Moreover, we define
\begin{equation}\label{eq:dt}
d_t\bw_h^{i+1}:=\frac{\bw_h^{i+1}-\bw_h^i}{k}\quad\text{for all }
i=0, \ldots, N-1.
\end{equation}

Finally, we denote by $\Pi_{\SS}$ the usual interpolation operator on
$\SS^1(\TT_h)$.%, and $\Pi_{\NN\DD}$ and $\Pi_{\RR\TT}$ the canonical interpolation operators onto $\NN\DD^1(\TT_h)$ (cf.~\cite{ciarlet}) and onto $\RR\TT_0^0(\TT_h|_\Gamma)$ (cf.~\cite[Proposition~3.6]{BrezziFortin}), respectively.
We are now ready to present the algorithm to compute approximate
solutions to problem~\cref{eq:strong}--\cref{eq:con}.

\subsection{Numerical algorithm}\label{section:alg}
In the sequel, when there is no confusion we use the same
notation $\bH$ for 
%we denote by $\bH=\bH|_{D_T}\colon D_T\to \R^3$ 
the restriction of $\bH\colon \R^3_T\to \R^3$ to the domain $D_T$.
\begin{algorithm}\label{algorithm}
\mbox{}

 \textbf{Input:} Initial data $\bm^0_h\in\SS^1(\TT_h)^3$, 
$(\bH^0_h,\lambda_h^0)\in\XX_h$, 
and parameter $\theta\in[0,1]$.

 \textbf{For} $i=0,\ldots,N-1$ \textbf{do:}
 \begin{enumerate}
  \item Compute the unique function $\bv^i_h\in \KK_{\bm_h^i}$
satisfying for all $\bphi_h\in \KK_{\bm_h^i}$
  \begin{align}\label{eq:dllg}
  \begin{split}
   \alpha\dual{\bv_h^i}{\bphi_h}_D 
&+ 
\dual{\bm_h^i\times \bv_h^i}{\bphi_h}_D
+ 
C_e\theta k \dual{\nabla \bv_h^i}{\nabla \bphi_h}_D
\\
&=
-C_e \dual{\nabla \bm_h^i}{\nabla \bphi_h}_D 
+
\dual{\bH_h^i}{\bphi_h}_D.
  \end{split}
  \end{align}
  \item Define $\bm_h^{i+1}\in \SS^1(\TT_h)^3$ nodewise by 
\begin{equation}\label{eq:mhip1}
\bm_h^{i+1}(z) =\bm_h^i(z) + k\bv_h^i(z) \quad\text{for all } z\in\NN_h.
\end{equation}
\item Compute the unique functions
$(\bH_h^{i+1},\lambda_h^{i+1})\in\XX_h$
 satisfying for all
$(\bxi_h,\zeta_h)\in \XX_h$
\begin{align}\label{eq:dsymm}
\dual{d_t\bH_h^{i+1}}{\bxi_h}_{D}
&
-\dual{d_t\dtn_h\lambda_h^{i+1}}{\zeta_h}_\Gamma
+
\sigma^{-1}\mu_0^{-1}\dual{\nabla\times\bH_h^{i+1}}{\nabla\times\bxi_h}_{D}
\nonumber
\\
&=
-\dual{\bv_h^i}{\bxi_h}_{D},
\end{align}
where $\dtn_h\colon H^{1/2}(\Gamma)\to \SS^1(\TT_h|_\Gamma)$ 
is the discrete Dirichlet-to-Neumann operator to be defined
later.
\end{enumerate}

\textbf{Output:} Approximations $(\bm_h^i,\bH_h^i,\lambda_h^i)$
for all $i=0,\ldots,N$.
\end{algorithm}

%The idea to use the linear formula~\cref{eq:mhip1} was introduced
The linear formula~\cref{eq:mhip1} was introduced
in~\cite{bartels} for harmonic map heat flow and adapted for LLG in~\cite{Abert_etal}. As already observed in~\cite{ellg,MLLG} (for bounded domains), we note that the linear systems~\eqref{eq:dllg} and~\eqref{eq:dsymm} 
 are decoupled and
can be solved successively.
\Cref{eq:dsymm} requires the computation
of~$\dtn_h\lambda$ for any~$\lambda\in H^{1/2}(\Gamma)$.
%of~$\dual{\dtn_h\lambda}{\zeta_h}_{\Gamma}$ for~$\lambda\in
%H^{1/2}(\Gamma)$ and~$\zeta_h\in\SS^1(\TT_h|_\Gamma)$.
This is done by use of
the boundary element method. 
%in two different ways.
Let~$\mu\in H^{-1/2}(\Gamma)$ 
and~$\mu_h\in\PP^0(\TT_h|_\Gamma)$
be, respectively, the solution of
%\begin{equation}\label{eq:VK}
%\slp\mu= (\dlp-1/2)\lambda
%\end{equation}
%and
\begin{align}\label{eq:bem}
\slp\mu= (\dlp-1/2)\lambda
\quad\text{and}\quad
\dual{\slp \mu_h}{\nu_h}_\Gamma =
\dual{(\dlp-1/2)\lambda}{\nu_h}_\Gamma\quad\forall\nu_h\in
 \PP^0(\TT_h|_\Gamma),
\end{align}
where $\PP^0(\TT_h|_\Gamma)$ is the space of 
piecewise-constant functions on~$\TT_h|_\Gamma$.

If the representation~\cref{eq:dtn} of~$\dtn$ is used, 
then~$\dtn\lambda=\mu$, and
we can uniquely define~$\dtn_h\lambda$
by solving
\begin{equation}\label{eq:bem1}
\dual{\dtn_h\lambda}{\zeta_h}_\Gamma
=
\dual{\mu_h}{\zeta_h}_\Gamma
\quad\forall\zeta_h\in\SS^1(\TT_h|_\Gamma).
\end{equation}
This is known as the Johnson-N\'ed\'elec coupling. 

If we use the representation~\cref{eq:dtn2} for~$\dtn\lambda$
then $\dtn\lambda = (1/2-\dlp^\prime)\mu-\hyp\lambda$.
In this case we can uniquely define~$\dtn_h\lambda$ by solving
\begin{align}\label{eq:bem2}
\dual{\dtn_h\lambda}{\zeta_h}_\Gamma 
= 
\dual{(1/2-\dlp^\prime)\mu_h}{\zeta_h}_\Gamma 
-
\dual{\hyp \lambda}{\zeta_h}_\Gamma
\quad\forall\zeta_h\in \SS^1(\TT_h|_\Gamma).
\end{align}
This approach yields an (almost) symmetric system and is called
\revision{symmetric coupling}.

In practice,~\cref{eq:dsymm} only requires the computation
of~$\dual{\dtn_h\lambda_h}{\zeta_h}_\Gamma$ for
any~$\lambda_h, \zeta_h\in\SS^1(\TT_h|_\Gamma)$. 
So in the implementation,
neither~\cref{eq:bem1} nor~\cref{eq:bem2} has to be solved. 
It suffices to solve the second equation in~\cref{eq:bem}
and compute the right-hand side of either~\cref{eq:bem1}
or~\cref{eq:bem2}.

It is proved in~\cite[Appendix~A]{afembem} that symmetric coupling
results in a discrete operator which is uniformly elliptic and
continuous:
\begin{align}\label{eq:dtnelliptic}
\begin{split}
-\dual{\dtn_h \zeta_h}{\zeta_h}_\Gamma&\geq
C_\dtn^{-1}\norm{\zeta_h}{H^{1/2}(\Gamma)}^2\quad\text{for all
}\zeta_h\in \SS^1(\TT_h|_\Gamma),\\
\norm{\dtn_h\zeta}{H^{-1/2}(\Gamma)}^2&\leq C_\dtn
\norm{\zeta}{H^{1/2}(\Gamma)}^2\quad\text{for all }\zeta\in H^{1/2}(\Gamma),
\end{split}
\end{align}
for some constant $C_\dtn>0$ which depends only on $\Gamma$. 
Even though \revision{we are convinced that the proposed algorithm  works for both
approaches, we are not aware of the essential ellipticity result of the 
form~\cref{eq:dtnelliptic} for the Johnson-N\'ed\'elec
approach. Thus, \comment{hereafter}, $\dtn_h$ is understood to be defined
by the symmetric coupling~\cref{eq:bem2}.}

%\begin{remark} %\mbox{}
%\begin{enumerate}
%\item
%In \Cref{section:weak} (\Cref{lem:bil for})
%we will show that Algorithm~\ref{algorithm} is well-defined.
%\item
%The idea to use the linear formula~\cref{eq:mhip1} was introduced
%in~\cite{bartels} and used in~\cite{Abert_etal}.
%\end{enumerate}
%\end{remark}
%%%%%%%%%%%%%%%%%%%%%%%%%%%
\subsection{Main results}
Before stating the main results, 
we first state some general
assumptions. Firstly, the weak convergence of approximate solutions
requires the following
conditions on $h$ and $k$, depending on the value of the parameter~$\theta$
in~\cref{eq:dllg}:
\begin{equation}\label{eq:hk 12}
\begin{cases}
k = o(h^2) \quad & \text{when } 0 \le \theta < 1/2, \\
k = o(h)   \quad & \text{when } \theta = 1/2, \\
\text{no condition} & \text{when } 1/2 < \theta \le 1.
\end{cases}
\end{equation}
Some supporting lemmas which have their own interests do not require any
condition when $\theta=1/2$. For those results, a slightly different
condition is required, namely
 \begin{equation}\label{eq:hk con}
\begin{cases}
k = o(h^2) \quad & \text{when } 0 \le \theta < 1/2, \\
\text{no condition} & \text{when } 1/2 \le \theta \le 1.
\end{cases}
\end{equation}
The initial data are assumed to satisfy
\begin{equation}\label{eq:mh0 Hh0}
\begin{split}
\sup_{h>0}
\left(
\norm{\bm_h^0}{H^1(D)}
+
\norm{\bH_h^0}{\Hcurl{D}}
+
\norm{\lambda_h^0}{H^{1/2}(\Gamma)}
\right)
&<\infty,
\\
\lim_{h\to0}
\norm{\bm_h^0-\bm^0}{\Ltwo{D}}
&= 0.
\end{split}
\end{equation}
% Finally we assume that the meshes~\{$\TT_h\}_{h>0}$ are defined
% such that
% \begin{equation}\label{eq:Th}
% \PP^0(\TT_{h'}|_\Gamma)
% \subset
% \PP^0(\TT_{h}|_\Gamma)
% \quad\text{and}\quad
% \SS^1(\TT_{h'}|_\Gamma)
% \subset
% \SS^1(\TT_{h}|_\Gamma)
% \quad\text{for } h<h'.
% \end{equation}
% This assumption is only needed in the proof of
% \Cref{lem:wea con} below.

The following three theorems state the main results of this paper.
The first theorem proves existence of weak solutions. The second theorem establishes a 
priori error estimates for the pure LLG case of~\cref{eq:llg},
i.e., $\bH_{\rm eff}=C_e\Delta\bm$ and there is no coupling with
the eddy current equations~\cref{eq:MLLG1}--\cref{eq:MLLG3}. 
The third theorem
provides a priori error estimates for the ELLG system.
\begin{theorem}[Existence of solutions]\label{thm:weakconv}
Under the assumptions~\cref{eq:hk 12} and~\cref{eq:mh0
Hh0}, the problem~\cref{eq:strong}--\cref{eq:con} has a 
solution~$(\bm,\bH,\lambda)$ in the sense of
\Cref{def:fembemllg}.
\end{theorem}

\begin{theorem}[Error estimates for LLG]
\label{thm:strongconvLLG}
Let 
\[
\bm \in 
W^{2,\infty}\big(0,T;\Hone{D}\big)
\cap 
W^{1,\infty}\big(0,T; \W^{1,\infty}(D)\cap \H^2(D)\big)
\]
denote a strong solution of~\cref{eq:llg}
and~\cref{eq:con1}--\cref{eq:con3}
with $\bH_{\rm eff}=C_e\Delta\bm$.
Then for $\theta>1/2$ (where~$\theta$ is the parameter in~\cref{eq:dllg})
and for all $h,k$ satisfying $0<h,k\leq 1$ and 
$k\le\alpha/(2C_e)$,
the following statements hold
\begin{align}\label{eq:strongconv}
\begin{split}
 \max_{0\leq i\leq N}&\norm{\bm(t_i)-\bm_h^{i}}{\Hone{D}}\leq C_{\rm
conv}\big( \norm{\bm^0-\bm^0_h}{\Hone{D}}+h+k\big)
\end{split}
\end{align}
and
\begin{align}\label{eq:strongconv2}
\norm{\bm-\bm_{hk}}{L^2(0,T;\H^1(D))}
\leq 
C_{\rm conv}\big(
\norm{\bm^0-\bm^0_h}{\Hone{D}}+h+k\big).
\end{align}
The constant $C_{\rm conv}>0$
depends only on the regularity of $\bm$, the shape regularity
of $\TT_h$, and the values of $\alpha$ and $\theta$.
Moreover, the strong solution $\bm$ is unique and coincides with
the weak solution from~\cref{thm:weakconv}.
\end{theorem}

%\comment{
%In the following, we call $(\bm,\bH,\lambda)$ {a strong solution of the ELLG system~\cref{eq:strong} and~\cref{eq:con}, if
%$(\bm,\bH,\lambda)$ is a solution in the sense of \Cref{def:fembemllg} and additionally sufficiently smooth such that~\cref{eq:llg2} is satisfied in the strong sense.}
%}

\begin{theorem}[Error estimates for for ELLG]\label{thm:strongconvELLG}
Let $(\bm,\bH,\lambda)$ be a strong solution of ELLG \comment{(in the sense
of~\Cref{def:fembemllg})}
with the following properties
\begin{align*}
\bm &\in W^{2,\infty}\big(0,T;\Hone{D}\big)\cap W^{1,\infty}\big(0,T;
\W^{1,\infty}(D)\cap \H^2(D)\big), \\
\bH&\in W^{2,\infty}(0,T;\H^2(D))\cap {\mathbb L}^\infty(D_T), \\
\lambda &\in W^{1,\infty}(0,T; H^1(\Gamma))\text{ such
that }\dtn\lambda_t\in L^\infty(0,T; H^{1/2}_{\comment{\rm pw}}(\Gamma)),
\end{align*}
where $H^{1/2}_{\comment{\rm pw}}(\Gamma)$ is defined piecewise on each
smooth part of $\Gamma$.
Then for $\theta>1/2$ and for all $h$, $k$ satisfying
$0<h,k\leq 1/2$ and $k\le\alpha/(2C_e)$,
there hold
\begin{align}\label{eq:convest1}
\begin{split}
\max_{0\leq i \leq N}&
\Big(\norm{\bm(t_i)-\bm_h^{i}}{\Hone{D}}^2+\norm{\bH(t_i)-\bH^i_h}{\Ltwo{D}}^2\\
&\qquad+\norm{\lambda(t_i)-\lambda_h^i}{H^{1/2}(\Gamma)}^2
+k\norm{\nabla\times(\bH(t_i)-\bH^i_h)}{\Ltwo{D}}^2\Big)
\\
&\leq C_{\rm conv}\Big(\norm{\bm^0-\bm_h^0}{\Hone{D}}^2+\norm{\lambda^0-\lambda_h^0}{H^{1/2}(\Gamma)}^2
\\
&\qquad\qquad+k\norm{\nabla\times(\bH^0-\bH^0_h)}{\Ltwo{D}}^2
+h^2+k^2\Big)
\end{split}
\end{align}
and
\begin{align}\label{eq:convest2}
\begin{split}
\norm{\bm-\bm_{hk}}{L^2(0,T;\H^1(D))}^2
&+
\norm{(\bH-\bH_{hk},\lambda-\lambda_{hk})}{L^2(0,T;\XX)}^2
\\
&\leq C_{\rm conv}
\Big(\norm{\bm^0-\bm_h^0}{\Hone{D}}^2+\norm{\bH^0-\bH^0_h}{\Ltwo{D}}^2
\\
&\qquad\qquad+k\norm{(\bH^0-\bH^0_h,\lambda^0-\lambda^0_h)}{\XX}^2
+h^2+k^2\Big).
\end{split}
\end{align}
Particularly, the strong solution $(\bm,\bH,\lambda)$ is unique and coincides with the weak solution from \Cref{thm:weakconv}.
The constant $C_{\rm conv}>0$ depends only on the smoothness of $\bm$, $\bH$, $\lambda$, and on the shape regularity of $\TT_h$.
\end{theorem}
\begin{remark}
 It is possible to replace the assumption $\bm(t)\in \W^{1,\infty}(D)$ in
Theorems~\ref{thm:strongconvLLG}--\ref{thm:strongconvELLG} by
$\nabla\bm(t)\in \L^4(D)$ and $\bm(t)\in \L^\infty(D)$.
 This, however, results in a reduced rate of convergence $\sqrt{k}$ instead
of $k$; see Remark~\ref{rem:reg} for further discussion.
\end{remark}

\begin{remark}
Little is known about the regularity of the solutions of~\cref{eq:strong}
in 3D; see~\cite{CimExist}. However, for the 2D case of $D$ being the flat torus
$\R^2/\Z^2$, Theorem~5.2 in~\cite{prohl2001} states the
existence of arbitrarily smooth local solutions of the full MLLG system,
given that $\norm{\nabla\bm^0}{L^2(D)}$ and the initial
values of the Maxwell system are sufficiently small.
Since the eddy-current equations are a particular simplification of
Maxwell's equations, this strongly endorses the assumption
that also for the
3D case of the ELLG equations there exist arbitrarily smooth local
solutions.
\end{remark}

\section{Proof of the main results}\label{sec:pro}
%\subsection{Derivation of the weak from in \Cref{def:fembemllg}}\label{subsec:equ def}
\subsection{Some lemmas}\label{subsec:som lem}
In this subsection we prove all important lemmas which are directly related
to the proofs of the main results.
The first lemma proves density properties of the discrete spaces.
\begin{lemma}\label{lem:den pro}
Provided that the meshes $\{\TT_h\}_{h>0}$ are regular, 
the union~$\bigcup_{h>0}\XX_h$ is dense in~$\XX$.
%there holds
%\[
%\overline{\bigcup_{h>0}\XX_h}
%=
%\XX
%\]
%where the closure is taken in the $\XX$-topology. 
There exists an interpolation operator
$\Pi_\XX:=(\Pi_{\XX,D},\Pi_{\XX,\Gamma})\colon 
\big(\H^2(D)\times H^2(\Gamma)\big)
\cap \XX\to \XX_h$ which satisfies
\begin{align}\label{eq:int}
\norm{(1-\Pi_\XX)(\bxi,\zeta)}{\Hcurl{D}\times H^{1/2}(\Gamma) }
&\leq C_{\XX} h( \norm{\bxi}{\H^2(D)}+ h^{1/2}\norm{\zeta}{H^2(\Gamma)}),
\end{align}
where $C_\XX>0$ depends only on $D$, $\Gamma$, and the shape
regularity of $\TT_h$.
\end{lemma}
\begin{proof}
%The density of ${\bigcup_{h>0}\SS^1(\TT_h|_\Gamma)}$ in
%$H^{1/2}(\Gamma)$ is well-known, and the density of the set
%${\bigcup_{h>0}\NN\DD^1(\TT_H)}$ in
%$\Hcurl{D}$  is a result of~\cite[Theorem~8.1]{hipt}
%and the density of~$\Hone{D}$ in~$\Ltwo{D}$. 
The interpolation operator
$\Pi_\XX:=(\Pi_{\XX,D},\Pi_{\XX,\Gamma})\colon 
\big(\H^2(D)\times H^2(\Gamma)\big)
\cap \XX\to \XX_h$ is constructed as follows.
The interior degrees of freedom (edges) of $\Pi_\XX(\bxi,\zeta)$
are equal to the interior degrees of freedom of
$\Pi_{\NN\DD}\bxi\in\NN\DD^1(\TT_h)$, where $\Pi_{\NN\DD}$ is
the usual interpolation operator onto $\NN\DD^1(\TT_h)$.
The degrees of freedom of $\Pi_\XX(\bxi,\zeta)$ which lie on
$\Gamma$ (nodes) are equal to $\Pi_\SS\zeta$.
By the definition of $\XX_h$, this fully determines $\Pi_\XX$.
Particularly, since $\bn\times
\bxi|_\Gamma=\bn\times\nabla_\Gamma\zeta$, there holds
$\Pi_{\NN\DD}\bxi|_\Gamma = \Pi_{\XX,\Gamma}(\bxi,\zeta)$.
Hence, the interpolation error can be bounded by
\begin{align*}
\norm{(1-\Pi_\XX)(\bxi,\zeta)}{\Hcurl{D}\times H^{1/2}(\Gamma) }&\leq \norm{(1-\Pi_{\NN\DD})\bxi}{\Hcurl{D}}+\norm{(1-\Pi_\SS)\zeta}{H^{1/2}(\Gamma)}\\
&\lesssim h( \norm{\bxi}{\H^2(D)}+ h^{1/2}\norm{\zeta}{H^2(\Gamma)}).
\end{align*}
Since $\big(\H^2(D)\times H^2(\Gamma)\big)\cap \XX$ is dense in
$\XX$, this concludes the proof.
\end{proof}

The following lemma gives an equivalent form
to~\cref{eq:wssymm2} and shows that 
Algorithm~\ref{algorithm} is well-defined.
\begin{lemma}\label{lem:bil for}
Let $a(\cdot,\cdot)\colon \XX\times\XX\to \R$, $a_h(\cdot,\cdot)\colon \XX_h\times\XX_h\to \R$, and
$b(\cdot,\cdot)\colon \Hcurl{D}\times\Hcurl{D}\to \R$ be
bilinear forms defined by
\begin{align*}
a(A,B)&:=\dual{\bpsi}{\bxi}_D -\dual{\dtn \eta}{\zeta}_\Gamma,
\\
a_h(A_h,B_h)
&:=
\dual{\bpsi_h}{\bxi_h}_D
-
\dual{\dtn_h\eta_h}{\zeta_h}_{\Gamma},
\\
b(\bpsi,\bxi)&:=
\sigma^{-1}\mu_0^{-1}
\dual{\nabla\times \bpsi}{\nabla\times \bxi}_{\revision{D}},
\end{align*}
%\begin{align*}
%\begin{split}
%a(A,B)&:=\dual{\bH}{\bxi}_D -\dual{\dtn \eta}{\zeta}_\Gamma,\\
%b(\bH,\bxi)&:=\dual{\nabla\times \bH}{\nabla\times \bxi}_\Gamma,
%\end{split}
%\end{align*}
for all $\bpsi, \bxi\in\Hcurl{D}$, $A:=(\bpsi,\eta)$, $B:=(\bxi,\zeta)\in\XX$, $A_h = (\bpsi_h,\eta_h), B_h = (\bxi_h,\zeta_h) \in \XX_h$.
Then 
\begin{enumerate}
\item
The bilinear forms satisfy,
for all 
$A=(\bpsi,\eta)\in\XX$ and
$A_h=(\bpsi_h,\eta_h)\in\XX_h$,
\begin{align}\label{eq:elliptic}
\begin{split}
a(A,A)&\geq C_{\rm ell} 
\big(
\norm{\bpsi}{\Ltwo{D}}^2+\norm{\eta}{H^{1/2}(\Gamma)}^2
\big),
\\
a_h(A_h,A_h)
&\geq 
C_{\rm ell}
\big(
\norm{\bpsi_h}{\Ltwo{D}}^2+\norm{\eta_h}{H^{1/2}(\Gamma)}^2
\big),
\\
b(\bpsi,\bpsi)
&\geq C_{\rm ell}\norm{\nabla\times
\bpsi}{\Ltwo{D}}^2.
\end{split}
\end{align}
\item
\Cref{eq:wssymm2} is equivalent to
\begin{equation}\label{eq:bil for}
\int_0^T
a(A_t(t),B) \, dt
%a(\partial_t A(t),B) \, dt
+
\int_0^T
b(\bH(t),\bxi) \, dt
=
-
\dual{\bm_t}{\bxi}_{D_T}
%\int_0^T
%a(\partial_t A,B)+
%b(A,B) \,dt
%=
%-
%\dual{\bm_t}{\bxi}_{D_T}
\end{equation}
for all~$B=(\bxi,\zeta)\in \XX$, where~$A=(\bH,\lambda)$.
%for all~$B\in L^2(0,T;\XX)$.
\item
\Cref{eq:dsymm} is of the form
\begin{align}\label{eq:eddygeneral}
a_h(d_t A_h^{i+1},B_h) + b(\bH_h^{i+1},\bxi_h)
=
-\dual{\bv_h^i}{\bxi_h}_\Gamma
\end{align}
where $A_h^{i+1}:=(\bH_h^{i+1},\lambda_h^{i+1})$ and
$B_h:=(\bxi_h,\zeta_h)$.
\item
Algorithm~\ref{algorithm} is well-defined in the sense
that~\cref{eq:dllg} and~\cref{eq:dsymm} have unique solutions.

\item 
%The ellipcitiy~\cref{eq:dtnelliptic} allows us to define from
%the bilinear form~$a_h(\cdot,\cdot)$ the norm
The norm
\begin{align}\label{eq:hnorm}
 \norm{B_h}{h}^2 := \norm{\bxi_h}{\Ltwo{D}}^2 
- 
\dual{\dtn_h \zeta_h}{\zeta_h}_\Gamma
\quad\forall B_h=(\bxi_h,\zeta_h)\in\XX_h,
\end{align}
is equivalent to the graph norm of $\Ltwo{D}\times
H^{1/2}(\Gamma)$ uniformly in $h$.
\end{enumerate}
\end{lemma}
\begin{proof}
The unique solvability of~\cref{eq:dsymm} follows immediately from the continuity and ellipticity of the bilinear forms $a_h(\cdot,\cdot)$ 
and $b(\cdot,\cdot)$.

The unique solvability of~\cref{eq:dllg} follows from the positive
definiteness of the left-hand side, the linearity of the right-hand side,
and the finite space dimension.
The ellipticity~\cref{eq:dtnelliptic} shows the norm equivalence in the final statement.
\end{proof}

The following lemma establishes an energy bound for the discrete
solutions.
\begin{lemma}\label{lem:denergygen}
Under the assumptions~\cref{eq:hk con} and~\cref{eq:mh0
Hh0}, there holds
for all $k<2\alpha$ and $j=1,\ldots,N$
\begin{align}\label{eq:denergy}
\begin{split}
\sum_{i=0}^{j-1}&
\left(
\norm{\bH_h^{i+1}-\bH_h^i}{\Ltwo{D}}^2+\norm{\lambda_h^{i+1}-\lambda_h^i}{H^{1/2}(\Gamma)}^2
\right)
\\
&+
k
\sum_{i=0}^{j-1}
\norm{\nabla\times \bH_h^{i+1}}{\Ltwo{D}}^2
+ \norm{\bH_h^j}{\Hcurl{D}}^2
+ \norm{\lambda_h^{j}}{H^{1/2}(\Gamma)}^2
\\
&+
\norm{\nabla \bm_h^{j}}{\Ltwo{D}}^2
+
\max\{2\theta-1,0\}k^2
\sum_{i=0}^{j-1}
\norm{\nabla\bv_h^i}{\Ltwo{D}}^2
+
k
\sum_{i=0}^{j-1}
\norm{\bv_h^i}{\Ltwo{D}}^2
\\
&+k \sum_{i=0}^{j-1} (\norm{d_t\bH_h^{i+1}}{\Ltwo{D}}^2 +\norm{d_t\lambda_h^{i+1}}{H^{1/2}(\Gamma)}^2)\\
&+\sum_{i=0}^{j-1} 
\norm{\nabla\times (\bH^{i+1}_h-\bH^{i}_h)}{\Ltwo{D}}^2
\leq C_{\rm ener}.
\end{split}
\end{align}
\end{lemma}
\revision{\begin{proof}
Choosing $B_h=A_h^{i+1}$ in~\cref{eq:eddygeneral} and
multiplying the resulting equation by $k$ we obtain
\begin{equation}\label{eq:e1}
a_h(A_h^{i+1}-A_h^i,A_h^{i+1})
+
kb(\bH_h^{i+1},\bH_h^{i+1})
= -k\dual{\bv_h^i}{\bH_h^i}_D
 -k\dual{\bv_h^i}{\bH_h^{i+1}-\bH_h^i}_D.
 \end{equation}
 Following the lines of~\cite[Lemma~5.2]{ellg} using the definition of $a_h(\cdot,\cdot)$ and~\cref{eq:dtnelliptic}, we end up with
\begin{align*}
a_h(A_h^{i+1}-A_h^i,A_h^{i+1})
&+ 
%\frac{2k}{C_e r}
k
b(\bH_h^{i+1},\bH_h^{i+1})
%\nonumber
%\\
%&+
+
\frac{C_e}{2}
\left(
\norm{\nabla \bm_h^{i+1}}{\Ltwo{D}}^2
-
\norm{\nabla\bm_h^i}{\Ltwo{D}}^2
\right)
\\
&+
(\theta-1/2)k^2 C_e
\norm{\nabla\bv_h^i}{\Ltwo{D}}^2
% +\frac{2\alpha k}{C_e}
+
(\alpha-\epsilon /2) k
\norm{\bv_h^i}{\Ltwo{D}}^2
\nonumber
\\
&\leq
\frac{k}{2\epsilon}
a_h(A_h^{i+1}-A_h^i,A_h^{i+1}-A_h^i).
%\norm{\bH_h^{i+1}-\bH_h^i}{\Ltwo{D}}^2.
\end{align*}
Summing over $i$ from $0$ to $j-1$ and 
(for the first term on the left-hand side)
applying Abel's summation by parts formula we derive as in~\cite[Lemma~5.2]{ellg}
% \begin{equation}\label{equ:Abe}
% \sum_{i=0}^{j-1}
% (u_{i+1}-u_{i})u_{i+1}
% =
% \frac{1}{2} |u_j|^2
% -
% \frac{1}{2} |u_0|^2
% +
% \frac{1}{2}
% \sum_{i=0}^{j-1}
% |u_{i+1}-u_{i}|^2,
% \end{equation}
% we deduce, after multiplying the equation by two and rearranging,
% \begin{align*}%\label{eq:e3}
% %\frac{1}{C_e r}
% &(1-k/\epsilon)
% \sum_{i=0}^{j-1}
% a_h(A_h^{i+1}-A_h^i,A_h^{i+1}-A_h^i)
% +
% 2k
% \sum_{i=0}^{j-1}
% b(\bH_h^{i+1},\bH_h^{i+1})
% %\frac{1}{C_e r}
% +
% a_h(A_h^j,A_h^j)
% \\
% &
% +
%  C_e
% \norm{\nabla \bm_h^{j}}{\Ltwo{D}}^2
% +
% (2\theta-1)k^2  C_e
% \sum_{i=0}^{j-1}
% \norm{\nabla\bv_h^i}{\Ltwo{D}}^2
% %\\
% %&
% +
% %\frac{2\alpha k}{C_e}
% (2\alpha-\epsilon ) k
% \sum_{i=0}^{j-1}
% \norm{\bv_h^i}{\Ltwo{D}}^2
% \\
% &
% \leq 
% %\frac{\eps k}{C_e}\sum_{i=0}^{j-1}\norm{\bv_h^i}{\Ltwo{D}}^2 
% %\frac{kr}{\eps}
% %\sum_{i=0}^{j-1}
% %\norm{\bH_h^i-\bH_h^{i+1}}{\Ltwo{D}}^2
% %+
% %\frac{1}{C_e r}
%  C_e
% \norm{\nabla \bm_h^{0}}{\Ltwo{D}}^2
% +
% a_h(A_h^0,A_h^0).
% \end{align*}
% Since $k<2\alpha$ we can choose $\eps>0$ such
% that $2\alpha-\epsilon >0$ and $1-k/\epsilon>0$. By noting the
% ellipticity~\cref{eq:dtnelliptic}, the bilinear forms
% $a_h(\cdot,\cdot)$ and $b(\cdot,\cdot)$ are elliptic in their
% respective (semi-)norms. We obtain
\begin{align}\label{eq:e4}
\begin{split}
\sum_{i=0}^{j-1}&
\left(
\norm{\bH_h^{i+1}-\bH_h^i}{\Ltwo{D}}^2+\norm{\lambda_h^{i+1}-\lambda_h^i}{H^{1/2}(\Gamma)}^2
\right)
\\
&
+
k
\sum_{i=0}^{j-1}
\norm{\nabla\times \bH_h^{i+1}}{\Ltwo{D}}^2
+ \norm{\bH_h^j}{\Ltwo{D}}^2
\\
&
+ \norm{\lambda_h^{j}}{H^{1/2}(\Gamma)}^2
+
\norm{\nabla \bm_h^{j}}{\Ltwo{D}}^2
\\
&
+
(2\theta-1)k^2
\sum_{i=0}^{j-1}
\norm{\nabla\bv_h^i}{\Ltwo{D}}^2
+
k
\sum_{i=0}^{j-1}
\norm{\bv_h^i}{\Ltwo{D}}^2
\\
& \qquad\qquad
\leq 
C
\left(
\norm{\nabla \bm_h^{0}}{\Ltwo{D}}^2
+
\norm{\bH_h^0}{\Ltwo{D}}^2
+ \norm{\lambda_h^{0}}{H^{1/2}(\Gamma)}^2
\right)
\leq
C,
\end{split}
\end{align}
where in the last step we used~\cref{eq:mh0 Hh0}.
It remains to consider the last three terms on the
left-hand side of~\cref{eq:denergy}.
Again, we consider~\cref{eq:eddygeneral} and select
$B_h=d_tA_h^{i+1}$ %as a test function 
to obtain after multiplication by~$2k$
\begin{align*}
 \begin{split}
 2k a_h(d_tA_h^{i+1},d_tA_h^{i+1})
&
+ 
2b(\bH_h^{i+1},\bH_h^{i+1}-\bH_h^i)
\\
&= -2k\dual{\bv_h^i}{d_t\bH_h^{i+1}}_D
\le
 k \norm{\bv_h^i}{\Ltwo{D}}^2
+
k \norm{d_t\bH_h^{i+1}}{\Ltwo{D}}^2,
\end{split}
\end{align*}
so that, noting~\cref{eq:e4} and~\cref{eq:elliptic},
\begin{align}\label{eq:e5}
\begin{split}
k \sum_{i=0}^{j-1} 
\left(
\norm{d_t\bH_h^{i+1}}{\Ltwo{D}}^2 
\right.
&+
\left.
\norm{d_t\lambda_h^{i+1}}{H^{1/2}(\Gamma)}^2
\right)
+  
2 \sum_{i=0}^{j-1} 
b(\bH_h^{i+1},\bH_h^{i+1}-\bH_h^i) 
\\
&
\lesssim
 k 
\sum_{i=0}^{j-1} 
\norm{\bv_h^i}{\Ltwo{D}}^2
\leq
C.
\end{split}
\end{align}
Using Abel's summation by parts formula as in~\cite[Lemma~5.2]{ellg} for the 
second sum on the left-hand side, and noting
%can be rewritten as
%\begin{align*}
%2 \sum_{i=0}^{j-1} b(\bH_h^{i+1},\bH_h^{i+1}-\bH_h^i)
%=
%b(\bH_h^{j},\bH_h^{j})
%-
%b(\bH_h^{0},\bH_h^0)
%+
%\sum_{i=0}^{j-1} 
%b(\bH_h^{i+1}-\bH_h^i,\bH_h^{i+1}-\bH_h^i).
%\end{align*}
%Inserting this into~\cref{eq:e5}, using 
the ellipticity of 
the bilinear form $b(\cdot,\cdot)$ and~\cref{eq:mh0 Hh0}, we obtain
together with~\cref{eq:e4}
\begin{align}\label{eq:d4}
\sum_{i=0}^{j-1}&
(\norm{\bH_h^{i+1}-\bH_h^i}{\Ltwo{D}}^2+\norm{\lambda_h^{i+1}-\lambda_h^i}{H^{1/2}(\Gamma)}^2)\notag
\\
&+
k
\sum_{i=0}^{j-1}
\norm{\nabla\times \bH_h^{i+1}}{\Ltwo{D}}^2
+ \norm{\bH_h^j}{\Hcurl{D}}^2
+ \norm{\lambda_h^{j}}{H^{1/2}(\Gamma)}^2
+
\norm{\nabla \bm_h^{j}}{\Ltwo{D}}^2\notag
\\
&+
(2\theta-1)k^2
\sum_{i=0}^{j-1}
\norm{\nabla\bv_h^i}{\Ltwo{D}}^2
+
k
\sum_{i=0}^{j-1}
\norm{\bv_h^i}{\Ltwo{D}}^2
\\
&+k \sum_{i=0}^{j-1} (\norm{d_t\bH_h^{i+1}}{\Ltwo{D}}^2 +\norm{d_t\lambda_h^{i+1}}{H^{1/2}(\Gamma)}^2)+\sum_{i=0}^{j-1} 
\norm{\nabla\times (\bH^{i+1}_h-\bH^{i}_h)}{\Ltwo{D}}^2
\leq 
C.\notag
\end{align}
Clearly, if $1/2\le\theta\le1$ then~\cref{eq:d4}
yields~\cref{eq:denergy}. If $0\le\theta<1/2$, we argue as in~\cite[Remark~6]{ellg} to conclude the proof.
% then since the mesh
% is regular, the inverse
% estimate~$\norm{\nabla\bv_h^i}{\Ltwo{D}}\lesssim
% h^{-1}\norm{\bv_h^i}{\Ltwo{D}}$ gives
% \begin{align*}
% (2\theta-1)k^2
% \sum_{i=0}^{j-1}
% \norm{\nabla\bv_h^i}{\Ltwo{D}}^2
% +
% k
% \sum_{i=0}^{j-1}
% \norm{\bv_h^i}{\Ltwo{D}}^2
% &\gtrsim
% \left(
% 1-k^2h^{-1}(1-2\theta)
% \right)
% k
% \sum_{i=0}^{j-1}
% \norm{\bv_h^i}{\Ltwo{D}}^2
% \\
% &\gtrsim
% k
% \sum_{i=0}^{j-1}
% \norm{\bv_h^i}{\Ltwo{D}}^2
% \end{align*}
% as $k^2h^{-1}\to0$ under the assumption~\cref{eq:hk con}.
% This estimate and~\cref{eq:d4} give~\cref{eq:denergy},
% completing the proof of the lemma.
\end{proof}}

Collecting the above results we obtain the following equations
satisfied by the discrete functions
%~$\bv_{hk}^-$ and~$\bH_{hk}$.
defined from~$\bm_h^i$, $\bH_h^i$, $\lambda_h^i$, and
$\bv_h^i$.

\begin{lemma}\label{lem:mhk Hhk}
Let $\bm_{hk}^{-}$, $A_{hk}^\pm:=(\bH_{hk}^\pm,\lambda_{hk}^\pm)$, and $\bv_{hk}^{-}$ be
defined from $\bm_h^i$, $\bH_h^i$, $\lambda_h^i$, and $\bv_h^i$ as
described in Subsection~\ref{subsec:dis spa}.
Then 
\begin{subequations}\label{eq:mhk Hhk}
\begin{align}
   \alpha\dual{\bv_{hk}^-}{\bphi_{hk}}_{D_T}
&+ 
\dual{(\bm_{hk}^-\times \bv_{hk}^-)}{\bphi_{hk}}_{D_T}
+ 
C_e \theta k \dual{\nabla \bv_{hk}^-}{\nabla \bphi_{hk}}_{D_T}
\nonumber
\\
&=
-C_e \dual{\nabla \bm_{hk}^-}{\nabla \bphi_{hk}}_{D_T}
+
\dual{\bH_{hk}^-}{\bphi_{hk}}_{D_T}
\label{eq:mhk Hhk1}
\\
\intertext{and with~$\partial_t$ denoting time derivative}
\int_0^T 
a_h(\partial_t A_{hk}(t),B_h)
\,dt
&+
\int_0^T 
b(\bH_{hk}^+(t),\bxi_h)
\,dt
=
-\dual{\bv_{hk}^-}{\bxi_h}_{D_T}
\label{eq:mhk Hhk2}
\end{align}
\end{subequations}
for all $\bphi_{hk}$ and $B_h:=(\bxi_h,\zeta_h)$ satisfying
$\bphi_{hk}(t,\cdot)\in\KK_{\bm_h^i}$ for $t\in[t_i,t_{i+1})$ and
$B_h\in\XX_h$.
%$(\bxi_h,\zeta_h)\in\XX_h$.
%$(\bxi_h(t,\cdot),\zeta_h(t,\cdot))\in\XX_h$ for all $t\in[0,T]$.
\end{lemma}
\begin{proof}
The lemma is a direct consequence of~\cref{eq:dllg} 
and~\cref{eq:eddygeneral}.
\end{proof}

\revision{In the following lemma, we state some auxiliary results, already proved in~\cite{Abert_etal}.
\begin{lemma}\label{lem:abert}
 There holds
 \begin{align}\label{eq:const}
\left|
\norm{\bm_h^j}{\Ltwo{D}}^2-\norm{\bm_h^0}{\Ltwo{D}}^2
\right|
&
\leq C kC_{\rm ener},
\end{align}
as well as
\begin{equation}\label{eq:mhk pm}
\norm{\bm_{hk}^\pm}{L^2(0,T;\Hone{D})}
\le T
\norm{\bm_{hk}^\pm}{L^\infty(0,T;\Hone{D})}
\leq CC_{\rm ener},
\end{equation}
where $C>0$ depends only on the shape regularity of $\TT_h$ and $T$. 
\end{lemma}
\begin{proof}
 The estimate~\eqref{eq:const} follows analogously to~\cite[Proposition~9]{Abert_etal}.
The estimate~\eqref{eq:mhk pm} then follows from~\cref{eq:mh0 Hh0}, \cref{eq:denergy}, and~\cref{eq:const}.
\end{proof}
}
The next lemma shows that the functions defined in the above
lemma form sequences which have convergent subsequences.
\begin{lemma}\label{lem:weakconv}
Assume that the assumptions~\cref{eq:hk con} and~\cref{eq:mh0
Hh0} hold.
As $h$, $k\to0$, the following limits exist up to extraction of
subsequences \revision{(all limits hold for the same subsequence)}
\begin{subequations}\label{eq:weakconv}
\begin{alignat}{2}
 \bm_{hk}&\rightharpoonup\bm\quad&&\text{in }\Hone{D_T},\label{eq:wc1}\\
 \bm_{hk}^\pm&\rightharpoonup\bm\quad&&\text{in }
L^2(0,T;\Hone{D}),\label{eq:wc1a}
\\
 \bm_{hk}^\pm
& \rightarrow\bm
\quad&&
\text{in } \Ltwo{D_T}, 
\label{eq:wc2}\\
  (\bH_{hk},\lambda_{hk})&\rightharpoonup(\bH,\lambda)\quad&&\text{in }
L^2(0,T;\XX),\label{eq:wc3}\\
    (\bH_{hk}^\pm,\lambda_{hk}^\pm)&\rightharpoonup(\bH,\lambda)
\quad&&\text{in } L^2(0,T;\XX),\label{eq:wc4}\\
 (\bH_{hk},\lambda_{hk})&\rightharpoonup (\bH,\lambda) \quad&&\text{in } 
H^1(0,T;\Ltwo{D}\times H^{1/2}(\Gamma)),\label{eq:wc31}\\ 
    \bv_{hk}^- &\rightharpoonup \bm_t\quad&&\text{in }\Ltwo{D_T},\label{eq:wc5}
\end{alignat}
\end{subequations}
for certain functions $\bm$, $\bH$, and $\lambda$ satisfying
$\bm\in \Hone{D_T}$, $\bH\in 
H^1(0,T;\Ltwo{D})$, and $(\bH,\lambda)\in 
L^2(0,T;\XX)$. Here~$\rightharpoonup$ denotes the weak convergence
and~$\to$ denotes the strong convergence in the relevant space.

Moreover, if the assumption~\cref{eq:mh0 Hh0} holds
then there holds additionally $|\bm|=1$ almost everywhere
in~$D_T$.
\end{lemma}

\begin{proof}
The proof works analogously to~\cite{Abert_etal} and is therefore omitted.
\end{proof}

We also need the following strong convergence property.
\begin{lemma}\label{lem:str con}
Under the assumptions~\cref{eq:hk 12} and~\cref{eq:mh0 Hh0} there holds
\begin{equation}\label{eq:mhk h12}
\norm{\bm_{hk}^--\bm}{L^2(0,T;\H^{1/2}(D))}
\to 0
\quad\text{as } h,k\to 0.
\end{equation}
\end{lemma}
\begin{proof}
It follows from the triangle inequality and the definitions of~$\bm_{hk}$
and~$\bm_{hk}^-$ that
\begin{align*}
 \norm{\bm_{hk}^-&-\bm}{L^2(0,T;\H^{1/2}(D))}^2\\
&\lesssim 
\norm{\bm_{hk}^--\bm_{hk}}{L^2(0,T;\H^{1/2}(D))}^2
+
\norm{\bm_{hk}-\bm}{L^2(0,T;\H^{1/2}(D))}^2 \\
&\leq
\sum_{i=0}^{N-1}
k^3\norm{\bv_h^i}{\H^{1/2}(D)}^2
+
\norm{\bm_{hk}-\bm}{L^2(0,T;\H^{1/2}(D))}^2 \\
&\leq 
\sum_{i=0}^{N-1}k^3\norm{\bv_h^i}{\Hone{D}}^2
+
\norm{\bm_{hk}-\bm}{L^2(0,T;\H^{1/2}(D))}^2.
\end{align*}
The second term on the right-hand side converges to zero due
to~\cref{eq:wc1} and the compact embedding of
\[
\Hone{D_T}
\simeq
\{\bv \, | \,
\bv\in L^2(0,T;\Hone{D}), \, \bv_t\in L^2(0,T;\Ltwo{D}) \}
\]
into $L^2(0,T;\H^{1/2}(D))$; see \cite[Theorem 5.1]{Lio69}.
For the first term on the right-hand side,
when $\theta>1/2$,~\cref{eq:denergy} implies
$
\sum_{i=0}^{N-1}k^3\norm{\bv_h^i}{\Hone{D}}^2
\lesssim k \to 0.
$
When $0\le\theta\leq 1/2$, a standard inverse inequality,~\cref{eq:denergy}
and~\cref{eq:hk 12} yield
\[
\sum_{i=0}^{N-1}k^3\norm{\bv_h^i}{\Hone{D}}^2
\lesssim 
\sum_{i=0}^{N-1}h^{-2}k^3\norm{\bv_h^i}{\Ltwo{D}}^2
\lesssim 
h^{-2}k^2 \to 0,
\]
completing the proof of the lemma.
\end{proof}

The following lemma involving the $\L^2$-norm of the cross product
of two vector-valued functions will be used when passing
to the limit of equation~\cref{eq:mhk Hhk1}.
\begin{lemma}\label{lem:h}
There exists a constant $C_{\rm sob}>0$ which depends only on $D$ such that
\begin{subequations}
\label{eq:m1}
 \begin{align}
  \norm{\bw_{1/2}\times\bw_1}{\Ltwo{D}}&\leq C_{\rm sob}\norm{\bw_{1/2}}{\H^{1/2}(D)}\norm{\bw_1}{\Hone{D}}\label{eq:l2}
\\
  \norm{\bw_0\times\bw_{1/2}}{\widetilde\H^{-1}(D)}&\leq C_{\rm sob}\norm{\bw_0}{\Ltwo{D}}\norm{\bw_{1/2}}{\H^{1/2}(D)},
 \end{align}
 \end{subequations}
 for all $\bw_0\in \Ltwo{D}$, $\bw_{1/2}\in\H^{1/2}(D)$, and all $\bw_{1}\in\Hone{D}$.
 \end{lemma}
\begin{proof}
It is shown in~\cite[Theorem~5.4, Part~I]{adams} that
the embedding $\iota\colon\Hone{D}\to\L^6(D)$ is continuous.
Obviously, the identity $\iota\colon \Ltwo{D}\to\Ltwo{D}$ is
continuous. By real interpolation, we find that $\iota\colon
[\Ltwo{D},\Hone{D}]_{1/2}\to [\Ltwo{D},\L^6(D)]_{1/2}$ is
continuous. Well-known results in interpolation theory show
$
[\Ltwo{D},\Hone{D}]_{1/2}= \H^{1/2}(D)
$
and
$
[\Ltwo{D},\L^6(D)]_{1/2}=\L^3(D)
$
with equivalent norms; see e.g.~\cite[Theorem~5.2.1]{BL}.
By using H\"older's inequality, we deduce
\begin{align*}
 \norm{\bw_{0}\times\bw_{1}}{\Ltwo{D}}\leq
\norm{\bw_{0}}{\L^3(D)}\norm{\bw_{1}}{\L^6(D)}
\lesssim
\norm{\bw_{0}}{\H^{1/2}(D)}\norm{\bw_1}{\Hone{D}}
\end{align*}
proving~\cref{eq:l2}.

For the second statement, there holds with the well-known identity
\begin{equation}\label{eq:abc}
\ba\cdot(\bb\times\bc)= \bb\cdot(\bc\times\ba)=
\bc\cdot(\ba\times\bb)
\quad\forall\ba,\bb,\bc\in\R^3
\end{equation}
that
 \begin{align*}
 \norm{\bw_0\times\bw_{1/2}}{\widetilde \H^{-1}(D)}&=\sup_{\bw_1\in \Hone{D}\setminus\{0\}}\frac{\dual{\bw_0\times\bw_{1/2}}{\bw_1}_D}{\norm{\bw_1}{\Hone{D}}}\\
 &\leq
 \sup_{\bw_1\in \Hone{D}\setminus\{0\}}\frac{\norm{\bw_0}{\Ltwo{D}}\norm{\bw_{1/2}\times\bw_1}{\Ltwo{D}}}{\norm{\bw_1}{\Hone{D}}}.
 \end{align*}
 The estimate~\cref{eq:l2} concludes the proof.
\end{proof}

Finally, to pass to the limit in equation~\cref{eq:mhk Hhk2} we
need the following result.
\begin{lemma}\label{lem:wea con}
For any sequence~$\{\lambda_h\}\subset H^{1/2}(\Gamma)$ 
and any function~$\lambda\in H^{1/2}(\Gamma)$, if
\begin{equation}\label{eq:zet con}
\lim_{h\to0}
\dual{\lambda_h}{\nu}_{\Gamma}
=
\dual{\lambda}{\nu}_{\Gamma}
\quad\forall\nu\in H^{-1/2}(\Gamma)
\end{equation}
then
\begin{equation}\label{eq:dtn zet con}
\lim_{h\to0}
\dual{\dtn_h\lambda_h}{\zeta}_{\Gamma}
=
\dual{\dtn\lambda}{\zeta}_{\Gamma}
\quad\forall\zeta\in H^{1/2}(\Gamma).
\end{equation}
\end{lemma}
\begin{proof}
Let~$\mu$ and~$\mu_h$ be defined
by~\cref{eq:bem} with~$\lambda$ in the second equation replaced 
by~$\lambda_h$. Then (recalling that Costabel's symmetric
coupling is used) $\dtn\lambda$ and~$\dtn_h\lambda_h$ are
defined via~$\mu$ and~$\mu_h$ by~\cref{eq:dtn2} 
and~\cref{eq:bem2}, respectively, namely,
$\dtn\lambda = (1/2-\dlp^\prime)\mu - \hyp\lambda$ and
$
\dual{\dtn_h\lambda_h}{\zeta_h}_\Gamma 
= 
\dual{(1/2-\dlp^\prime)\mu_h}{\zeta_h}_\Gamma 
-
\dual{\hyp \lambda_h}{\zeta_h}_\Gamma
$
for all~$\zeta_h\in \SS^1(\TT_h|_\Gamma)$.
For any~$\zeta\in H^{1/2}(\Gamma)$, 
let~$\{\zeta_h\}$ be a sequence
in~$\SS^1(\TT_h|_\Gamma)$ 
satisfying~$\lim_{h\to0}\norm{\zeta_h-\zeta}{H^{1/2}(\Gamma)}=0$.
By using the triangle inequality and the above representations
of~$\dtn\lambda$ and~$\dtn_h\lambda_h$ we deduce
\begin{align}\label{eq:dtn dtn}
\big|
\dual{\dtn_h\lambda_h}{\zeta}
-
\dual{\dtn\lambda}{\zeta}_\Gamma
\big| 
&\leq 
\big|\dual{\dtn_h\lambda_h-\dtn\lambda}{\zeta_h}_\Gamma\big|
+
\big|\dual{\dtn_h\lambda_h-\dtn\lambda}{\zeta-\zeta_h}_\Gamma\big|
\notag
\\ 
&\le
\big|
\dual{(\tfrac12-\dlp^\prime)(\mu_h-\mu)}{\zeta_h}_\Gamma 
\big|
+
\big|
\dual{\hyp (\lambda_h-\lambda)}{\zeta_h}_\Gamma
\big|
\notag
\\
&\quad
+
\big|\dual{\dtn_h\lambda_h-\dtn\lambda}{\zeta-\zeta_h}_\Gamma\big|
\notag
\\
&\le
\big|
\dual{(\tfrac12-\dlp^\prime)(\mu_h-\mu)}{\zeta_h}_\Gamma 
\big|
+
\big|
\dual{\hyp (\lambda_h-\lambda)}{\zeta}_\Gamma
\big|
\notag
\\
&\quad
+
\big|
\dual{\hyp (\lambda_h-\lambda)}{\zeta_h-\zeta}_\Gamma
\big|
+
\big|\dual{\dtn_h\lambda_h-\dtn\lambda}{\zeta-\zeta_h}_\Gamma\big|.
\end{align}
The second term on the right-hand side of~\cref{eq:dtn dtn}
goes to zero as~$h\to0$ due to~\cref{eq:zet con} and the
self-adjointness of~$\hyp$. The third term converges to zero
due to the strong convergence~$\zeta_h\to\zeta$
in~$H^{1/2}(\Gamma)$ and the boundedness of~$\{\lambda_h\}$
in~$H^{1/2}(\Gamma)$, which is a consequence of~\cref{eq:zet
con} and the Banach-Steinhaus Theorem. The last term tends to
zero due to the convergence of~$\{\zeta_h\}$ and the boundedness
of~$\{\dtn_h\lambda_h\}$; see~\cref{eq:dtnelliptic}.
Hence~\cref{eq:dtn zet con} is proved if we prove
\begin{equation}\label{eq:muh mu}
\lim_{h\to0}
\dual{(1/2-\dlp^\prime)(\mu_h-\mu)}{\zeta_h}_{\Gamma}
= 0.
\end{equation}
\revision{This, however, follows from standard convergence arguments \comment{in
boundary element methods} and concludes the proof.}
\end{proof}

\subsection{Further lemmas for the proofs of \Cref{thm:strongconvLLG,thm:strongconvELLG}}
In this section, we prove all necessary results for the proof of the a~priori error estimates.
\begin{lemma} \revision{Recall the operators $\dtn$ and $\dtn_h$ defined in~\cref{eq:dtn2} and~\cref{eq:bem2}.
	Given $\lambda$ such that $\dtn \lambda\in H^{1/2}_-(\Gamma)$, there holds}
	\begin{align}\label{eq:dtnh}
	\sup_{\xi_h\in\SS^1(\TT_h|_\Gamma)\setminus\{0\}}\frac{\dual{\dtn
\lambda -
\dtn_h\lambda}{\xi_h}_\Gamma}{\norm{\xi_h}{H^{1/2}(\Gamma)}}\leq
C_\dtn h\norm{\dtn\lambda}{H^{1/2}_{\comment{\rm pw}}(\Gamma)}.
	\end{align}
\end{lemma}
\begin{proof}
	By definition of $\mu_h$ in~\cref{eq:bem}  as the Galerkin approximation of $\mu=\dtn\lambda$, there
	holds by standard arguments
	\begin{align*}
	\norm{\mu-\mu_h}{H^{-1/2}(\Gamma)}\lesssim
h\norm{\mu}{H_{\comment{\rm pw}}^{1/2}(\Gamma)}.
	\end{align*}
	Hence, there holds with the mapping properties of $\dlp^\prime$
	\begin{align*}
	\dual{\dtn \lambda - \dtn_h\lambda}{\xi_h}_\Gamma=\dual{(1/2-\dlp^\prime)(\mu-\mu_h)}{\xi_h}_\Gamma\lesssim
	h\norm{\mu}{H_{\comment{\rm pw}}^{1/2}(\Gamma)}\norm{\xi_h}{H^{1/2}(\Gamma)}.
	\end{align*}
	This concludes the proof.
\end{proof}
The next lemma proves that the time derivative of the exact solution,
namely $\bm_t(t_i)$, 
can be approximated in the discrete tangent space $\KK_{\bm_h^i}$
in such a way
that the error can be controlled by the error in the approximation
of~$\bm(t_i)$ by $\bm_h^i$.
\begin{lemma}\label{lem:orthoproj} Assume the following regularity of the
strong solution $\bm$ of~\cref{eq:strong}:
\begin{align*}
C_{\rm reg}
&:=
\norm{\bm}{W^{1,\infty}(0,T;\H^2(D))}
+
\norm{\bm}{W^{1,\infty}(0,T;\W^{1,\infty}(D))}<\infty.
\end{align*}
For any $i=1,\ldots, N$ let $\P_h^i\colon \H^1(D)\to \KK_{\bm_h^i}$ denote
the orthogonal projection onto $\KK_{\bm_h^i}$. Then
\begin{align*}
\norm{\bm_t(t_i)-\P_h^i\bm_t(t_i)}{\Hone{D}}
\leq
C_{\P} \big( h + \norm{\bm(t_i)-\bm_h^i}{\Hone{D}} \big),
\end{align*}
where $C_{\P}>0$ depends only on $C_{\rm reg}$ and the shape regularity of
$\TT_h$.
\end{lemma}
\begin{proof}
We fix $i\in\{1,\ldots,N\}$. Due to the well-known result
\[
\norm{\bm_t(t_i)-\P_h^i\bm_t(t_i)}{\Hone{D}}
=
\inf_{\bw\in\KK_{\bm_h^i}}
\norm{\bm_t(t_i)-\bw}{\Hone{D}}
\]
and the estimate (recalling the definition of the interpolation~$\Pi_{\SS}$ in
Subsection~\ref{subsec:dis spa})
\begin{align*}
\norm{\bm_t(t_i)-\bw}{\Hone{D}}
&\le
\norm{\bm_t(t_i)-\Pi_\SS\bm_t(t_i)}{\Hone{D}}
+
\norm{\Pi_\SS\bm_t(t_i)-\bw}{\Hone{D}}
\\
&\le
C_{\rm reg} \ h
+
\norm{\Pi_\SS\bm_t(t_i)-\bw}{\Hone{D}}
\quad\forall\bw\in\KK_{\bm_h^i},
\end{align*}
it suffices to prove that
\begin{equation}\label{eq:Ihmt w}
\inf_{\bw\in\KK_{\bm_h^i}}
\norm{\Pi_\SS\bm_t(t_i)-\bw}{\Hone{D}}
\lesssim
h + \norm{\bm(t_i)-\bm_h^i}{\Hone{D}},
\end{equation}
where the constant depends only on $C_{\rm reg}$ and the shape regularity
of~$\TT_h$.

To this end we first note that the assumption on the regularity of the
exact solution~$\bm$ implies that after a modification on a set of measure
zero, we can assume that $\bm$ is continuous in~$D_T$. Hence
\[
|\bm(t_i,x)| = 1 \quad\text{and}\quad \bm_t(t_i,x)\cdot\bm(t_i,x) = 0
\quad\forall x\in D.
\]
Thus by using the elementary identity
\[
\ba\times(\bb\times \bc)=(\ba\cdot \bc)\bb-(\ba\cdot \bb)\bc
\quad\forall\ba,\bb,\bc\in\R^3
\]
it can be easily shown that
\begin{equation}\label{eq:mtz}
\Pi_\SS\bm_t(t_i,z)
=
\bm_t(t_i,z)
=
R_z^i\bm(t_i,z)
\end{equation}
where, for any $z\in\NN_h$,
the mapping $R_z^i : \R^3\to\R^3$ is defined by
\[
R_z^i\ba
=
-
\ba
\times
\big(
\bm(t_i,z)
\times
\bm_t(t_i,z)
\big)
\quad\forall\ba\in\R^3.
\]
We note that this mapping has the following properties
\begin{equation}\label{eq:Rzi}
R_z^i\ba \cdot \ba = 0
\quad\text{and}\quad
|R_z^i\ba| \le |\bm_t(t_i,z)|\ |\ba|
\quad\forall\ba\in\R^3.
\end{equation}

Next, prompted by~\cref{eq:mtz} and~\cref{eq:Rzi}
in order to prove~\cref{eq:Ihmt w} 
we define $\bw\in\SS^1(\TT_h)$ by
\begin{equation}\label{eq:w def}
\bw(z)
=
R_z^i
\bm_h^i(z)
\quad\forall z\in\NN_h.
\end{equation}
Then $\bw\in\KK_{\bm_h^i}$ and we can
estimate~$\norm{\Pi_\SS\bm_t(t_i)-\bw}{\Hone{D}}^2$ by
\begin{align*}
\norm{\Pi_\SS\bm_t(t_i)-\bw}{\Hone{D}}^2
&=
\sum_{T\in\TT_h}
\norm{\Pi_\SS\bm_t(t_i)-\bw}{\Hone{T}}^2
\\
&\lesssim
\sum_{T\in\TT_h}
\norm{\Pi_\SS\bm_t(t_i)-\bw-\bw_T}{\Hone{T}}^2
+
\sum_{T\in\TT_h}
\norm{\bw_T}{\Hone{T}}^2
\end{align*}
where $\bw_T$ is polynomial of degree~1 on~$T$ defined by
\begin{equation}\label{eq:bwT}
\bw_T(z)
=
R_z^i
\big(
\bm(t_i,z_T)-\bm_h^i(z_T)
\big)
\quad\forall z\in\NN_h\cap T.
\end{equation}
Here, $z_T$ is the node in~$T$ satisfying
\begin{equation}\label{eq:z0}
|\bm(t_i,z_T)-\bm_h^i(z_T)|
\le
|\bm(t_i,z)-\bm_h^i(z)|
\quad\forall z\in\NN_h\cap T.
\end{equation}
Thus,~\cref{eq:Ihmt w} is proved if we prove
\begin{equation}\label{eq:A1}
\sum_{T\in\TT_h}
\norm{\Pi_\SS\bm_t(t_i)-\bw-\bw_T}{\Hone{T}}^2
\lesssim
h + \norm{\bm(t_i)-\bm_h^i}{\Hone{D}}
\end{equation}
and
\begin{equation}\label{eq:A2}
\sum_{T\in\TT_h}
\norm{\bw_T}{\Hone{T}}^2
\lesssim
\norm{\bm(t_i)-\bm_h^i}{\Hone{D}}.
\end{equation}

To prove~\cref{eq:A1} we denote
$A := \sum_{T\in\TT_h} \norm{\Pi_\SS\bm_t(t_i)-\bw-\bw_T}{\Hone{T}}^2$
and use a standard inverse estimate 
and the equivalence~\cite[Lemma~3.2]{thanh}
to have
\begin{align*}
A
&\lesssim
h^{-2}
\sum_{T\in\TT_h}
\norm{\Pi_\SS\bm_t(t_i)-\bw-\bw_T}{\Ltwo{T}}^2
\\
&\simeq
h
\sum_{T\in\TT_h}
\sum_{z\in\NN_h\cap T}
\big|
\Pi_\SS\bm_t(t_i,z) - \bw(z) - \bw_T(z)
\big|^2.
\end{align*}
This together with~\cref{eq:mtz}--\cref{eq:bwT} and the regularity
assumption of~$\bm$ yields
\begin{align*}
A
&\lesssim
h
\sum_{T\in\TT_h}
\sum_{z\in\NN_h\cap T}
\left|
R_z^i
\Big(
\bm(t_i,z) - \bm_h^i(z) - \bm(t_i,z_T) + \bm_h^i(z_T)
\Big)
\right|^2
\\
&\lesssim
h
\sum_{T\in\TT_h}
\sum_{z\in\NN_h\cap T}
\big|
\bm(t_i,z) - \bm_h^i(z) - \bm(t_i,z_T) + \bm_h^i(z_T)
\big|^2
\\
&=
h
\sum_{T\in\TT_h}
\sum_{z\in\NN_h\cap T}
\left|
\big( \Pi_\SS\bm(t_i,z) - \bm_h^i(z) \big)
-
\big( \Pi_\SS\bm(t_i,z_T) - \bm_h^i(z_T) \big)
\right|^2.
\end{align*}
Since $\Pi_\SS\bm-\bm_h^i$ is polynomial of degree 1 on $T$, 
\Cref{lem:h1seminorm} in the Appendix gives
\begin{align*}
A
&\lesssim
\sum_{T\in\TT_h}
\norm{\nabla\big(\Pi_\SS\bm(t_i) - \bm_h^i\big)}{\Ltwo{T}}^2
=
\norm{\nabla\big(\Pi_\SS\bm(t_i) - \bm_h^i\big)}{\Ltwo{D}}^2
\\
&\le
\norm{\Pi_\SS\bm(t_i) - \bm_h^i}{\Hone{D}}^2.
\end{align*}
By using the triangle inequality and the approximation property of the
interpolation operator~$\Pi_\SS$, noting that $\bm\in L^\infty(0,T;{\mathbb
H}^2(D))$, we obtain~\cref{eq:A1}.

It remains to prove~\cref{eq:A2}. Denoting 
$\bvarphi(z) = \bm(t_i,z)-\bm_h^i(z)$ for $z\in\NN_h\cap T$,
it follows successively from~\cite[Lemma~3.2]{thanh}, 
\Cref{lem:h1seminorm},~\cref{eq:bwT} and~\cref{eq:Rzi} that
\begin{align}\label{eq:wT1}
\norm{\bw_T}{\Hone{T}}^2
&\simeq
h^3
\sum_{z\in\NN_h\cap T}
|\bw_T(z)|^2
+
h
\sum_{z\in\NN_h\cap T}
|\bw_T(z)-\bw_T(z_T)|^2
\nonumber
\\
&=
h^3
\sum_{z\in\NN_h\cap T}
|R_z^i\bvarphi(z_T)|^2
+
h
\sum_{z\in\NN_h\cap T}
\left|
\big(R_z^i-R_{z_T}^i\big)
\bvarphi(z_T)
\right|^2
\nonumber
\\
&\lesssim
h^3
\sum_{z\in\NN_h\cap T}
|\bvarphi(z_T)|^2
+
h
\sum_{z\in\NN_h\cap T}
\left|
\big(R_z^i-R_{z_T}^i\big)
\bvarphi(z_T)
\right|^2.
\end{align}
For the term in the last sum on the right-hand side, we use the 
triangle inequality and the regularity of~$\bm$ to obtain
\begin{align*}
\left|
\big(R_z^i-R_{z_T}^i\big)
\bvarphi(z_T)
\right|^2
&\le
|\bvarphi(z_T)|^2
\left|
\bm(t_i,z)\times\bm_t(t_i,z)
-
\bm(t_i,z_T)\times\bm_t(t_i,z_T)
\right|^2
\\
&\lesssim
|\bvarphi(z_T)|^2
|\bm(t_i,z)|^2
|\bm_t(t_i,z)-\bm_t(t_i,z_T)|^2
\\
&\qquad+
|\bvarphi(z_T)|^2
|\bm_t(t_i,z_T)|^2
\big(
|\bm(t_i,z)-\bm(t_i,z_T)|^2
\\
&\lesssim
|\bvarphi(z_T)|^2
\big(
|\bm_t(t_i,z)-\bm_t(t_i,z_T)|^2
+
|\bm(t_i,z)-\bm(t_i,z_T)|^2
\big)
\\
&\lesssim
h^2 |\bvarphi(z_T)|^2,
\end{align*}
where in the last step we used also Taylor's Theorem and
$|z-z_T|\le h$, noting that~$\nabla\bm\in\L^{\infty}(D_T)$
and~$\nabla\bm_t\in\L^{\infty}(D_T)$.
Therefore,~\cref{eq:wT1}, \cref{eq:z0} and~\cite[Lemma~3.2]{thanh} imply
\[
\norm{\bw_T}{\Hone{T}}^2
\lesssim
h^3
\sum_{z\in\NN_h\cap T}
|\bvarphi(z_T)|^2
\le
h^3
\sum_{z\in\NN_h\cap T}
|\bvarphi(z)|^2
\simeq
\norm{\bm(t_i)-\bm_h^i}{\Ltwo{T}}^2.
\]
Summing over~$T\in\TT_h$ we obtain~\cref{eq:A2}, completing the proof of
the lemma.
\end{proof}

The following lemma shows that for the pure LLG equation,
$\bm_t(t_i)$ solves the same equation as
$\bv_h^i$, up to an error term.
\begin{lemma}\label{lem:help}
Let $\bm$ denote a strong solution
of~\cref{eq:llg},~\cref{eq:con1}--\cref{eq:con3} which satisfies
$\nabla\bm \in {\mathbb L}^\infty(D_T)$ and $\bm_t\in
L^{\infty}(0,T;\H^2(D))$. Then, for $i=0,\ldots,N$ there holds
\begin{align}\label{eq:help}
\alpha\dual{\bm_t(t_i)}{\bphi_h}_D
&+
\dual{(\bm(t_i)\times\bm_t(t_i))}{\bphi_h}_D
+ 
C_e \theta k
\dual{\nabla\bm_t(t_i)}{\nabla \bphi_h}_D
\nonumber \\
&=
-C_e \dual{\nabla \bm(t_i)}{\nabla \bphi_h}_D
+
R(\bphi_h)
\quad\forall\bphi_h\in \KK_{\bm_h^i},
  \end{align}
where
\begin{equation}\label{eq:Rphih}
|R(\bphi_h)|
\leq 
C_{\rm R}
\big(h+\theta k+\norm{\bm(t_i)-\bm_h^i}{\Ltwo{D}}\big)
\norm{\bphi_h}{\Ltwo{D}}.
\end{equation}
Here, the constant $C_{\rm R}>0$ depends only on the regularity assumptions
on $\bm$ and the shape regularity of~$\TT_h$.
\end{lemma}
\begin{proof}
Note that $|\bm|=1$ implies $(\bm\cdot\Delta \bm)=-|\nabla\bm|^2$.
Recalling that~\cref{eq:llg} is equivalent to~\cref{eq:llg2}, this
identity and~\cref{eq:llg2} give, for all $\bphi_h\in \KK_{\bm_h^i}$,
\begin{align}\label{eq:mt phih}
   \alpha\dual{\bm_t(t_i)}{\bphi_h}_{D} &+ \dual{(\bm(t_i)\times
\bm_t(t_i))}{\bphi_h}_{D}\nonumber\\
   &=C_e \dual{\Delta
\bm(t_i)}{\bphi_h}_{D}+C_e\dual{|\nabla\bm(t_i)|^2\bm(t_i)}{\bphi_h}_D\nonumber\\
   &=-C_e \dual{\nabla \bm(t_i)}{\nabla
\bphi_h}_{D}+C_e\dual{|\nabla\bm(t_i)|^2\bm(t_i)}{\bphi_h}_D
  \end{align}
where in the last step we used~\cref{eq:con1} and integration by
parts. Hence~\cref{eq:help} holds with
\begin{equation}\label{eq:Rphih2}
R(\bphi_h)
:=
C_e \theta k
\dual{\nabla\bm_t(t_i)}{\nabla \bphi_h}_D
+
C_e\dual{|\nabla\bm(t_i)|^2\bm(t_i)}{\bphi_h}_D.
\end{equation}
It remains to show~\cref{eq:Rphih}.
The condition~\cref{eq:con1} implies $\partial_n\bm_t=0$ on
$\Gamma_T$, and thus the first term on the right-hand side
of~\cref{eq:Rphih2} can be estimated as
\begin{align}\label{eq:B1}
C_e\theta k|\dual{ \nabla \bm_t(t_i)}{\nabla \bphi_h}_D|
&\lesssim 
\theta k \norm{\Delta\bm_t(t_i)}{\Ltwo{D}}\norm{\bphi_h}{\Ltwo{D}}
\lesssim
\theta k \norm{\bphi_h}{\Ltwo{D}},
\end{align}
where in the last step we used the fact that~$\bm_t\in
L^\infty(0,T;\H^2(D))$.
The second term on the right-hand side of~\cref{eq:Rphih2} can be
estimated as
 \begin{align}\label{eq:inter0}
\left|\dual{|\nabla\bm(t_i)|^2\bm(t_i)}{\bphi_h}_D\right|
 &\leq
\norm{\nabla\bm}{L^\infty(D_T)}^2
\norm{\bm(t_i)\cdot\bphi_h}{L^1(D)}
\nonumber \\
&\lesssim
\norm{\big(\bm(t_i)-\bm_h^i\big)\cdot\bphi_h}{L^1(D)}
+
\norm{\bm_h^i\cdot\bphi_h}{L^1(D)}.
\nonumber \\
&\lesssim
\norm{\bm(t_i)-\bm_h^i}{\Ltwo{D}}\norm{\bphi_h}{\Ltwo{D}}
+
\norm{\bm_h^i\cdot\bphi_h}{L^1(D)}.
 \end{align}
Since $\bm_h^i\cdot\bphi_h$ is a quadratic function on each
$T\in\TT_h$ which vanishes at every node in~$T$, the
semi-norm~$\norm{D^2(\bm_h^i\cdot\bphi_h)}{L^1(T)}$ is a norm,
where $D^2$ is the partial derivative operator of order~2.
If $\widehat T$ is the reference element, then a scaling argument 
and norm equivalence on finite dimensional spaces give
  \begin{align}\label{eq:inter}
  \norm{\bm_h^i\cdot\bphi_h}{L^1(D)}&=\sum_{T\in\TT_h}\norm{\bm_h^i\cdot\bphi_h}{L^1(T)}\simeq
h^3\sum_{T\in\TT_h}\norm{\widehat
\bm_h^i\cdot\widehat\bphi_h}{L^1(\widehat T)}\nonumber\\
  &\simeq h^3\sum_{T\in\TT_h}\norm{D^2(\widehat \bm_h^i\cdot\widehat\bphi_h)}{L^1(\widehat T)}\simeq h^2\sum_{T\in\TT_h}\norm{D^2(\bm_h^i\cdot\bphi_h)}{L^1( T)}.
  \end{align}
  Let $\partial_i$, $i=1,2,3$, denote the directional derivatives
in $\R^3$. Since $\bm_h^i$ and $\bphi_h$ are polynomials of
degree~1 on~$T$, there holds
  \begin{align*}
  |\partial_i\partial_j(\bm_h^i\cdot\bphi_h)|&=|\partial_i((\partial_j\bm_h^i)\cdot\bphi_h+\bm_h^i\cdot(\partial_j\bphi_h))|\\
  &=|(\partial_j\bm_h^i)\cdot(\partial_i\bphi_h)+(\partial_i\bm_h^i)\cdot(\partial_j\bphi_h)|\leq
2|\nabla\bm_h^i||\nabla\bphi_h|,
  \end{align*}
implying 
$|D^2(\bm_h^i\cdot\bphi_h)|\lesssim |\nabla\bm_h^i||\nabla\bphi_h|$.
This and~\cref{eq:inter} yield
  \begin{align}\label{eq:zeros}
  \norm{\bm_h^i\cdot\bphi_h}{L^1(D)}&\lesssim
h^2\sum_{T\in\TT_h}\norm{|\nabla\bm_h^i||\nabla\bphi_h|}{L^1(
T)}=h^2\norm{|\nabla\bm_h^i||\nabla\bphi_h|}{L^1( D)}\nonumber\\
  &\leq h^2 \norm{\nabla\bm_h^i}{\Ltwo{D}}\norm{\nabla\bphi_h}{\Ltwo{D}}
\lesssim
h \norm{\bphi_h}{\Ltwo{D}},
  \end{align}
where in the last step we used the energy bound~\cref{eq:denergy},
and a standard inverse estimate.
Estimate~\cref{eq:Rphih} now follows
from~\cref{eq:Rphih2}--\cref{eq:inter0} and~\cref{eq:zeros},
completing the proof of the lemma.
\end{proof}
\begin{remark}\label{rem:reg}
 It is noted that the assumption $\nabla\bm\in\L^\infty(D_T)$ 
can be replaced by $\nabla\bm\in L^\infty(0,T;\L^4(D))$ to
obtain a weaker bound
 \begin{align*}
  |R(\bphi_h)|
\lesssim \theta k\norm{\bphi_h}{\Ltwo{D}}+\big({h}+\norm{\bm(t_i)-\bm_h^i}{\Hone{D}}\big)
\norm{\bphi_h}{\Hone{D}}.
 \end{align*}
This can be done by use of the continuous embedding
$\Hone{D}\to \L^4(D)$ in~\cref{eq:inter0} and obvious
modifications in the remainder of the proof.
With straightforward modifications, the proof of \Cref{lem:mt vhi} is still valid to prove a weaker estimate with $k$ instead of $k^2$ on the right-hand side of~\cref{eq:mt vhi}.
This eventually results in a reduced rate of convergence $\sqrt{k}$ in \Cref{thm:strongconvLLG}. Analogous arguments hold true for the corresponding results in \Cref{thm:strongconvELLG}.
\end{remark}

As a consequence of the above lemma, we can estimate the approximation of
$\bm_t(t_i)$ by~$\bv_h^i$ as follows.

\begin{lemma}\label{lem:mt vhi}
Under the assumptions of Lemmas~\ref{lem:orthoproj} and~\ref{lem:help} there
holds, with $1/2<\theta\le1$,
\begin{align}\label{eq:mt vhi}
\begin{split}
\frac{\alpha}{C_e}&
\norm{\bm_t(t_i) - \bv_h^i}{\Ltwo{D}}^2
+
%C_e \theta 
k
\norm{\nabla\bm_t(t_i) - \nabla\bv_h^i}{\Ltwo{D}}^2
\\
&+
%C_e
2
\dual{\nabla\bm(t_i)-\nabla\bm_h^i}{\nabla\bm_t(t_i)-\nabla\bv_h^i}
\\
&\qquad\qquad\qquad\qquad\le
C_{\bm}
\left(
h^2 + k^2 +
\norm{\bm(t_i) - \bm_h^i}{\H^1(D)}^2
\right).
\end{split}
\end{align}
\end{lemma}
\begin{proof}
Subtracting~\cref{eq:dllg} from~\cref{eq:help} we obtain
\begin{align*}
\alpha\dual{\bm_t(t_i)-\bv_h^i}{\bphi_h}_D
&+ 
C_e\theta k\dual{\nabla\bm_t(t_i)-\nabla\bv_h^i}{\nabla\bphi_h}_D
\\
+
C_e&\dual{\nabla\bm(t_i)-\nabla\bm_h^i}{\nabla\bphi_h}_D
\\
&=
\dual{\bm_h^i\times\bv_h^i-\bm(t_i)\times\bm_t(t_i)}{\bphi_h}_D+R(\bphi_h),
\quad\bphi_h\in\KK_{\bm_h^i}.
\end{align*}
Using the above equation, writing $\bphi := \bm_t(t_i) - \bv_h^i$ 
and referring to the left-hand side
of~\cref{eq:mt vhi} we try to estimate
\begin{align}\label{eq:phi nab phi}
\begin{split}
\alpha
&\norm{\bphi}{\Ltwo{D}}^2
+
C_e\theta k
\norm{\nabla\bphi}{\Ltwo{D}}^2
+
C_e
\dual{\nabla\bm(t_i)-\nabla\bm_h^i}{\nabla\bphi}
\\
&=
\alpha
\big(
\norm{\bphi}{\Ltwo{D}}^2
-
\dual{\bphi}{\bphi_h}_D 
\big)
+
C_e\theta k
\big(
\norm{\nabla\bphi}{\Ltwo{D}}^2
-
\dual{\nabla\bphi}{\nabla\bphi_h}_D 
\big)
\\
&\quad 
+
C_e
\dual{\nabla\bm(t_i)-\nabla\bm_h^i}{\nabla\bphi-\nabla\bphi_h}
\\
&\quad+
\dual{\bm_h^i\times\bv_h^i-\bm(t_i)\times\bm_t(t_i)}{\bphi_h}_D
+
R(\bphi_h)
\\
&
=: T_1 + \cdots + T_5.
\end{split}
\end{align}
Choosing~$\bphi_h=\P_h^i\bm_t(t_i)-\bv_h^i\in\KK_{\bm_h^i}$
where~$\P_h^i$
is defined in \Cref{lem:orthoproj}, we deduce from that lemma that
\begin{align}\label{eq:ests1}
\begin{split}
|T_1|
&\lesssim
\norm{\bphi}{\Ltwo{D}}
\norm{\bphi-\bphi_h}{\Ltwo{D}}
\\
&\lesssim
\norm{\bphi}{\Ltwo{D}}
\big(h+\norm{\bm(t_i)-\bm_h^i}{\Hone{D}}\big),\\
|T_2|
&\lesssim
k
\norm{\nabla\bphi}{\Ltwo{D}}
\norm{\nabla\bphi-\nabla\bphi_h}{\Ltwo{D}}
\\
&\lesssim
k
\norm{\nabla\bphi}{\Ltwo{D}}
\big(h+\norm{\bm(t_i)-\bm_h^i}{\Hone{D}}\big),\\
|T_3|
&\lesssim
\norm{\nabla\bm(t_i)-\nabla\bm_h^i}{\Ltwo{D}}
\norm{\nabla\bphi-\nabla\bphi_h}{\Ltwo{D}}
\\
&\lesssim
\norm{\nabla\bm(t_i)-\nabla\bm_h^i}{\Ltwo{D}}
\big(h+\norm{\bm(t_i)-\bm_h^i}{\Hone{D}}\big)
\\
&\lesssim
h^2
+
\norm{\bm(t_i)-\bm_h^i}{\Hone{D}}^2.
\end{split}
\end{align}
For the term $T_4$ since $\dual{\bm_h^i\times \bphi}{\bphi}_D=0$
implies
$\dual{\bm_h^i\times \bv_h^i}{\bphi}_D =
\dual{\bm_h^i\times \bm_t(t_i)}{\bphi}_D$
we deduce, with the help of \Cref{lem:orthoproj} again,
\begin{align*}
|T_4|
&\le
|\dual{\bm_h^i\times\bv_h^i-\bm(t_i)\times\bm_t(t_i)}{\bphi_h-\bphi}_D|
\\
&\qquad+
|\dual{\bm_h^i\times\bv_h^i-\bm(t_i)\times\bm_t(t_i)}{\bphi}_D|
\nonumber \\
&=
|\dual{\bm_h^i\times\bv_h^i-\bm(t_i)\times\bm_t(t_i)}{\bphi_h-\bphi}_D|
+
|\dual{(\bm_h^i -\bm(t_i))\times\bm_t(t_i)}{\bphi}_D|
\nonumber \\
&\lesssim
\norm{\bm_h^i\times\bv_h^i-\bm(t_i)\times\bm_t(t_i)}{\widetilde\H^{-1}(D)}
\norm{\bphi-\bphi_h}{\Hone{D}}
\\
&\qquad +
\norm{\bm_h^i-\bm(t_i)}{\Ltwo{D}}\norm{\bphi}{\Ltwo{D}}
\nonumber \\
&\lesssim
\norm{\bm_h^i\times\bv_h^i-\bm(t_i)\times\bm_t(t_i)}{\widetilde\H^{-1}(D)}
\big(h+\norm{\bm(t_i)-\bm_h^i}{\Hone{D}}\big)
\nonumber \\
&\quad+
\norm{\bm_h^i-\bm(t_i)}{\Ltwo{D}}\norm{\bphi}{\Ltwo{D}}.
\end{align*}
\Cref{lem:h} implies
\begin{align*}
 \norm{\bm_h^i&\times\bv_h^i-\bm(t_i)\times\bm_t(t_i)}{\widetilde\H^{-1}(D)}\\
 &\leq \norm{\bm_h^i\times(\bv_h^i-\bm_t(t_i))}{\widetilde\H^{-1}(D)}
 +\norm{(\bm_h^i-\bm(t_i))\times\bm_t(t_i)}{\widetilde\H^{-1}(D)}\\
 &\lesssim\norm{\bm_h^i}{\H^{1}(D)}\norm{\bm_t(t_i)-\bv_h^i}{\Ltwo{D}}
 +\norm{\bm_t(t_i)}{\Ltwo{D}}\norm{\bm(t_i)-\bm_h^i}{\Hone{D}}\\
 &\lesssim 
 \norm{\bphi}{\Ltwo{D}}
 +\norm{\bm(t_i)-\bm_h^i}{\Hone{D}},
\end{align*}
where in the last step we used the regularity of $\bm_t$ and the
bound~\cref{eq:const} and~\cref{eq:denergy}. 
Therefore
\begin{align}\label{eq:T4}
|T_4|
&\lesssim
\big(
\norm{\bphi}{\Ltwo{D}}
+
\norm{\bm(t_i)-\bm_h^i}{\Hone{D}}
\big)
\big(h+\norm{\bm(t_i)-\bm_h^i}{\Hone{D}}\big)
\nonumber \\
%&\quad
%+
%\norm{\bm_h^i-\bm(t_i)}{\Ltwo{D}}\norm{\bphi}{\Ltwo{D}}
%\nonumber \\
&\lesssim
\norm{\bphi}{\Ltwo{D}}
\big(h+\norm{\bm(t_i)-\bm_h^i}{\Hone{D}}\big)
+
h^2
+
\norm{\bm(t_i)-\bm_h^i}{\Hone{D}}^2.
\end{align}
Finally, for~$T_5$, \Cref{lem:help}, the triangle inequality,
and \Cref{lem:orthoproj} give
\begin{align}\label{eq:T5}
|T_5|
&\lesssim 
\big(h+k+\norm{\bm(t_i)-\bm_h^i}{\Ltwo{D}}\big)
\norm{\bphi_h}{\Ltwo{D}}
\nonumber\\
&\lesssim 
\big(h+k+\norm{\bm(t_i)-\bm_h^i}{\Ltwo{D}}\big)
\norm{\bphi}{\Ltwo{D}}
\nonumber\\
&\quad+ 
\big(h+k+\norm{\bm(t_i)-\bm_h^i}{\Ltwo{D}}\big)
\big(h+\norm{\bm(t_i)-\bm_h^i}{\H^1(D)}\big)
\nonumber\\
&\lesssim 
\big(h+k+\norm{\bm(t_i)-\bm_h^i}{\Ltwo{D}}\big)
\norm{\bphi}{\Ltwo{D}}
+
h^2 + k^2
+
\norm{\bm(t_i)-\bm_h^i}{\H^1(D)}^2.
\end{align}
Altogether,~\cref{eq:phi nab phi}--\cref{eq:T5} yield, for any
$\epsilon>0$,
\begin{align*}
\frac{\alpha}{C_e}
&
\norm{\bphi}{\Ltwo{D}}^2
+
\theta k
\norm{\nabla\bphi}{\Ltwo{D}}^2
+
\dual{\nabla\bm(t_i)-\nabla\bm_h^i}{\nabla\bphi}
\\
&\lesssim
\big(
\norm{\bphi}{\Ltwo{D}}
+
k
\norm{\nabla\bphi}{\Ltwo{D}}
\big)
\big(h+k+\norm{\bm(t_i)-\bm_h^i}{\H^1(D)}\big)
+ h^2 + k^2
\\
&\quad+
\norm{\bm(t_i)-\bm_h^i}{\Hone{D}}^2
\\
&\le
\epsilon
\norm{\bphi}{\Ltwo{D}}^2
+
\epsilon k^2
\norm{\nabla\bphi}{\Ltwo{D}}^2
+
(1+\epsilon^{-1})
\big(h^2+k^2+\norm{\bm(t_i)-\bm_h^i}{\H^1(D)}^2\big).
\end{align*}
The required estimate~\cref{eq:mt vhi} is obtained for
$\epsilon=\min\{\alpha/(2C_e),\theta-1/2\}$.
\end{proof}

Three more lemmas are required for the proof of
\Cref{thm:strongconvELLG}.
Analogously to \Cref{lem:help} we show that for the ELLG system,
$\bm_t(t_i)$ solves the same equation as~$\bv_h^i$, up to an error term.
\begin{lemma}\label{lem:help2}
Let $(\bm,\bH,\lambda)$ denote a strong solution of ELLG
 which satisfies 
\begin{align*}
\bm &\in L^{\infty}(0,T;\W^{1,\infty}(D)),
\\
\bm_t &\in L^{\infty}(0,T;\H^{2}(D)),
\\
\bH &\in \L^\infty(D_T). 
\end{align*}
Then, for $i=0,\ldots,N$ there holds
  \begin{align}\label{eq:help2}
   \alpha\dual{\bm_t(t_i)}{\bphi_h}_{D} &+
\dual{(\bm(t_i)\times\bm_t(t_i))}{\bphi_h}_{D}
+
C_e \theta k\dual{\nabla\bm_t(t_i)}{\nabla \bphi_h}_D
\nonumber \\
   &=-C_e \dual{\nabla \bm(t_i)}{\nabla \bphi_h}_{D}
   +\dual{\bH(t_i)}{\bphi_h}_D+\widetilde R(\bphi_h)
\quad\forall
\bphi_h\in \KK_{\bm_h^i},
  \end{align}
where
\begin{equation}\label{eq:Rphih3}
|\widetilde R(\bphi_h)|
\leq 
C_{\rm \widetilde R}
\big(
h+\theta k +
\norm{\bm(t_i)-\bm_h^i}{\Ltwo{D}}
\big)
\norm{\bphi_h}{\Ltwo{D}}.
\end{equation}
Here, the constant $C_{\rm \widetilde R}>0$ depends only on
the regularity of~$\bm$ and~$\bH$, and the shape regularity of~$\TT_h$.
\end{lemma}
\begin{proof}
The proof is similar to that of \Cref{lem:help}. Instead
of~\cref{eq:mt phih} we now have
\begin{align*}
  \begin{split}
   \alpha\dual{\bm_t(t_i)}{\bphi_h}_{D} &+
\dual{(\bm(t_i)\times \bm_t(t_i))}{\bphi_h}_{D}\\
   &=-C_e \dual{\nabla \bm(t_i)}{\nabla
\bphi_h}_{D}+C_e\dual{|\nabla\bm(t_i)|^2\bm(t_i)}{\bphi_h}_D\\
   &\quad +
\dual{\bH(t_i)}{\bphi_h}_D
-
\dual{(\bm(t_i)\cdot\bH(t_i))\bm(t_i)}{\bphi_h}_D,
   \end{split}
  \end{align*}
and thus
\[
\widetilde R(\bphi_h)
:=
R(\bphi_h)
-
\dual{(\bm(t_i)\cdot\bH(t_i))\bm(t_i)}{\bphi_h}_D,
\]
where $R(\bphi_h)$ is given in~\cref{eq:Rphih2}.
Therefore, it suffices to estimate the term
$\dual{(\bm(t_i)\cdot\bH(t_i))\bm(t_i)}{\bphi_h}_D$. Since
 \begin{align*}
|\dual{(\bm(t_i)\cdot\bH(t_i))\bm(t_i)}{\bphi_h}_D|
&\leq  
\norm{\bm}{\L^\infty(D_T)}
\norm{\bH}{\L^\infty(D_T)}
\norm{\bm(t_i)\cdot\bphi_h}{L^1(D)},
\end{align*}
the proof follows exactly the same way as that of \Cref{lem:help};
cf.~\cref{eq:inter0}. Thus we prove~\cref{eq:Rphih3}.
\end{proof}
 
We will 
establish a recurrence estimate for the eddy
current part of the solution. We first introduce
\begin{equation}\label{eq:Ei ei}
\be_i:=\Pi_{\NN\DD}\bH(t_i)-\bH_h^{i},
\quad
f_i := \Pi_{\SS}\lambda(t_i)-\lambda_h^i,
\quad\text{and}\quad
\bE_i:=
(\be_i,f_i),
\end{equation}
where $\Pi_{\NN\DD}$ denotes the usual interpolation
operator onto $\NN\DD^1(\TT_h)$.

%Moreover, we require the norm 
%\begin{align}\label{eq:hnorm}
% \norm{\bE_i}{h}^2 := \norm{\Pi_{\NN\DD}\bH(t_i)-\bH_h^{i}}{\Ltwo{D}}^2 - \dual{\dtn_h (\Pi_{\SS}\lambda(t_i)-\lambda_h^i)}{\Pi_{\SS}\lambda(t_i)-\lambda_h^i}_\Gamma
%\end{align}
%(the ellipcitiy~\cref{eq:dtnelliptic}
%guarantees that this is actually a norm equivalent to the graph norm of $\Ltwo{D}\times H^{1/2}(\Gamma)$ uniformly in $h$).

\begin{lemma}\label{lem:cea}
Let $(\bm,\bH,\lambda)$ be a strong solution of ELLG
which satisfies 
\begin{align*}
\bm &\in W^{1,\infty}(0,T;\Ltwo{D})\cap \W^{1,\infty}(D_T), \\
\bH &\in L^\infty(0,T; \H^2(D))
\cap W^{1,\infty}(0,T;\Hone{D})
\cap W^{2,\infty}(0,T;\Ltwo{D}), \\
\lambda &\in W^{1,\infty}(0,T; H^1(\Gamma))\text{ such
that }\dtn\lambda_t\in L^\infty(0,T; H^{1/2}_{\comment{\rm pw}}(\Gamma)).
\end{align*}
Then for all $k\in(0,1/2]$ and $i\in\{1,\ldots,N-1\}$ there holds 
\begin{align}\label{eq:Eip}
\norm{\bE_{i+1}}{h}^2 &+ \frac{k}{2}\norm{\nabla\times\be_{i+1}}{\Ltwo{D}}^2\nonumber\\
&\leq 
(1+2k)\norm{\bE_{i}}{h}^2
+
C_{\bH}k\big(\norm{\bm_t(t_i)-\bv_h^i}{\Ltwo{D}}^2+ h^2 + k^2\big) 
\end{align}
where the constant $C_{\bH}>0$ depends only on the smoothness of $\bH$, $\lambda$, and on the shape-regularity of $\TT_h$.
\end{lemma}
\begin{proof}
Recalling equation~\cref{eq:eddygeneral} and in view
of~\cref{eq:Ei ei},
we will establish a similar equation for~$\Pi_{\NN\DD}\bH(t_i)$.
With~$A:=(\bH,\lambda)$ and~$A_h:=
(\Pi_{\NN\DD}\bH,\Pi_{\SS}\lambda)\in\XX_h$,
it follows from~\cref{eq:wssymm2} that
\begin{align}\label{eq:Pih H}
\begin{split}
a_h(\partial_t &A_h(t),B_{h}) 
+ 
b(\Pi_{\NN\DD}\bH(t),\bxi_{h})
\\
&=  -\dual{\bm_t(t)}{\bxi_{h}}_{D}
-
a_h(\partial_t(A-A_h)(t),B_{h}) 
+
b((1-\Pi_{\NN\DD})\bH(t),\bxi_{h}) 
\notag
\\
&\quad
+
\dual{(\dtn-\dtn_h)\lambda_t(t)}{\zeta_{h}}_\Gamma
\\
&=:-\dual{\bm_t(t)}{\bxi_{h}}_{D}+R_0(A(t),B_{h})
\quad\forall B_h=(\bxi_h,\zeta_h)\in\XX_h.
\end{split}
\end{align}
The continuity of $a_h(\cdot,\cdot)$ and $b(\cdot,\cdot)$, the approximation properties of
$\Pi_{\NN\DD}$ and $\Pi_{\SS}$, 
and~\cref{eq:dtnelliptic},~\cref{eq:dtnh} give
\begin{align*}
 |R_0&(A(t),B_{h})|\\
&\lesssim
\norm{(1-\Pi_{\NN\DD})\bH_t(t)}{\Ltwo{D}}
\norm{\bxi_{h}}{\Ltwo{D}}
+
\norm{(1-\Pi_{\SS})\lambda_t(t)}{H^{1/2}(\Gamma)}
\norm{\zeta_{h}}{H^{1/2}(\Gamma)}
\\
&\quad
+
\norm{(1-\Pi_{\NN\DD})\bH(t)}{\Hcurl{D}}
\norm{\bxi_{h}}{\Ltwo{D}}
+
h\norm{\dtn\lambda_t(t)}{H^{1/2}_{\comment{\rm pw}}(\Gamma)}
\norm{\zeta_h}{H^{1/2}(\Gamma)}
\\
%&\lesssim 
%h
%\int_0^T
%\Big(
%\norm{\bH_t(t)}{\Hone{D}}\norm{\bxi}{\Ltwo{D}}
%+
%\norm{\lambda_t(t)}{H^2(\Gamma)}\norm{\zeta}{H^{1/2}(\Gamma)}
%\\
%&\quad+ 
%\norm{(1-\Pi_{\NN\DD})\bH(t)}{\Hcurl{D}}\norm{\bxi}{\Hcurl{D}}
%\, dt
%\\
&\lesssim 
h
\big(
\norm{\bH_t(t)}{\Hone{D}}
+
\norm{\bH(t)}{\H^2(D)}
+
\norm{\dtn\lambda_t(t)}{H^{1/2}_{\comment{\rm pw}}(\Gamma)}
\big)
\norm{B}{\XX}.
\end{align*}
The regularity assumptions on $\bH$ and $\lambda$ yield
\begin{align}\label{eq:Rsize}
 |R_0(A(t),B_{h})|\lesssim h\norm{B_{h}}{\XX}.
\end{align}
Recalling definition~\cref{eq:dt}
and using Taylor's Theorem, we have
\begin{align*}
d_tA_h(t_{i+1})
=
\partial_t A_h(t_i) + r_i,
\end{align*}
where $r_i$ is the remainder (in the integral form)
of the Taylor expansion which satisfies
$\norm{r_i}{\Ltwo{D}\times H^{1/2}(\Gamma)}\lesssim k(\norm{\bH_{tt}}{L^\infty(0,T;\Ltwo{D})}+\norm{\lambda_{tt}}{L^\infty(0,T;H^{1/2}(\Gamma)})$.
Therefore,~\cref{eq:Pih H} and~\cref{eq:Rsize} imply
for all $B_{h}=(\bxi_{h},\zeta_{h})\in\XX_{h}$
\begin{align}\label{eq:auxx3}
a_h(d_t A_h(t_{i+1}),B_{h}) + b(\Pi_{\NN\DD}\bH(t_{i}),\bxi_{h})
=-\dual{\bm_t(t_{i})}{\bxi_{h}}_D+R_1(A(t_{i}),B_{h}),
\end{align}
where 
$R_1(A(t_{i}),B_{h}):=R_0(A(t_{i}),B_{h}) - a_h(r_i,B_{h})$ so that
\begin{align}\label{eq:R1size}
 |R_1(A(t_i),B_{h})|
&\lesssim (h + k(\norm{\bH_{tt}}{L^\infty(0,T;\Ltwo{D})}+\norm{\lambda_{tt}}{L^\infty(0,T;H^{1/2}(\Gamma)}))\norm{B_{h}}{\XX}
\notag
\\
&\lesssim (h+k)\norm{B_{h}}{\XX},
\end{align}
where we used the uniform continuity of $a_h(\cdot,\cdot)$ in $h$ obtained from~\cref{eq:dtnelliptic}.
Subtracting~\cref{eq:eddygeneral} from~\cref{eq:auxx3} and 
setting~$B_h=\bE_{i+1}$ yield
\begin{align*}
a_h(d_t \bE_{i+1},\bE_{i+1}) + b(\be_{i+1},\be_{i+1})=
\dual{\bv_h^i-\bm_t(t_i)}{\be_{i+1}}_D+R_1(A(t_i),\bE_{i+1}).
\end{align*}
Multiplying the above equation by~$k$ and using the ellipticity properties of $a_{h}(\cdot,\cdot)$, 
$b(\cdot,\cdot)$, and the Cauchy-Schwarz inequality, and
recalling definition~\cref{eq:hnorm} of the $h$-norm, we deduce
\begin{align*}
 \norm{\bE_{i+1}}{h}^2 &+ k\norm{\nabla\times\be_{i+1})}{\Ltwo{D}}^2\nonumber\\
 &\leq \norm{\bE_{i}}{h}\norm{\bE_{i+1}}{h}+
k\norm{\bv_h^i-\bm_t(t_i)}{\Ltwo{D}}\norm{\be_{i+1}}{\Ltwo{D}}\nonumber
+ Ck(h+k)\norm{\bE_{i+1}}{\XX},
\end{align*}
for some constant $C>0$ which does not depend on $h$ or $k$.
With Young's inequality, this implies
\begin{align}\label{eq:E e}
\begin{split}
\norm{\bE_{i+1}}{h}^2 
&+ 
k\norm{\nabla\times\be_{i+1}}{\Ltwo{D}}^2
\\&\leq 
\frac{1}{2} \norm{\bE_{i}}{h}^2
+
\frac{1}{2} \norm{\bE_{i+1}}{h}^2
+
\frac{k}{2}\frac{C_\dtn}{(C_\dtn-1)}
\norm{\bv_h^i-\bm_t(t_i)}{\Ltwo{D}}^2
\\
&\quad 
+
\frac{k}{2}
\frac{C_\dtn-1}{C_\dtn}
\norm{\be_{i+1}}{\Ltwo{D}}^2
+ 
\frac{k}{2C_\dtn}\norm{\bE_{i+1}}{\XX}^2
+
\frac{kC_\dtn}{2}
C^2
(h+k)^2.
\end{split}
\end{align}
Note that
\begin{align*}
\frac{k}{2} 
\Big(\frac{C_\dtn-1}{C_\dtn}
&\norm{\be_{i+1}}{\Ltwo{D}}^2
+ 
\frac{1}{C_\dtn}\norm{\bE_{i+1}}{\XX}^2
\Big)\\
&=
\frac{k}{2}
\big(
\norm{\be_{i+1}}{\Ltwo{D}}^2
+
\frac{1}{C_\dtn}
\norm{\nabla\times\be_{i+1}}{\Ltwo{D}}^2
+
\frac{1}{C_\dtn}
\norm{f_i}{H^{1/2}(\Gamma)}^2
\big)
\\
&\le
\frac{k}{2}
\big(
\norm{\be_{i+1}}{\Ltwo{D}}^2
+
\frac{1}{C_\dtn}
\norm{\nabla\times\be_{i+1}}{\Ltwo{D}}^2
-
\dual{\dtn_h f_i}{f_i}_{\Gamma}
\big)
\\
&\le
\frac{k}{2}
%(2+C_\dtn)
\norm{\bE_{i+1}}{h}^2
+
\frac{k}{2C_\dtn}
\norm{\nabla\times\be_{i+1}}{\Ltwo{D}}^2.
\end{align*}
Hence~\cref{eq:E e} yields (after multiplying by 2)
\begin{align*}
%\big(1-k(2+C_\dtn)\big)
(1-k)
\norm{\bE_{i+1}}{h}^2 
&+ 
k(2-\frac{1}{C_\dtn})
\norm{\nabla\times\be_{i+1}}{\Ltwo{D}}^2
\\
&\leq 
\norm{\bE_{i}}{h}^2
+
\frac{kC_\dtn}{C_\dtn-1}
\norm{\bv_h^i-\bm_t(t_i)}{\Ltwo{D}}^2
+
k C_\dtn C^2 (h+k)^2.
\end{align*}
Dividing by~$1-k$, using the fact that~$1\le1/(1-k)\le1+2k\le2$
(since $0<k\le1/2$), and noting that~$C_\dtn\ge1$,
we obtain the desired estimate~\cref{eq:Eip},
concluding the proof.
\end{proof}

Similarly to \Cref{lem:mt vhi} we now prove the following lemma for the
ELLG system.
\begin{lemma}\label{lemma:cea0}
Let $(\bm,\bH,\lambda)$ be a strong solution of ELLG
 which satisfies 
\begin{align*}
\bm &\in W^{2,\infty}\big(0,T;\Hone{D}\big)\cap W^{1,\infty}\big(0,T;
\W^{1,\infty}(D)\cap \H^2(D)\big), \\
\bH&\in L^\infty(0,T; \H^2(D)\cap L^\infty(D))
\cap
W^{2,\infty}(0,T;\Ltwo{D}). 
\end{align*}
Then, for $1/2<\theta\le1$, we have
\begin{align*}%\label{eq:mt vhiH}
\frac{\alpha}{C_e}&
\norm{\bm_t(t_i)  - \bv_h^i}{\Ltwo{D}}^2
+
k
\norm{\nabla\bm_t(t_i) - \nabla\bv_h^i}{\Ltwo{D}}^2
\\
&+
2
\dual{\nabla\bm(t_i)-\nabla\bm_h^i}{\nabla\bm_t(t_i)-\nabla\bv_h^i}
\\
&\qquad\qquad\le
C_{\bH}
\Big(
h^2 + k^2 +
\norm{\bm(t_i) - \bm_h^i}{\H^1(D)}^2
+
\norm{\bH(t_i) - \bH_h^i}{\Ltwo{D}}^2
\Big).
\end{align*}
\end{lemma}
\begin{proof}
The proof follows that of \Cref{lem:mt vhi}.
Subtracting~\cref{eq:dllg} from~\cref{eq:help2} and putting
$\bphi:=\bm_t(t_i)-\bv_h^i$ we obtain
for $\bphi_h:=\P_h^i\bm_t(t_i)-\bv_h^i\in\KK_{\bm_h^i}$
\begin{align*}
\alpha\dual{\bphi}{\bphi_h}_D
&+ 
C_e\theta k\dual{\nabla\bphi}{\nabla\bphi_h}_D
+
C_e\dual{\nabla\bm(t_i)-\nabla\bm_h^i}{\nabla\bphi_h}_D
\\
&=
\dual{\bm_h^i\times\bv_h^i-\bm(t_i)\times\bm_t(t_i)}{\bphi_h}_D
+\dual{\bH(t_i)-\bH_h^i}{\bphi_h}_D
+
\widetilde R(\bphi_h).
\end{align*}
Similarly to~\cref{eq:phi nab phi} we now have
\begin{align*}
\alpha\norm{\bphi}{\Ltwo{D}}^2
&+
C_e\theta
k\norm{\nabla\bphi}{\Ltwo{D}}^2
+
C_e\dual{\nabla(\bm(t_i)-\bm_h^i)}{\nabla\bphi}_D\\
&=
T_1 + \cdots + T_4 + \widetilde T_5 + \widetilde T_6,
\end{align*}
where $T_1$, \ldots, $T_4$ are defined as in~\cref{eq:phi nab phi} whereas
\[
\widetilde T_5 :=  \dual{\bH(t_i)-\bH_h^i}{\bphi_h}_D
\quad\text{and}\quad
\widetilde T_6 := \widetilde R(\bphi_h).
\]
Estimates for~$T_1$, \ldots, $T_4$ have been carried out in the proof of
\Cref{lem:mt vhi}. For $\widetilde T_5$ we have
\begin{align*}
|\widetilde T_5|
&\le
\norm{\bH(t_i)-\bH_h^i}{\Ltwo{D}}\norm{\bphi_h}{\Ltwo{D}}
\\
&\lesssim
\norm{\bH(t_i)-\bH_h^i}{\Ltwo{D}}\norm{\bphi}{\Ltwo{D}}
+
\norm{\bH(t_i)-\bH_h^i}{\Ltwo{D}}
\big(
h + \norm{\bm(t_i)-\bm_h^i}{\Ltwo{D}}
\big)
\\
&\lesssim
\norm{\bH(t_i)-\bH_h^i}{\Ltwo{D}}\norm{\bphi}{\Ltwo{D}}
+
h^2 
+
\norm{\bH(t_i)-\bH_h^i}{\Ltwo{D}}^2
\\
&\qquad+ 
\norm{\bm(t_i)-\bm_h^i}{\Ltwo{D}}^2,
\end{align*}
where we used the triangle inequality and invoked
\Cref{lem:orthoproj} to estimate
$\norm{\bphi-\bphi_h}{\Ltwo{D}}$.
Finally, for $\widetilde T_6$ we use~\cref{eq:Rphih3}, the
triangle inequality, and \Cref{lem:orthoproj} to obtain
\begin{align*}
|\widetilde T_6|
&\lesssim
\big(
h + k +
\norm{\bm(t_i)-\bm_h^i}{\Ltwo{D}}
\big)
\norm{\bphi_h}{\Ltwo{D}}
\\
&\lesssim
\big(
h + k +
\norm{\bm(t_i)-\bm_h^i}{\Ltwo{D}}
\big)
\norm{\bphi}{\Ltwo{D}}
\\
&\quad
+
\big(
h + k +
\norm{\bm(t_i)-\bm_h^i}{\Ltwo{D}}
\big)
\big(
h + \norm{\bm(t_i)-\bm_h^i}{\Hone{D}}
\big)
\\
&\lesssim
h^2 + k^2 +
\norm{\bm(t_i)-\bm_h^i}{\Hone{D}}^2
+
\big(
h + k +
\norm{\bm(t_i)-\bm_h^i}{\Ltwo{D}}
\big)
\norm{\bphi}{\Ltwo{D}}.
\end{align*}
The proof finishes in exactly the same manner as that of
\Cref{lem:mt vhi}.
\end{proof}  

\subsection{Proof of \Cref{thm:weakconv}}\label{section:weak}
We are now ready to prove that the
problem~\cref{eq:strong}--\cref{eq:con} has a weak solution.
\begin{proof}
We recall from~\cref{eq:wc1}--\cref{eq:wc5} that $\bm\in \Hone{D_T}$, 
$(\bH,\lambda)\in L^2(0,T;\XX)$ and $\bH\in H^1(0,T;\Ltwo{D})$.
By virtue of \Cref{lem:bil for} it suffices to prove that
$(\bm,\bH,\lambda)$ satisfies~\cref{eq:wssymm1} 
and~\cref{eq:bil for}.

Let $\bphi\in C^\infty(D_T)$ and $B:=(\bxi,\zeta)\in L^2(0,T;\XX)$.
On the one hand, we define the test function
$\bphi_{hk}:=\Pi_{\SS}(\bm_{hk}^-\times\bphi)$
as the usual interpolant of~$\bm_{hk}^-\times\bphi$ into
$\SS^1(\TT_h)^3$. 
By definition, $\bphi_{hk}(t,\cdot) \in\KK_{\bm_h^j}$ for all
$t\in[t_j,t_{j+1})$. 
On the other hand, it follows from \Cref{lem:den pro}
that there exists~$B_h:=(\bxi_h,\zeta_h)\in\XX_h$ converging
to~$B\in\XX$.
\Cref{eq:mhk Hhk} hold with these
test functions. The main idea of the proof is to
pass to the limit in~\cref{eq:mhk Hhk1} and~\cref{eq:mhk Hhk2} to
obtain~\cref{eq:wssymm1} and~\cref{eq:bil for}, respectively.

In order to prove that~\cref{eq:mhk Hhk1}
implies~\cref{eq:wssymm1} we \comment{need to} prove that as~$h,k\to0$
\begin{subequations}\label{eq:conv}
\begin{align}
\dual{\bv_{hk}^-}{\bphi_{hk}}_{D_T} 
&\to
\dual{\bm_t}{ \bm\times\bphi}_{D_T},
\label{eq:conv1}
\\
\dual{\bm_{hk}^-\times\bv_{hk}^-}{\bphi_{hk}}_{D_T}
&\to
\dual{\bm\times\bm_t}{\bm\times\bphi}_{D_T},
\label{eq:conv2}
\\
k\dual{\nabla\bv_{hk}^-}{\nabla\bphi_{hk}}_{D_T}
&\to0,
\label{eq:conv3}
\\
\dual{\nabla\bm_{hk}^-}{\nabla\bphi_{hk}}_{D_T} 
&\to
\dual{\nabla\bm}{\nabla(\bm\times\bphi)}_{D_T},
\label{eq:conv4}
\\
\dual{\bH_{hk}^-}{\bphi_{hk}}_{D_T} 
&\to
\dual{\bH}{ \bm\times\bphi}_{D_T}.
\label{eq:conv5}
\end{align}
\end{subequations}
\comment{The proof has been carried out in~\cite{Abert_etal, alouges,ellg} and is therefore
omitted.}

Next, recalling that~$B_h\to B$ in~$\XX$
we prove that~\cref{eq:mhk Hhk2}
implies~\cref{eq:bil for} by proving
\begin{subequations}\label{eq:convb}
\begin{align}
\dual{\partial_t\bH_{hk}}{\bxi_h}_{D_T}
&\to
\dual{\bH_t}{\bxi}_{D_T},
\label{eq:convb1} 
\\
\dual{\dtn_h\partial_t\lambda_{hk}}{\zeta_h}_{\Gamma_T}
&\to
\dual{\dtn\lambda_t}{\zeta}_{\Gamma_T}, \label{eq:spec}
\\
\dual{\nabla\times\bH_{hk}^+}{\nabla\times\bxi_h}_{D_T}
&\to
\dual{\nabla\times\bH}{\nabla\times\bxi}_{D_T},
\\
\dual{\bv_{hk}^-}{\bxi_h}_{D_T}
&\to
\dual{\bv}{\bxi}_{D_T}.
\end{align}
\end{subequations}
The proof is similar to that of~\cref{eq:conv} (where we
use \Cref{lem:wea con}
for the proof of~\cref{eq:spec}) and is
therefore omitted. 

\comment{Passing to the limit in~\cref{eq:mhk Hhk1}--\cref{eq:mhk Hhk2} and
using properties~\cref{eq:conv}--\cref{eq:convb} prove
Items~3 and~5 of \Cref{def:fembemllg}.}
%This proves~(3) and~(5) of \Cref{def:fembemllg}.

Finally, we obtain $\bm(0,\cdot)=\bm^0$, $\bH(0,\cdot)=\bH^0$, and
$\lambda(0,\cdot)=\lambda^0$ from the weak convergence and the
continuity of the trace operator.
This and $|\bm|=1$ yield Statements~(1)--(2) of
\Cref{def:fembemllg}. To obtain~(4), note that
$\nabla_\Gamma\colon H^{1/2}(\Gamma)\to \H_\perp^{-1/2}(\Gamma)$ and
$\bn\times(\bn\times(\cdot))\colon \Hcurl{D}\to  \H_\perp^{-1/2}(\Gamma)$ are bounded
linear operators; see~\cite[Section~4.2]{buffa2} for exact definition of the spaces and the result. 
Weak convergence then proves \comment{Item~4} of
\Cref{def:fembemllg}.
Estimate~\cref{eq:energybound2} follows by weak
lower-semicontinuity and the energy bound~\cref{eq:denergy}.
This completes the proof of the theorem.
\end{proof}

\subsection{Proof of \Cref{thm:strongconvLLG}}\label{section:strong}
In this subsection we invoke 
Lemmas~\ref{lem:orthoproj},~\ref{lem:help}, and~\ref{lem:mt vhi} 
to prove a priori error estimates for the pure LLG equation.
\begin{proof}
It follows from Taylor's Theorem,~\cref{eq:mhip1}, and Young's
inequality that
\begin{align}\label{eq:mt mh}
\begin{split}
\norm{\bm(t_{i+1})&-\bm_h^{i+1}}{\Hone{D}}^2\\
%&=
%\norm{\big(\bm(t_{i}) + k\bm_t(t_i) + \frac{k^2}{2}\bm_{tt}(t_i^*)\big)
%-
%\big(\bm_h^{i}+k\bv_h^i\big)}{\Hone{D}}^2
%\nonumber\\
&\leq 
(1+k)\norm{\bm(t_i)+k\bm_t(t_i)-(\bm_h^{i}+k\bv_h^i)}{\Hone{D}}^2
\\
&\qquad+
(1+k^{-1})\frac{k^4}{4}\norm{\bm_{tt}}{L^\infty(0,T;\Hone{D})}^2
\\
&\leq
(1+k)\norm{\bm(t_i)-\bm_h^{i}}{\Hone{D}}^2
+
(1+k)k^2\norm{\bm_t(t_i)-\bv_h^{i}}{\Hone{D}}^2
\\
&\quad+
2k(1+k)\dual{\bm(t_i)-\bm_h^{i}}{\bm_t(t_i)-\bv_h^{i}}_D
\\
&\quad+
2k(1+k)
\dual{\nabla\bm(t_i)-\nabla\bm_h^{i}}{\nabla\bm_t(t_i)-\nabla\bv_h^{i}}_D
\\
&\quad+
k^3\norm{\bm}{W^{2,\infty}(0,T;\Hone{D})}^2,
\end{split}
 \end{align}
recalling that $0<k\le1$. The third term on the right-hand side is
estimated as
\begin{align*}
\big|
2k(1+k)\dual{\bm(t_i)-\bm_h^{i}}{\bm_t(t_i)-\bv_h^{i}}_D
\big|
&\le
\delta^{-1}
k(1+k) 
\norm{\bm(t_i)-\bm_h^{i}}{\Ltwo{D}}^2
\\
&\quad
+
\delta
k(1+k) 
\norm{\bm_t(t_i)-\bv_h^{i}}{\Ltwo{D}}^2,
\end{align*}
for any $\delta>0$, so that~\cref{eq:mt mh} becomes
\begin{align}\label{eq:mt mh2}
\begin{split}
\norm{\bm(t_{i+1})&-\bm_h^{i+1}}{\Hone{D}}^2\\
&\le
(1+k)(1+\delta^{-1}k) \norm{\bm(t_i)-\bm_h^{i}}{\Hone{D}}^2
\\
&\quad+
k(1+k)
\Big(
(k+\delta)
\norm{\bm_t(t_i)-\bv_h^{i}}{\Ltwo{D}}^2
+
k
\norm{\nabla\bm_t(t_i)-\nabla\bv_h^{i}}{\Ltwo{D}}^2
\\
&\quad+
2\dual{\nabla\bm(t_i)-\nabla\bm_h^{i}}{\nabla\bm_t(t_i)-\nabla\bv_h^{i}}_D
\Big)
\\
&\quad+
k^3\norm{\bm}{W^{2,\infty}(0,T;\Hone{D})}^2.
\end{split}
\end{align}
Due to the assumption $k<\alpha/(2C_e)$ we can choose $\delta=\alpha/(2C_e)$ 
such that $k+\delta\le\alpha/C_e$ and use~\cref{eq:mt vhi} to deduce
\begin{align*}
\norm{\bm(t_{i+1})&-\bm_h^{i+1}}{\Hone{D}}^2\\
&\le
(1+k)(1+\delta^{-1}k) \norm{\bm(t_i)-\bm_h^{i}}{\Hone{D}}^2
\\
&\quad+
2k C_{\bm}
\big(
h^2 + k^2 +
\norm{\bm(t_i) - \bm_h^i}{\H^1(D)}^2
\big)
+
k^3\norm{\bm}{W^{2,\infty}(0,T;\Hone{D})}^2
\\
&=
\big(1+(1+\delta^{-1}+2C_{\bm})k+\delta^{-1}k^2\big) 
\norm{\bm(t_i)-\bm_h^{i}}{\Hone{D}}^2
\\
&\quad+
2k C_{\bm}
(h^2 + k^2)
+ k^3\norm{\bm}{W^{2,\infty}(0,T;\Hone{D})}^2.
\end{align*}
Applying \Cref{lem:dis Gro} (in the Appendix below) with
$a_i:=\norm{\bm(t_i)-\bm_h^{i}}{\Hone{D}}^2$,
$b_i:= (1+\delta^{-1}+2C_{\bm})k+\delta^{-1}k^2$,
and
$
c_i
:=
k 
\big(
2C_{\bm} (h^2 + k^2)
+
k^2\norm{\bm}{W^{2,\infty}(0,T;\Hone{D})}^2
\big),
$
we deduce
\begin{align*}
\norm{\bm(t_j)&-\bm_h^{j}}{\Hone{D}}^2\\
&\lesssim
e^{t_j}
\Big(\norm{\bm(0)-\bm^0_h}{\Hone{D}}^2
+
t_j
\big(
2C_{\bm} (h^2 + k^2)
+
k^2\norm{\bm}{W^{2,\infty}(0,T;\Hone{D})}^2
\big)
\Big)
\\
&\lesssim 
\norm{\bm^0-\bm^0_h}{\Hone{D}}^2+h^2+k^2,
\end{align*}
proving~\cref{eq:strongconv}. 

To prove~\cref{eq:strongconv2} we first note that
\begin{align}\label{eq:mmhk}
\begin{split}
\norm{\bm&-\bm_{hk}}{L^2(0,T;\Hone{D})}^2\\
&=
\sum_{i=0}^{N-1}
\int_{t_i}^{t_{i+1}}
\Big(
\norm{\frac{t_{i+1}-t}{k}\big(\bm(t)-\bm_{h}^i\big)
\\
&\quad
+ 
\frac{t-t_i}{k}\big(\bm(t)-\bm_h^{i+1}\big)}{\Hone{D}}^2
\Big) \, dt
\\
&\lesssim
\sum_{i=0}^{N-1}
\int_{t_i}^{t_{i+1}}
\Big(
\norm{\bm(t)-\bm_{h}^{i}}{\Hone{D}}^2
+
\norm{\bm(t)-\bm_{h}^{i+1}}{\Hone{D}}^2
\Big)
\, dt
\\
&\lesssim 
\max_{0\leq i\leq N}
\norm{\bm(t_i)-\bm_h^{\revision{i}}}{\Hone{D}}^2
+
k^2 \norm{\bm_t}{L^{\infty}(0,T;\Hone{D})}^2,
\end{split}
\end{align}
where in the last step we used Taylor's Theorem.
The uniqueness of the strong solution~$\bm$ follows
from~\cref{eq:strongconv2} and the fact that the $\bm_{hk}$ are uniquely
determined by Algorithm~\ref{algorithm}.
With the weak convergence proved in \Cref{thm:weakconv},
we obtain that this weak solution coincides with $\bm$.
This concludes the proof.
\end{proof}
\begin{remark}
Since
\begin{align*}
\norm{\bm-\bm_{hk}}{\Hone{D_T}}^2
&=
\int_0^T
\big(
\norm{\bm(t)-\bm_{hk}(t)}{\Hone{D}}^2
+
\norm{\bm_t(t)-\partial_t\bm_{hk}(t)}{\Ltwo{D}}^2
\big)
\, dt
\\
&\lesssim
\sum_{i=0}^{N-1}
\int_{t_i}^{t_{i+1}}
\Big(
\norm{\bm(t)-\bm_{h}^{i}}{\Hone{D}}^2
+
\norm{\bm(t)-\bm_{h}^{i+1}}{\Hone{D}}^2
\\
&\quad
+
\norm{\bm_t(t)-\bv_{h}^{i}}{\Ltwo{D}}^2
\Big)
\, dt
\\
&\lesssim 
\max_{0\leq i\leq N}
\norm{\bm(t_i)-\bm_h^{j}}{\Hone{D}}^2
+
\max_{0\leq i\leq N}
\norm{\bm_t(t_i)-\bv_{h}^{i}}{\Ltwo{D}}^2
\\
&\quad
+
k^2 \norm{\bm_t}{L^{\infty}(0,T;\Hone{D})}^2
+
k^2 \norm{\bm_{tt}}{L^{\infty}(0,T;\Ltwo{D})}^2,
\end{align*}
by using~\cref{eq:mt vhi} for the second term on the right-hand side
and using Young's inequality~$2ab\le ka^2 + k^{-1}b^2$ for the
inner product in~\cref{eq:mt vhi},
we obtain a weaker convergence in the $\Hone{D_T}$-norm, namely
\begin{align*}
\norm{\bm-\bm_{hk}}{\H^1(D_T)}\leq 
C_{\rm conv}
k^{-1/2}
\big(
\norm{\bm^0-\bm^0_h}{\Hone{D}}+h+k\big),
\end{align*}
provided that $hk^{-1/2}\to0$ when $h,k\to0$.
\end{remark}

\subsection{Proof of \Cref{thm:strongconvELLG}}\label{section:strong2}
This section bootstraps the results of the previous section to
include the full ELLG system into the analysis. 

\begin{proof}
Similarly to the proof of \Cref{thm:strongconvLLG}, we
derive~\cref{eq:mt mh2} with $\delta = \alpha/(4C_e)$.
Multiplying~\cref{eq:Eip} by $\beta = \alpha/(4C_eC_{\bH})$ and adding the
resulting equation to~\cref{eq:mt mh2} yields
\begin{align*}
\norm{&\bm(t_{i+1})-\bm_h^{i+1}}{\Hone{D}}^2
+
\beta\norm{\bE_{i+1}}{h}^2
+
\beta \frac{k}{2} \norm{\nabla\times\be_{i+1}}{\Ltwo{D}}^2
\nonumber\\
&\leq 
(1+k)(1+\delta^{-1}k)
\norm{\bm(t_i)-\bm_h^{i}}{\Hone{D}}^2
\nonumber\\
&\quad
+
k(1+k)
\Big(
(k+\delta+\beta C_{\bH})
\norm{\bm_t(t_i)-\bv_h^{i}}{\Ltwo{D}}^2
+
k
\norm{\nabla\bm_t(t_i)-\nabla\bv_h^{i}}{\Ltwo{D}}^2
\nonumber\\
&\quad+
2\dual{\nabla\bm(t_i)-\nabla\bm_h^{i}}{\nabla\bm_t(t_i)-\nabla\bv_h^{i}}_D
\Big)
+
k^3\norm{\bm}{W^{2,\infty}(0,T;\Hone{D})}^2
\nonumber\\
&\quad+
(1+2k)\beta\norm{\bE_{i}}{h}^2
+
\beta C_{\bH}k(h^2+k^2).
\end{align*}
The assumption~$k\le\alpha/(2C_e)$, see
\Cref{thm:strongconvELLG}, implies
$k+\delta+\beta C_{\bH}\le\alpha/C_e$. 
By invoking \Cref{lemma:cea0} 
we infer
\begin{align}\label{eq:mti mhi}
\norm{\bm(t_{i+1})&-\bm_h^{i+1}}{\Hone{D}}^2
+
\beta\norm{\bE_{i+1}}{h}^2
+
\beta \frac{k}{2} \norm{\nabla\times\be_{i+1}}{\Ltwo{D}}^2
\nonumber\\
&\le
(1+k)(1+\delta^{-1}k) \norm{\bm(t_i)-\bm_h^{i}}{\Hone{D}}^2
\nonumber\\
&\quad+
k(1+k) C_{\bH}
\Big(
h^2 + k^2 +
\norm{\bm(t_i) - \bm_h^i}{\H^1(D)}^2
+
\norm{\bH(t_i)-\bH_h^i}{\Ltwo{D}}^2
\Big)
\nonumber\\
&\quad
+
k^3\norm{\bm}{W^{2,\infty}(0,T;\Hone{D})}^2
+
(1+2k)\beta\norm{\bE_{i}}{h}^2
+
\beta C_{\bH}k(h^2+k^2)
\nonumber\\
&=
\big(1+(1+\delta^{-1}+C_{\bH})k+(\delta^{-1}+C_{\bH})k^2\big) 
\norm{\bm(t_i)-\bm_h^{i}}{\Hone{D}}^2
\nonumber\\
&\quad+
k(1+k)C_{\bH}
\norm{\bH(t_i)-\bH_h^i}{\Ltwo{D}}^2
+
(1+2k)\beta\norm{\bE_{i}}{h}^2
\nonumber\\
&\quad
+ 
k^3\norm{\bm}{W^{2,\infty}(0,T;\Hone{D})}^2
+
k C_{\bH} (1 + k + \beta)(h^2 + k^2).
\end{align}
The approximation properties of $\Pi_{\NN\DD}$ and the regularity of $\bH$
imply 
\[
\norm{\bH(t_i)-\bH_h^i}{\Ltwo{D}}^2
\lesssim 
\norm{\bE_i}{\Ltwo{D}}^2 + h^2,
\]
where the hidden constant depends only on the shape regularity of
$\TT_h$ and on the regularity of $\bH$. Hence, we obtain
from~\cref{eq:mti mhi}
\begin{align*}
\norm{\bm(t_{i+1})&-\bm_h^{i+1}}{\Hone{D}}^2
+
\beta\norm{\bE_{i+1}}{h}^2
+
\beta \frac{k}{2} \norm{\nabla\times\be_{i+1}}{\Ltwo{D}}^2
\\
&\leq 
(1+C_{\rm comb}k)
\norm{\bm(t_i)-\bm_h^{i}}{\Hone{D}}^2
+
(1+Ck)\beta\norm{\bE_{i}}{h}^2
\\
&\quad
+ 
k^3\norm{\bm}{W^{2,\infty}(0,T;\Hone{D})}^2
+
k C (h^2 + k^2),
\end{align*}
where $C_{\rm comb} := 1+2\delta^{-1}+2C_{\bH}$ and
for some constant $C>0$ which is independent of $k,h$ and $i$. 
Hence, we find a constant $\widetilde C_{\rm comb}>0$ such that
\begin{align}\label{eq:mEe}
\norm{\bm(t_{i+1})&-\bm_h^{i+1}}{\Hone{D}}^2
+
\beta\norm{\bE_{i+1}}{h}^2
+
\beta \frac{k}{2} \norm{\nabla\times\be_{i+1}}{\Ltwo{D}}^2
\nonumber
\\
&\leq 
(1+\widetilde C_{\rm comb}k)
\big(
\norm{\bm(t_i)-\bm_h^{i}}{\Hone{D}}^2
+
\beta\norm{\bE_{i}}{h}^2
\big)
+
k\widetilde C_{\rm comb}
\big(h^2+k^2\big).
\end{align}
Applying \Cref{lem:dis Gro} (in the Appendix below) with
$a_i:=\norm{\bm(t_{i+1})-\bm_h^{i+1}}{\Hone{D}}^2
+
\beta\norm{\bE_{i+1}}{h}^2
+
\beta \frac{k}{2} \norm{\nabla\times\be_{i+1}}{\Ltwo{D}}^2$, 
$b_i=\widetilde C_{\rm comb}k$, and 
$c_i=k\widetilde C_{\rm comb}(h^2+k^2)$ 
we deduce, for all $i=0,\ldots,N$,
\begin{align}\label{eq:almostconvest1}
\begin{split}
\norm{\bm(t_{i+1})&-\bm_h^{i+1}}{\Hone{D}}^2
+
\norm{\bE_{i+1}}{h}^2
+
k\norm{\nabla\times\be_{i+1}}{\Ltwo{D}}^2
\\
&\lesssim 
\norm{\bm^0-\bm_h^0}{\Hone{D}}^2
+
\norm{\bE_0}{h}^2
+
k\norm{\nabla\times \be_0}{\Ltwo{D}}^2
+
C_{\rm comb}(h^2+k^2).
\end{split}
\end{align}
Since
\begin{align}\label{eq:eh}
\begin{split}
\big|
\norm{\nabla\times(\bH(t_i)-\bH^i_h)}{\Ltwo{D}}^2
-
\norm{\nabla\times\be_i}{\Ltwo{D}}^2
\big|
&\lesssim h^2,
\\
\big|
\norm{(\bH(t_i)-\bH^i_h,\lambda(t_i)-\lambda^i_h)}{\XX}^2
-
\norm{\bE_i}{\XX}^2
\big|
&\lesssim h^2,
\end{split}
\end{align}
(which is a result of the approximation properties of $\Pi_{\NN\DD}$ and 
$\Pi_{\SS}$ and the regularity assumptions on $\bH$ and
$\lambda$)
estimate~\cref{eq:convest1} follows immediately.

%The proof of~\cref{eq:convest2} follows in the same manner as that
%of~\cref{eq:strongconv2} by using the same argument as that used
%to obtain~\cref{eq:mmhk}. 

To prove~\cref{eq:convest2} it suffices to estimate the term with $k$ factor on the
left-hand side of that inequality
because the other terms can be estimated in exactly the same
manner as in the proof of \Cref{thm:strongconvLLG}. By using Taylor's
Theorem and~\cref{eq:eh} we deduce
\begin{align}\label{eq:HHhk}
\norm{\nabla\times(\bH-\bH_{hk})}{\Ltwo{D_T}}^2
&\lesssim
\sum_{i=1}^{N}\Big(
k\norm{\nabla\times(\bH(t_{i+1})-\bH_h^{i+1})}{\Ltwo{D}}^2
+
C_{\bH}
k^3\Big)
\nonumber
\\
&\lesssim
\sum_{i=1}^{N} 
k\norm{\nabla\times \be_{i+1}}{\Ltwo{D}}^2+h^2 + k^2,
\end{align}
where $
C_{\bH}
:=
\norm{\bH}{W^{1,\infty}(0,T;\Hcurl{D})}^2 
$.
On the other hand, it follows from~\cref{eq:mEe} that
\begin{align*}
k&\norm{\nabla\times \be_{i+1}}{\Ltwo{D}}^2\\
&\lesssim
\Big(
\norm{\bm(t_{i})-\bm_h^{i}}{\Hone{D}}^2
-
\norm{\bm(t_{i+1})-\bm_h^{i+1}}{\Hone{D}}^2
\Big)
+
k
\norm{\bm(t_{i})-\bm_h^{i}}{\Hone{D}}^2
\\
&\quad
+
\beta
\Big(
\norm{\bE_i}{h}^2
-
\norm{\bE_{i+1}}{h}^2
\Big)
+
k
\norm{\bE_i}{h}^2
+
k(h^2 + k^2),
\end{align*}
which then implies by using telescoping series,~\cref{eq:almostconvest1}, and~\cref{eq:eh}
\begin{align*}
\sum_{i=1}^N
k\norm{\nabla\times \be_{i+1}}{\Ltwo{D}}^2
&\lesssim
\norm{\bm^00-\bm_h^0}{\Hone{D}}^2
+
\norm{\bE_0}{h}^2
\\
&\quad
+
\max_{0\le i\le N}
\Big(
\norm{\bm(t_{i})-\bm_h^{i}}{\Hone{D}}^2
+
\norm{\bE_i}{h}^2
\Big)
+
h^2 + k^2
\\
&\lesssim
\norm{\bm^0-\bm_h^0}{\Hone{D}}^2
+
\norm{\bH^0-\bH_h^0}{\Ltwo{D}}^2
+
\norm{\lambda^0-\lambda_h^0}{H^{1/2}(\Gamma)}^2
\\
&\quad
+
\max_{0\le i\le N}
\Big(
\norm{\bm(t_{i})-\bm_h^{i}}{\Hone{D}}^2
+
\norm{\bH(t_i)-\bH_h^i}{\Ltwo{D}}^2
\\
&\qquad\qquad\qquad+
\norm{\lambda(t_i)-\lambda_h^i}{H^{1/2}(\Gamma)}^2
\Big)
+
h^2 + k^2.
\end{align*}
The required result now follows from~\cref{eq:HHhk} and~\cref{eq:convest1}.
Uniqueness is also obtained as in the
proof of \Cref{thm:strongconvLLG}, completing the proof the
theorem.
\end{proof}

\section{Numerical experiments}\label{section:numerics}
The following numerical experiments are carried out by use of the
FEM toolbox FEniCS~\cite{fenics} (\texttt{fenicsproject.org})
and the BEM toolbox BEM++~\cite{bempp} (\texttt{bempp.org}).
We use GMRES to solve the linear systems and blockwise diagonal
scaling as preconditioners.

The values of the constants in these examples are taken from the
standard problem \#1 proposed by the Micromagnetic Modelling Activity
Group at the National Institute of Standards and
Technology~\cite{mumag}.
As domain serves the unit cube $D=[0,1]^3$ with initial conditions
\begin{align*}
 \bm^0(x_1,x_2,x_3):=\begin{cases} (0,0,-1)&\text{for } d(x)\geq 1/4,\\
                      (2Ax_1,2Ax_2,A^2-d(x))/(A^2+d(x))&\text{for }d(x)<1/4,
                     \end{cases}
\end{align*}
where $d(x):= |x_1-0.5|^2+|x_2-0.5|^2$ and $A:=(1-2\sqrt{d(x)})^4/4$ and
\begin{align*}
 \bH^0= \begin{cases} (0,0,3)&\text{in } D,\\
         (0,0,3)-\bm^0&\text{in } D^\ast.
        \end{cases}
\end{align*}
We choose the constants
\begin{align*}
 \alpha=0.5,\quad \sigma=\begin{cases}1&\text{in }D,\\ 0& \text{in }D^\ast,\end{cases}\quad \mu_0=1.25667\times 10^{-6},\quad C_e=\frac{2.6\times 10^{-11}}{\mu_0 \,6.4\times 10^{11}}.
\end{align*}
\subsection{Example 1}
For time and space discretisation of $D_T:= [0,5]\times D$, we apply a
uniform partition in space ($h=0.1$) and time ($k=0.002$).
\Cref{fig:en} plots the corresponding energies over time.
\Cref{fig:m} shows a series of magnetisations $\bm(t_i)$ at certain times $t_i\in[0,5]$. 
\Cref{fig:h} shows that same for the magnetic field $\bH(t_i)$.
\begin{figure}
\centering
\psfrag{energy}{\tiny energy}
\psfrag{time}{\tiny time $t$}
\psfrag{menergy}{\tiny $\norm{\nabla\bm_{hk}(t)}{\Ltwo{D}}$ }
\psfrag{henergy}{\tiny $\norm{\bH_{hk}(t)}{\Hcurl{D}}$}
\psfrag{sum}{\tiny $\norm{\nabla\bm_{hk}(t)}{\Ltwo{D}}+\norm{\bH_{hk}(t)}{\Hcurl{D}}$}
\includegraphics[width=0.6\textwidth]{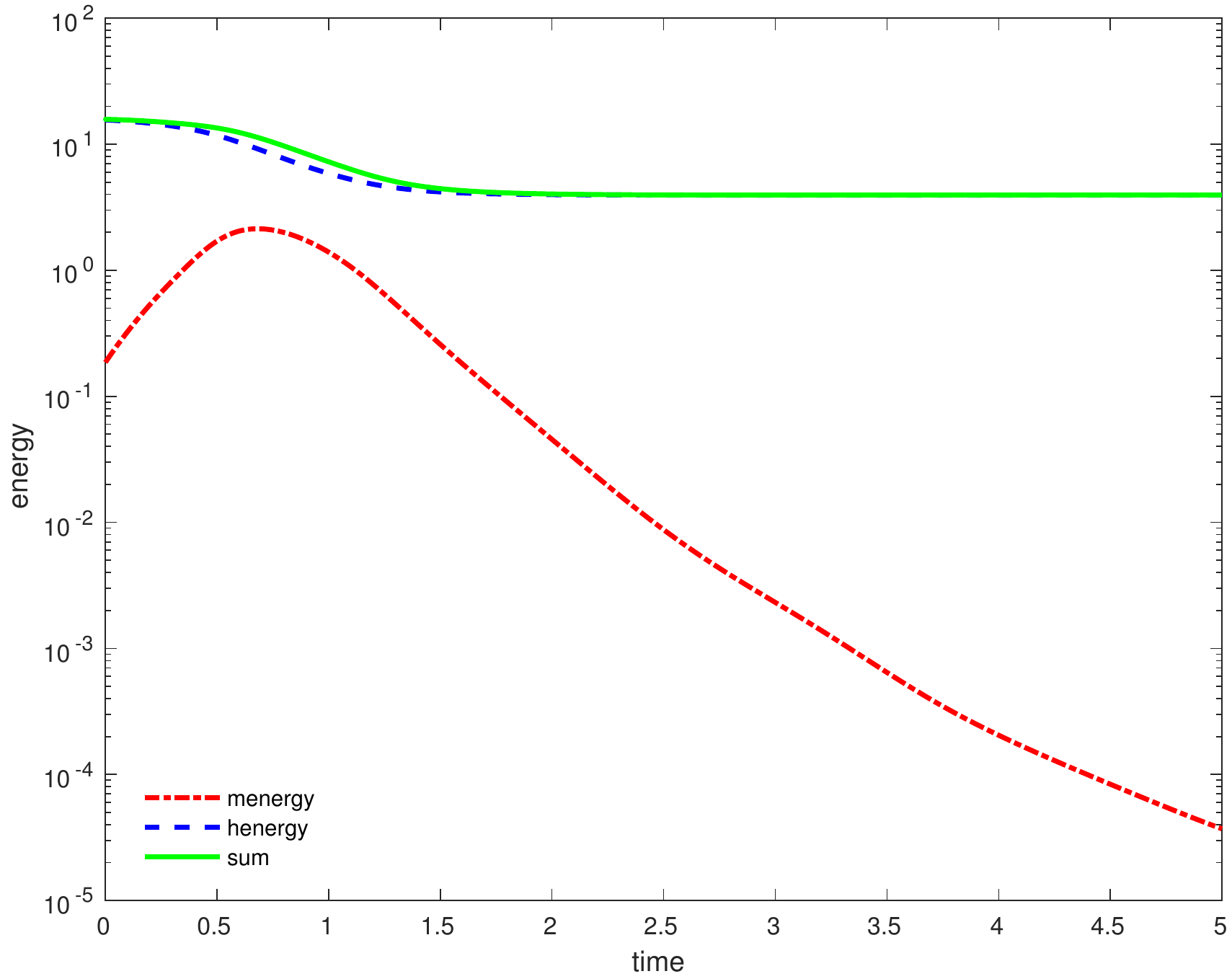}
\caption{\comment{Plot of
$\norm{\nabla\bm_{hk}(t)}{\Ltwo{D}}$ and $\norm{\bH_{hk}(t)}{\Hcurl{D}}$ over the time.}}
\label{fig:en}
\end{figure}

\begin{figure}
\centering
\includegraphics[width=0.24\textwidth]{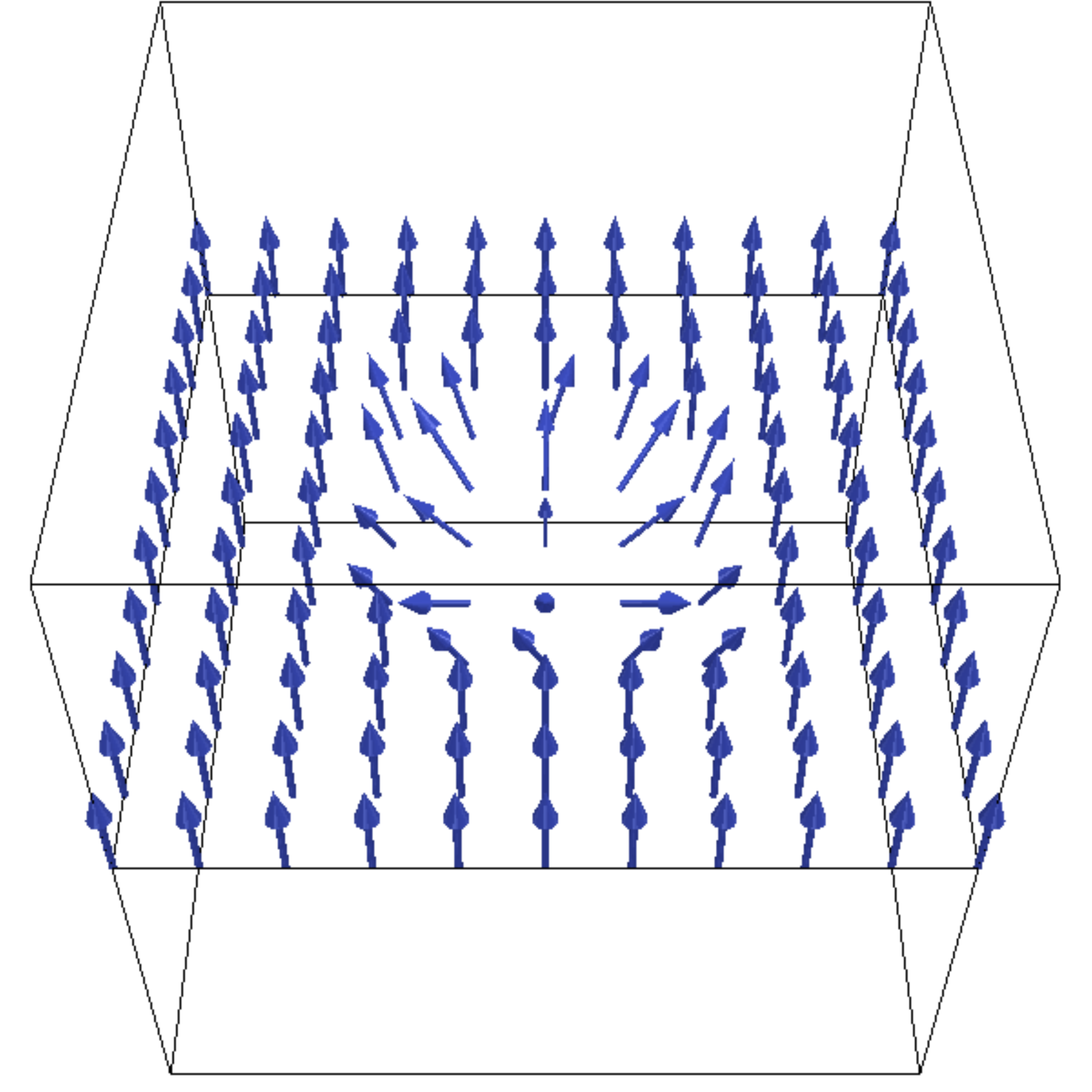}
\includegraphics[width=0.24\textwidth]{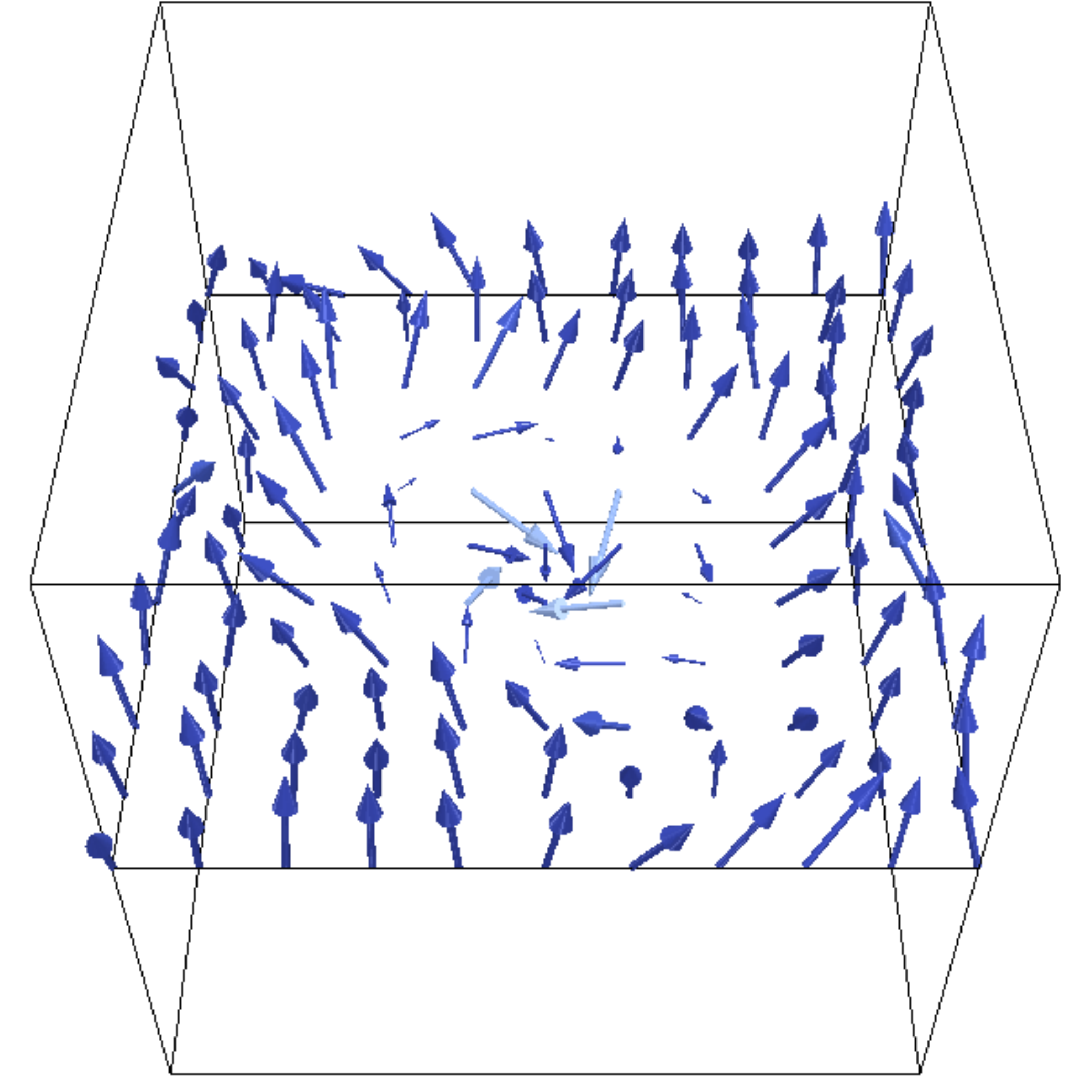}
\includegraphics[width=0.24\textwidth]{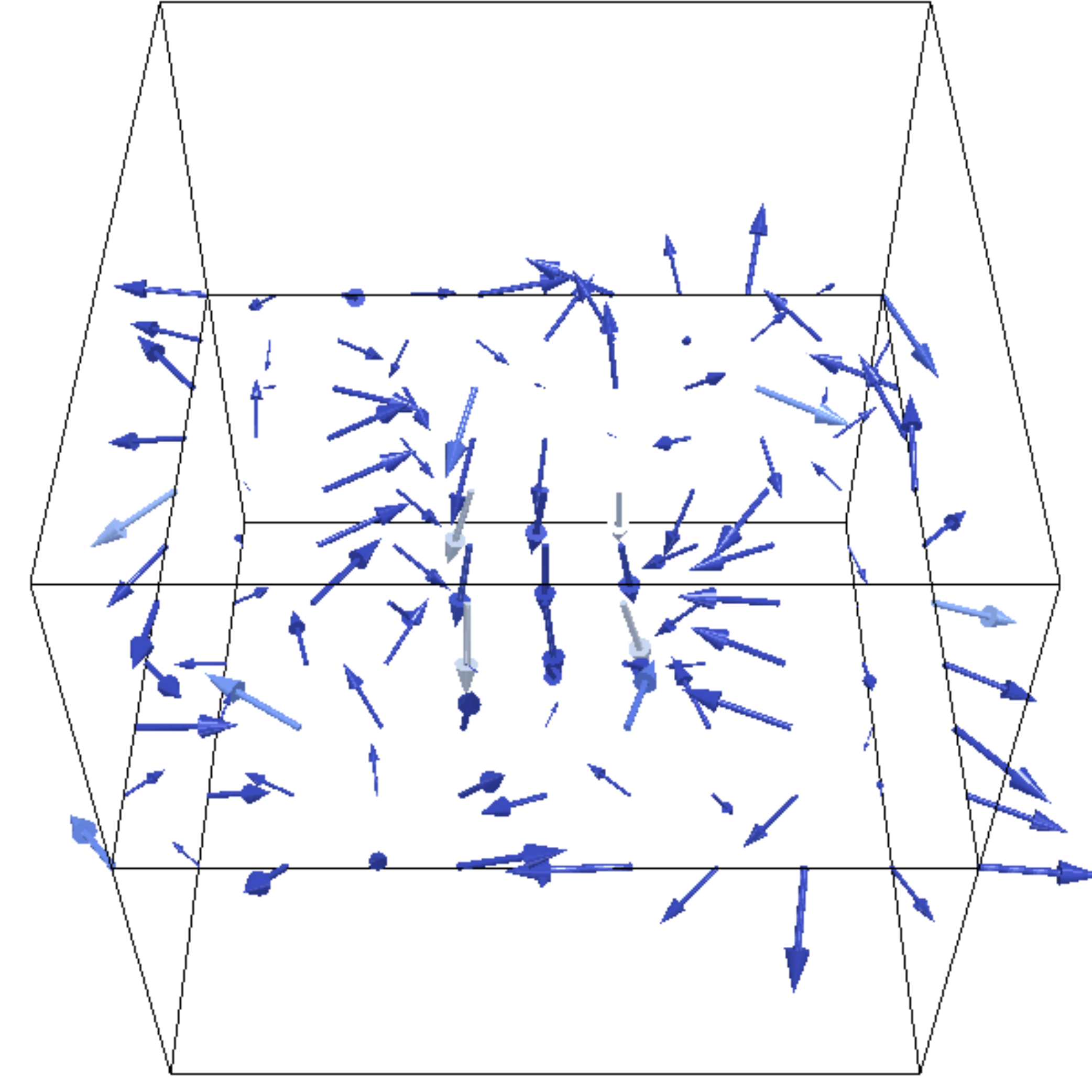}
\includegraphics[width=0.24\textwidth]{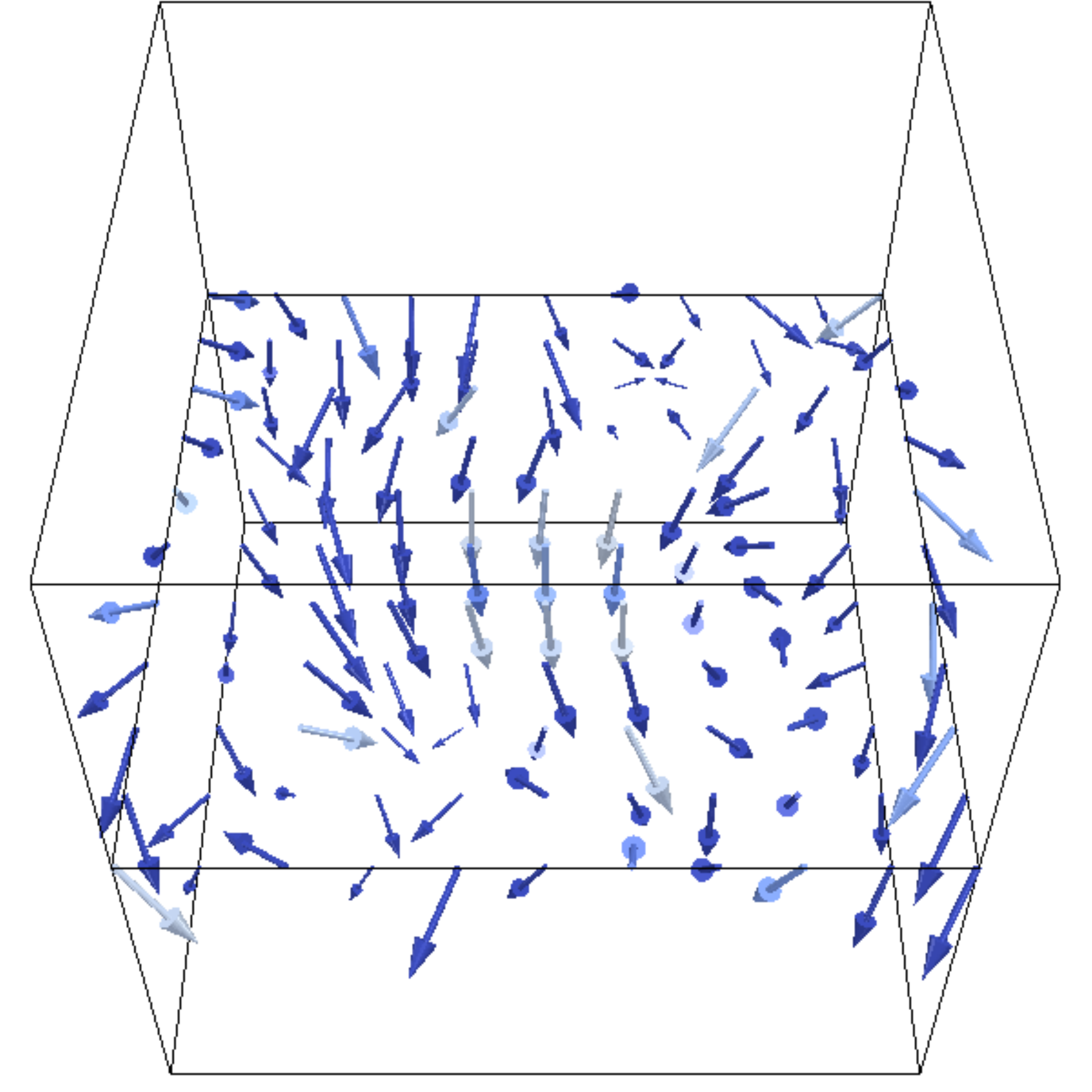}
\includegraphics[width=0.24\textwidth]{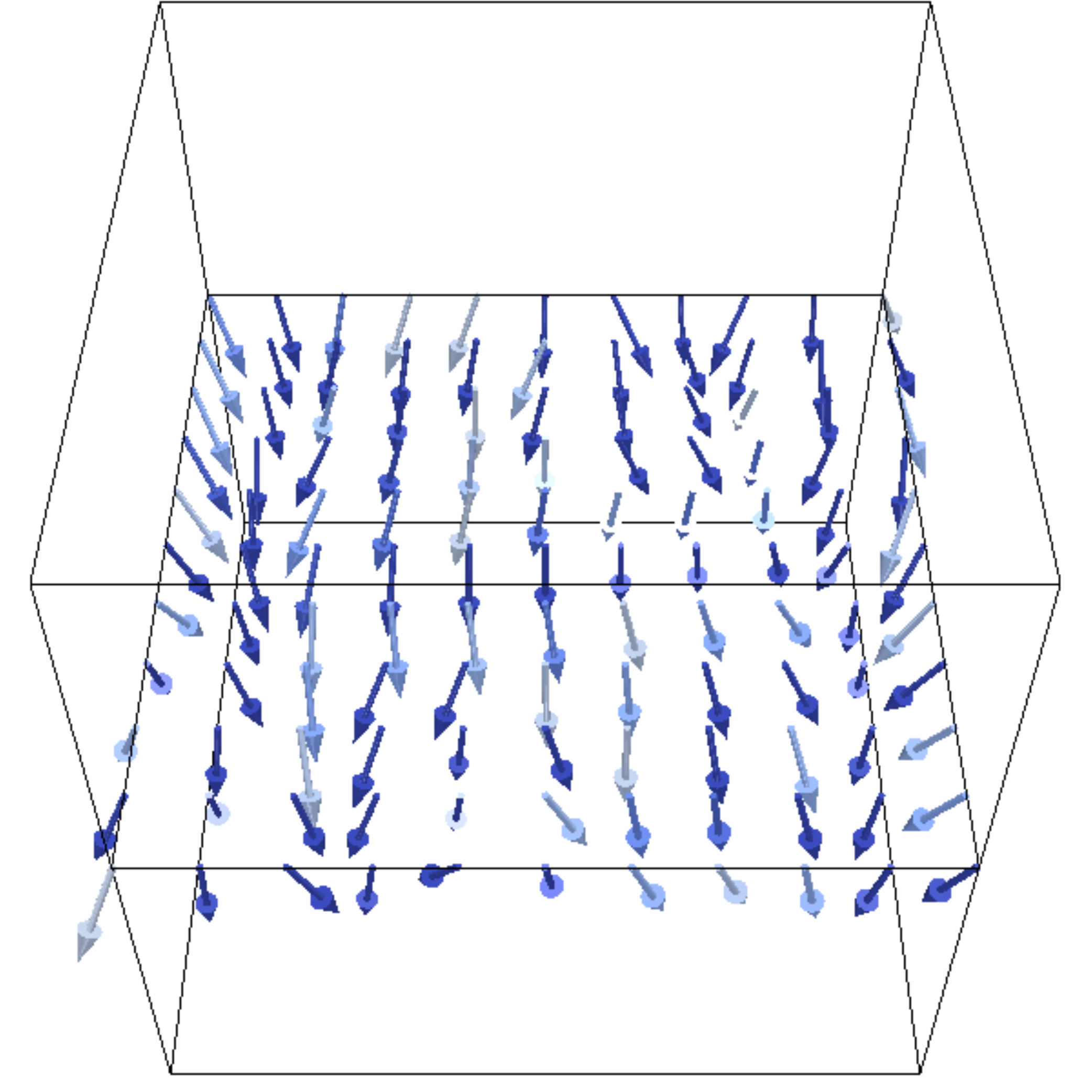}
\includegraphics[width=0.24\textwidth]{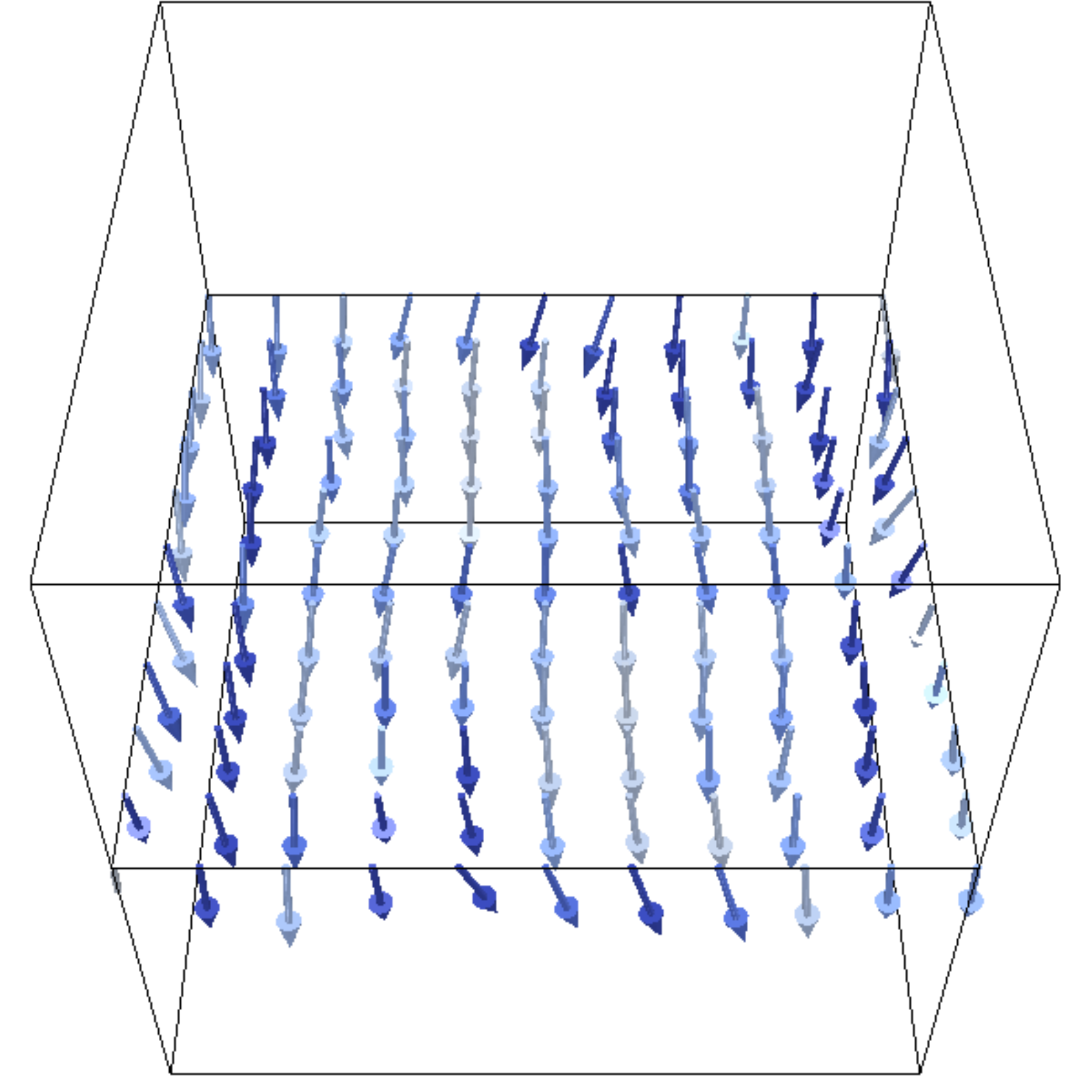}
\includegraphics[width=0.24\textwidth]{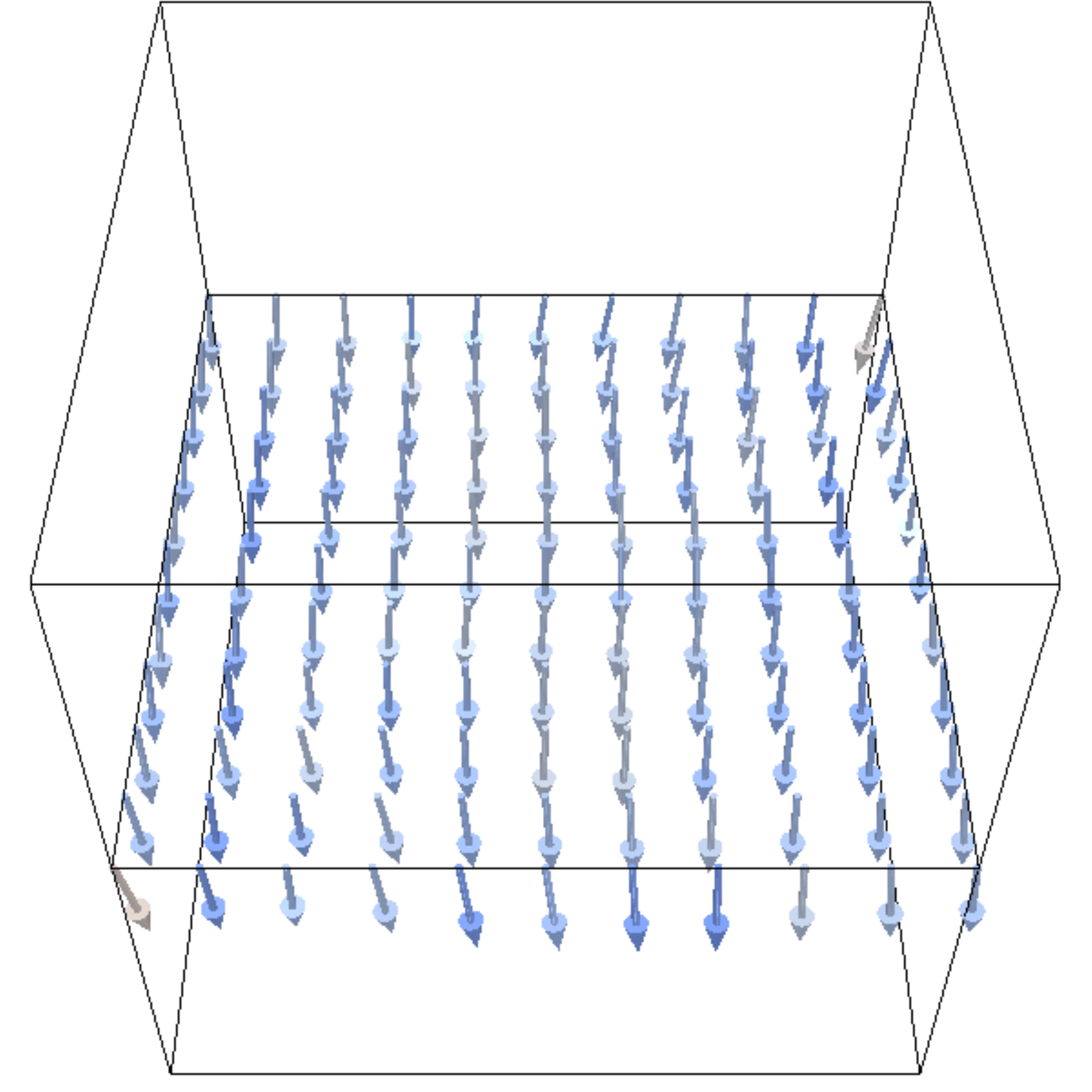}
\includegraphics[width=0.24\textwidth]{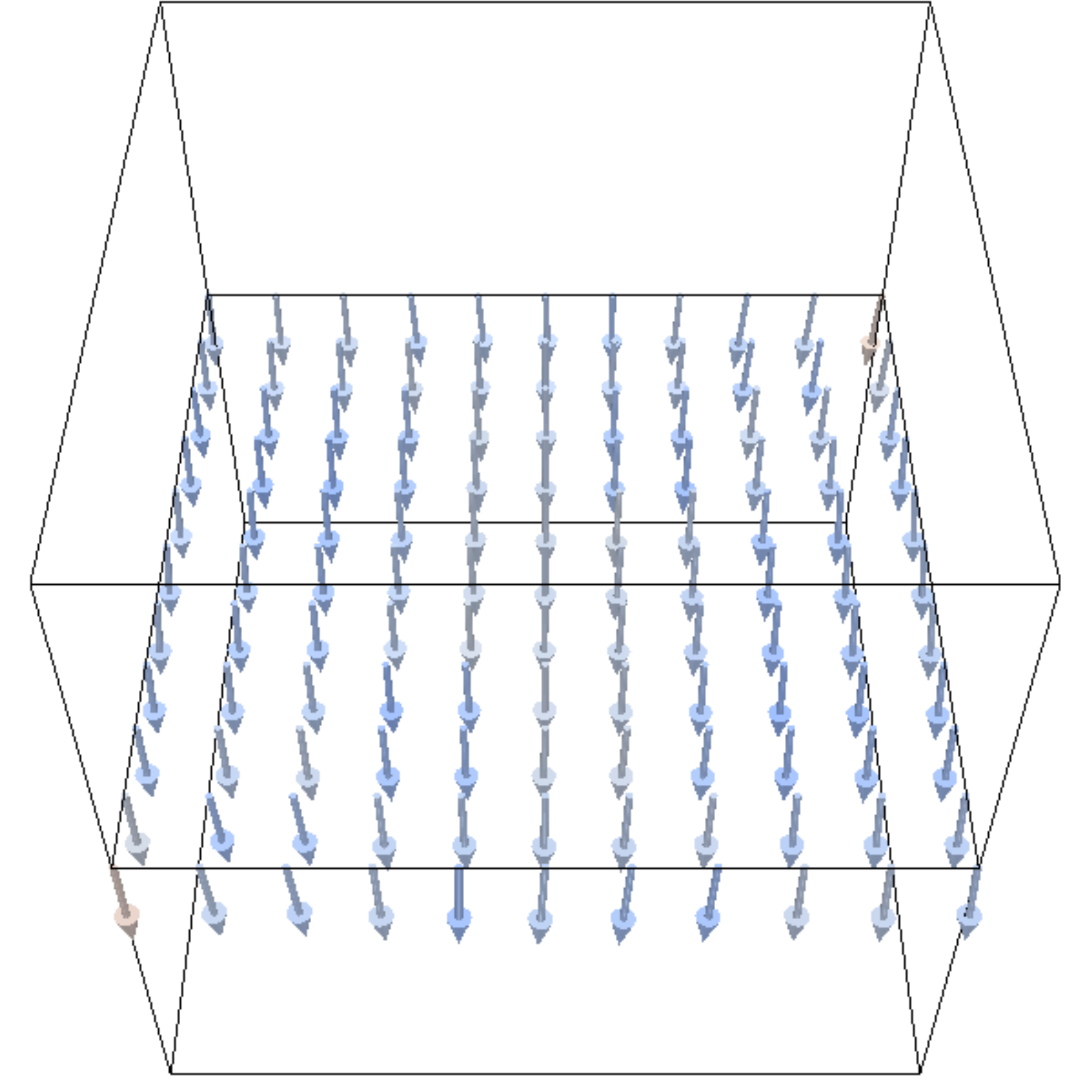}
\includegraphics[width=0.24\textwidth]{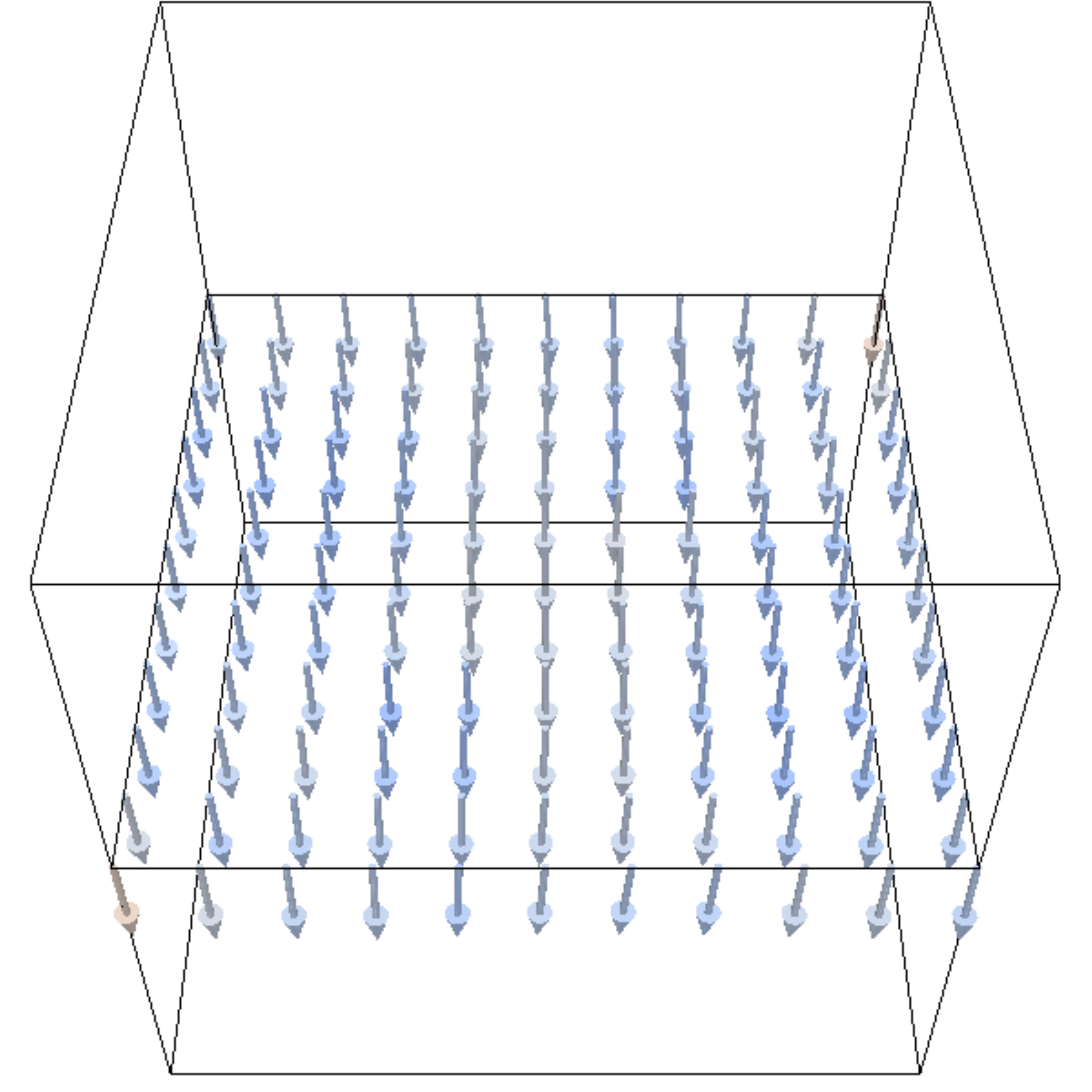}
\includegraphics[width=0.24\textwidth]{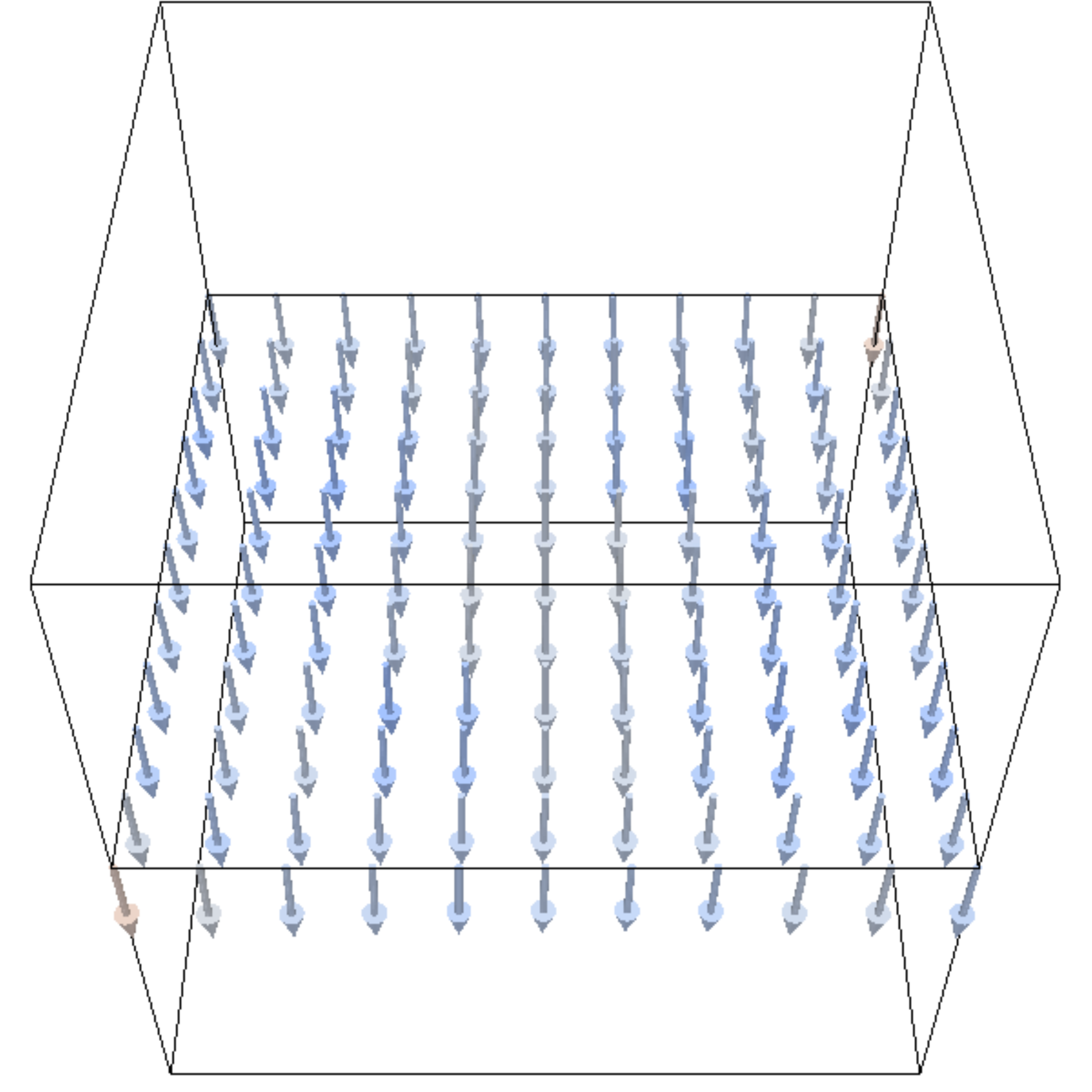}
\includegraphics[width=0.24\textwidth]{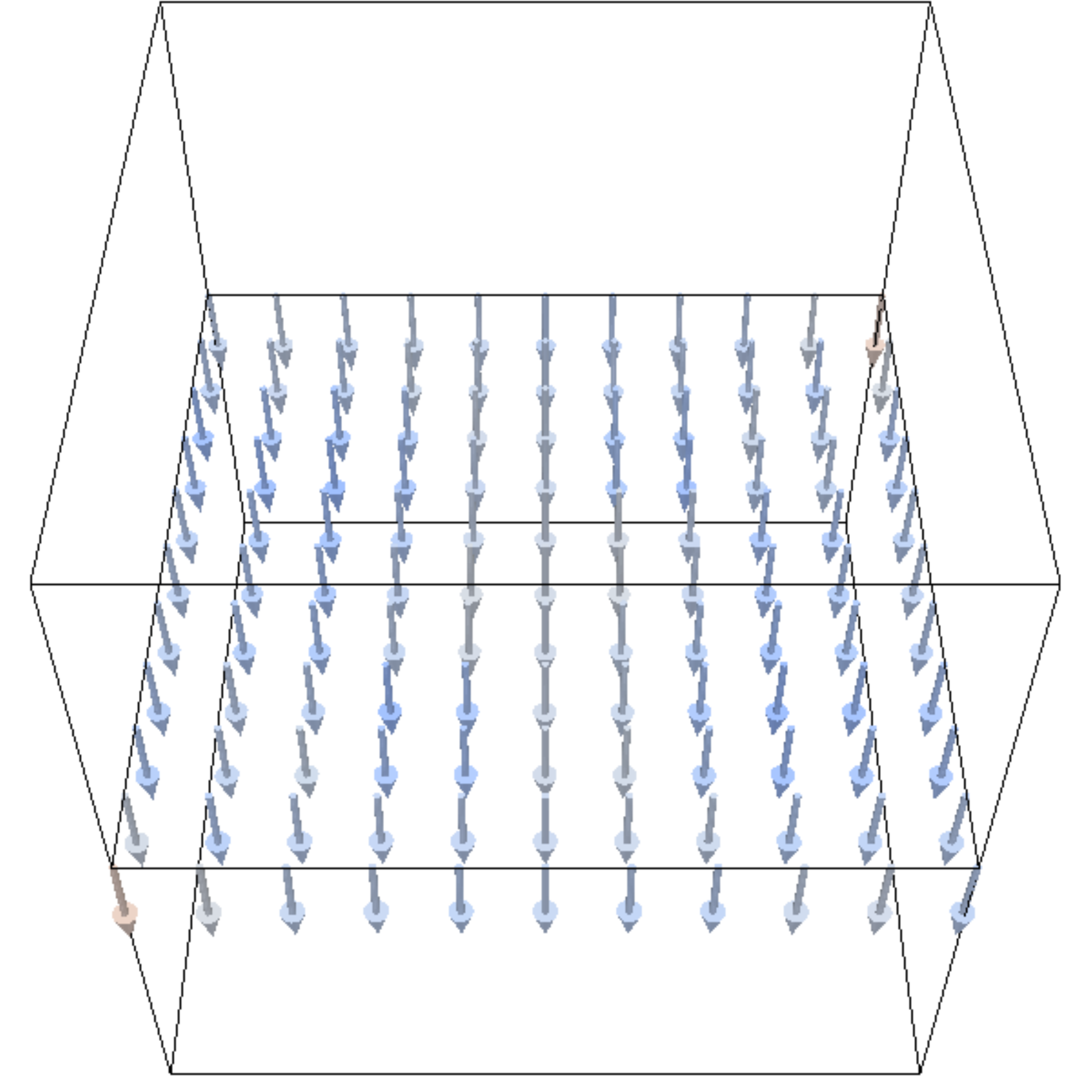}
\includegraphics[width=0.24\textwidth]{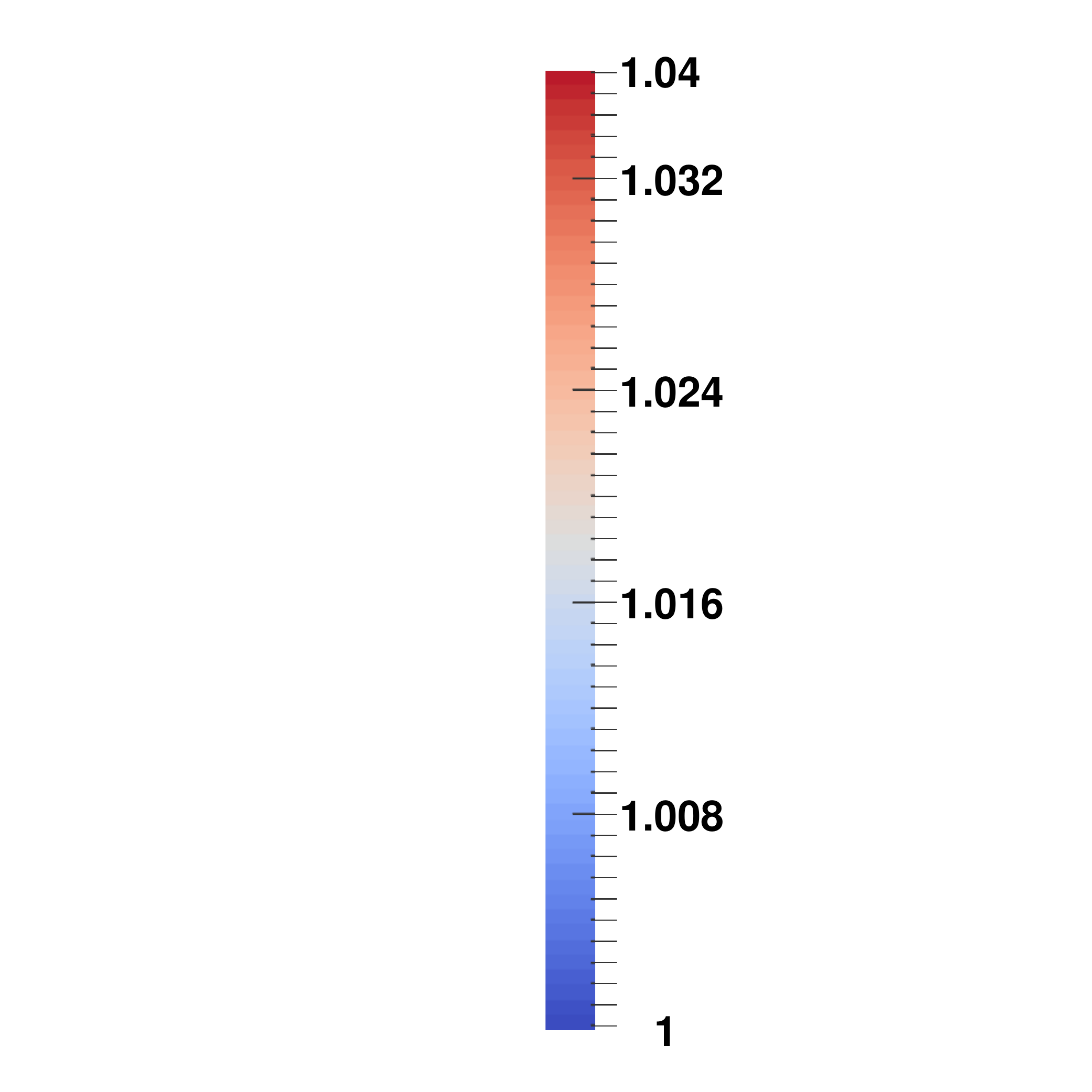}
\caption{Slice of the magnetisation $\bm_{hk}(t_i)$ at $[0,1]^2\times \{1/2\}$ for $i=0,\ldots,10$ with $t_i=0.2i$. The colour of the vectors represents the magnitude $|\bm_{hk}|$.
We observe that the magnetisation aligns itself with the initial magnetic field $\bH^0$ by performing a damped precession.}
\label{fig:m}
\end{figure}

\begin{figure}
\centering
\includegraphics[width=0.24\textwidth]{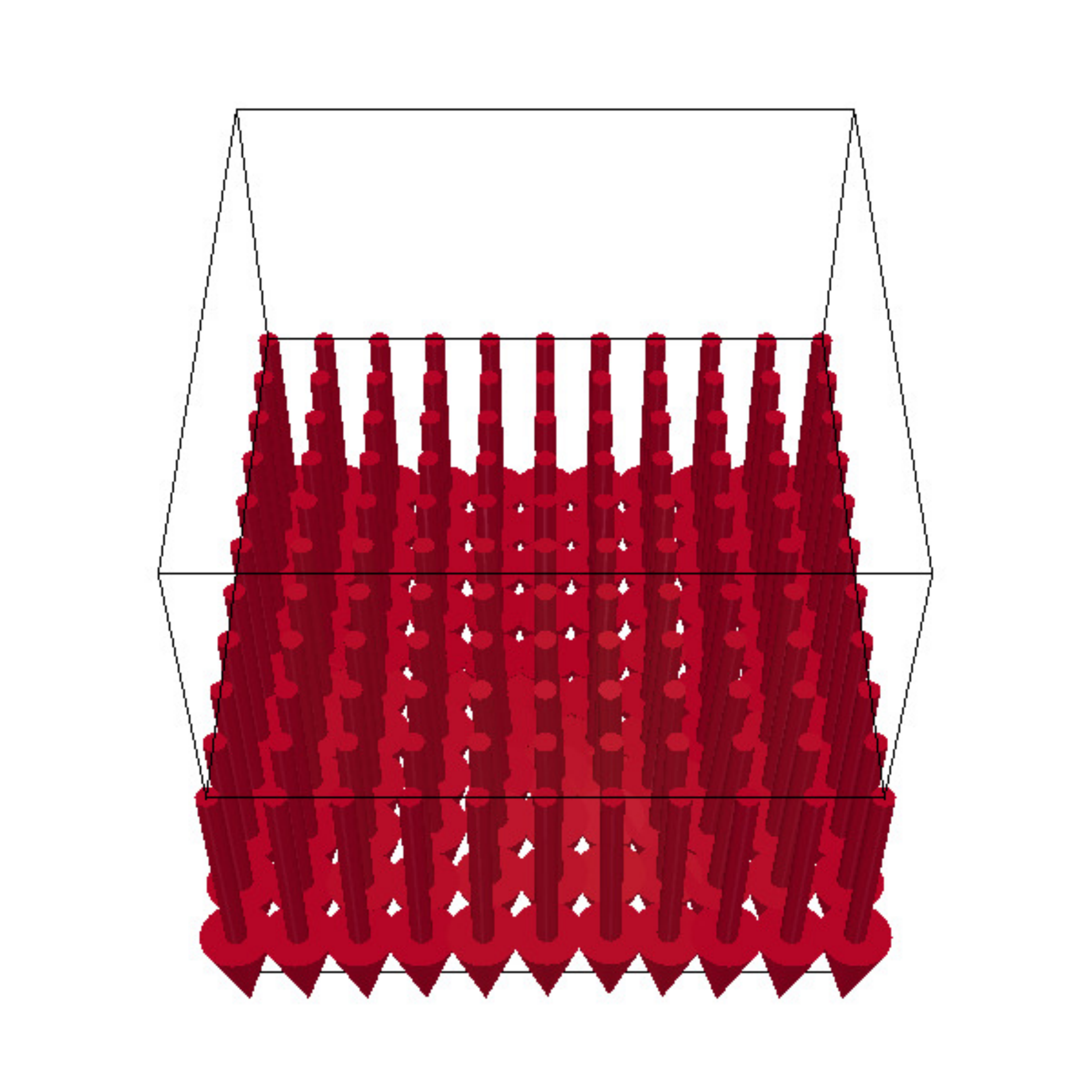}
\includegraphics[width=0.24\textwidth]{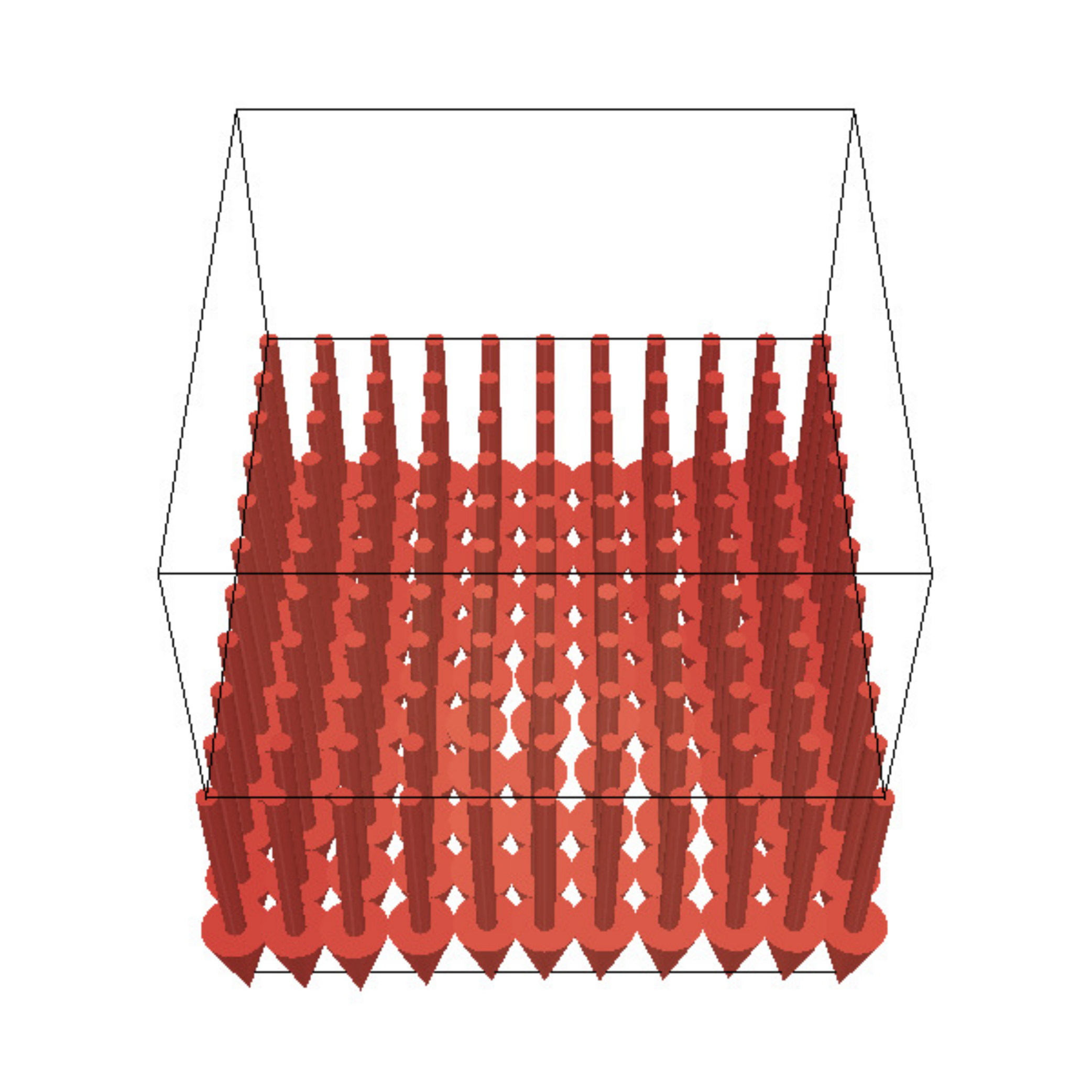}
\includegraphics[width=0.24\textwidth]{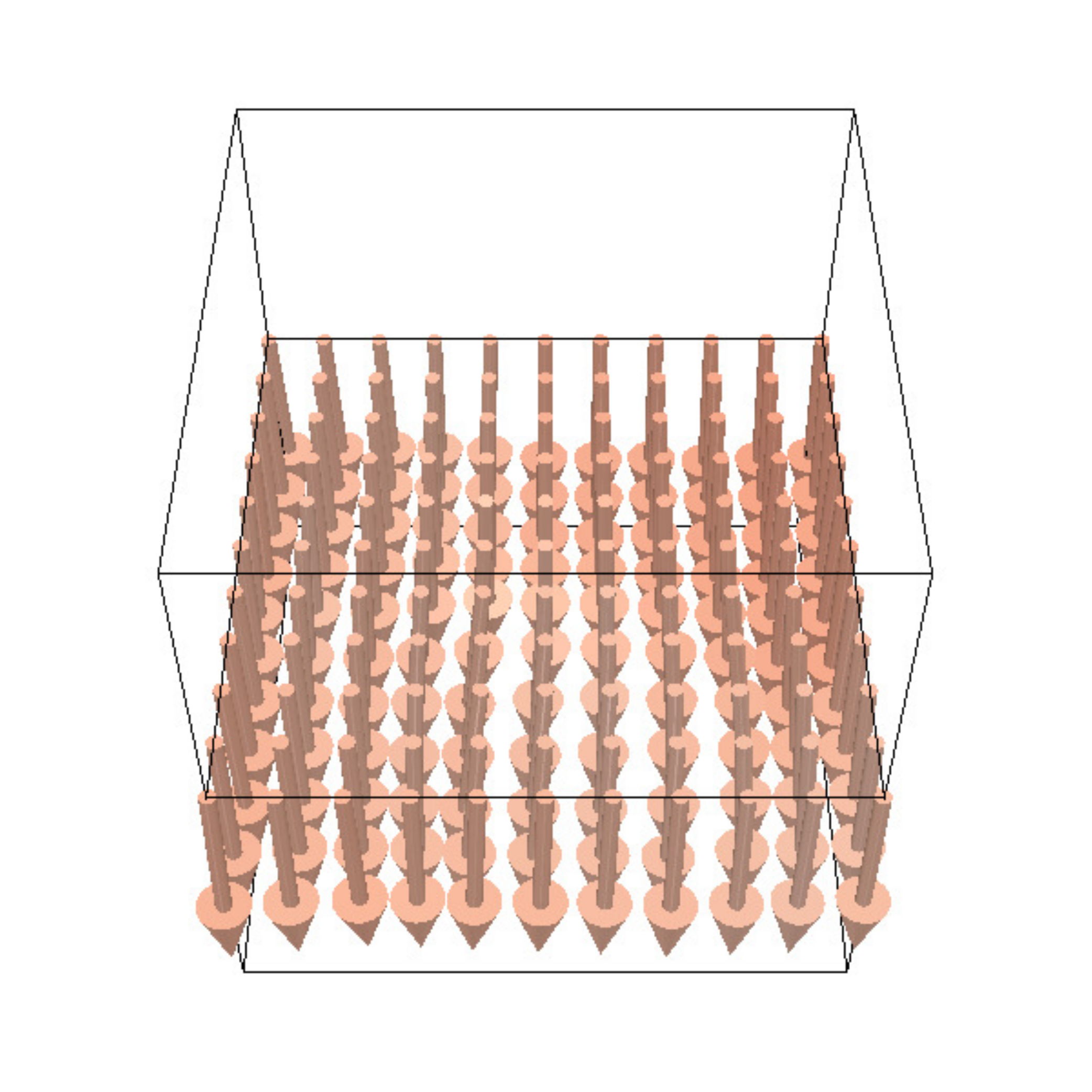}
\includegraphics[width=0.24\textwidth]{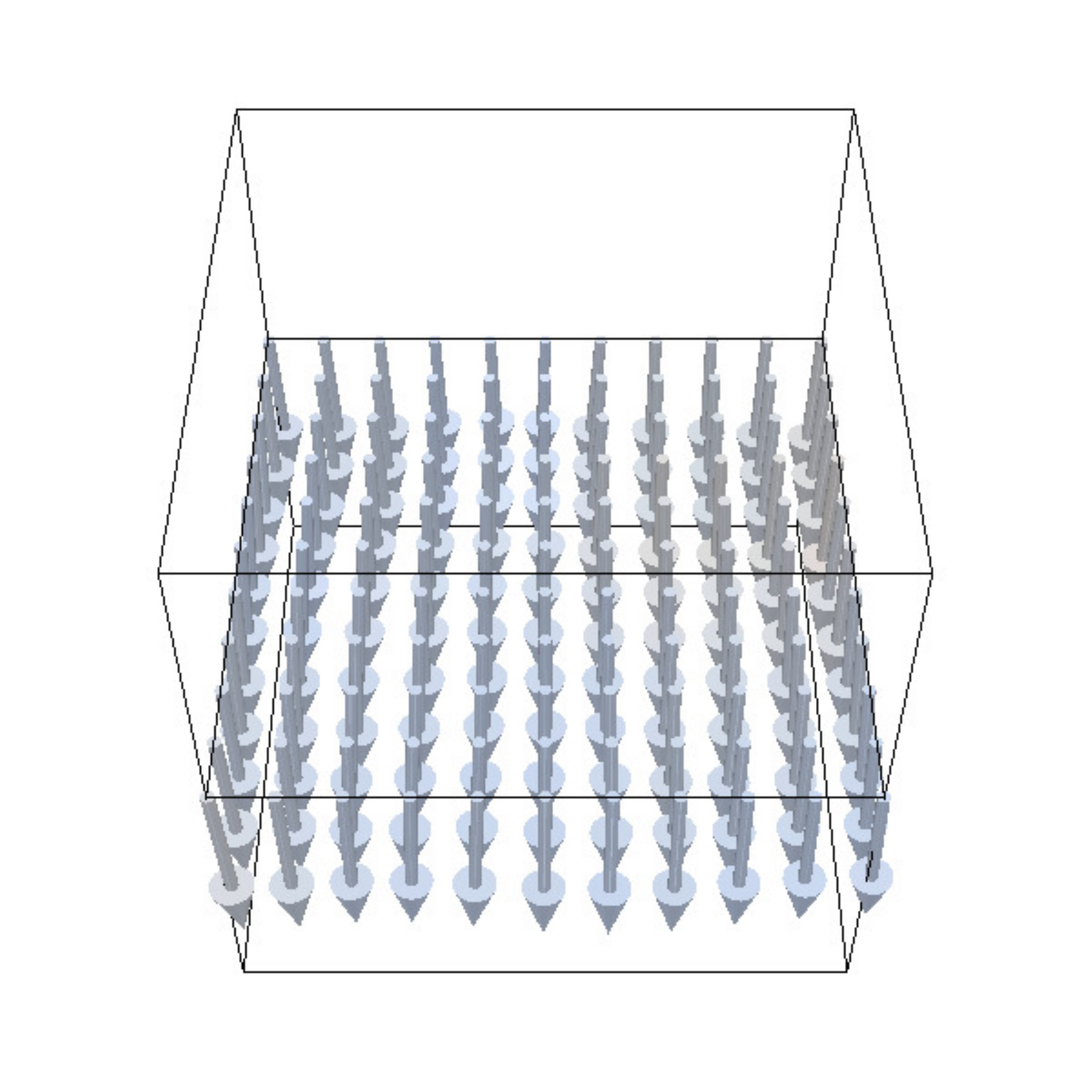}
\includegraphics[width=0.24\textwidth]{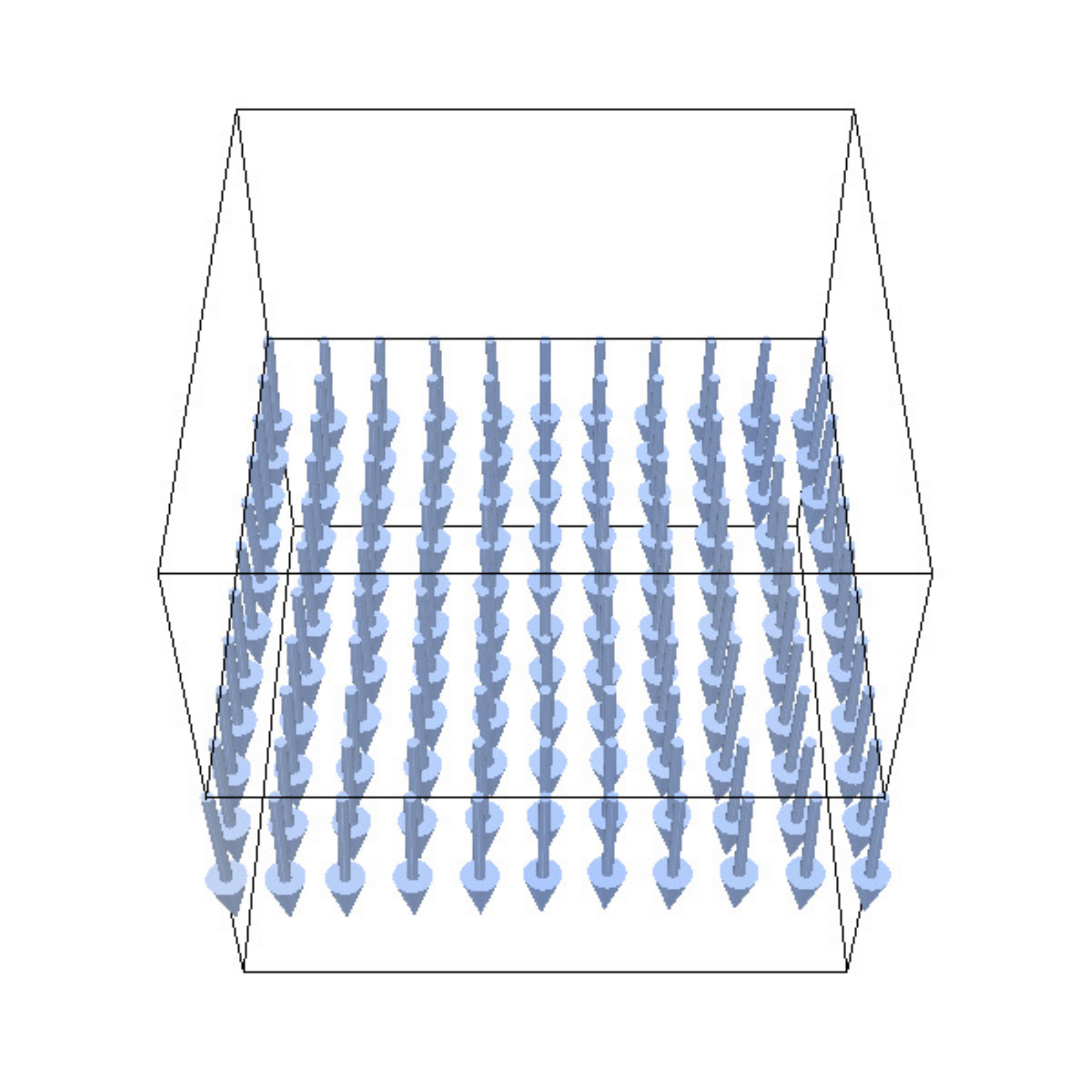}
\includegraphics[width=0.24\textwidth]{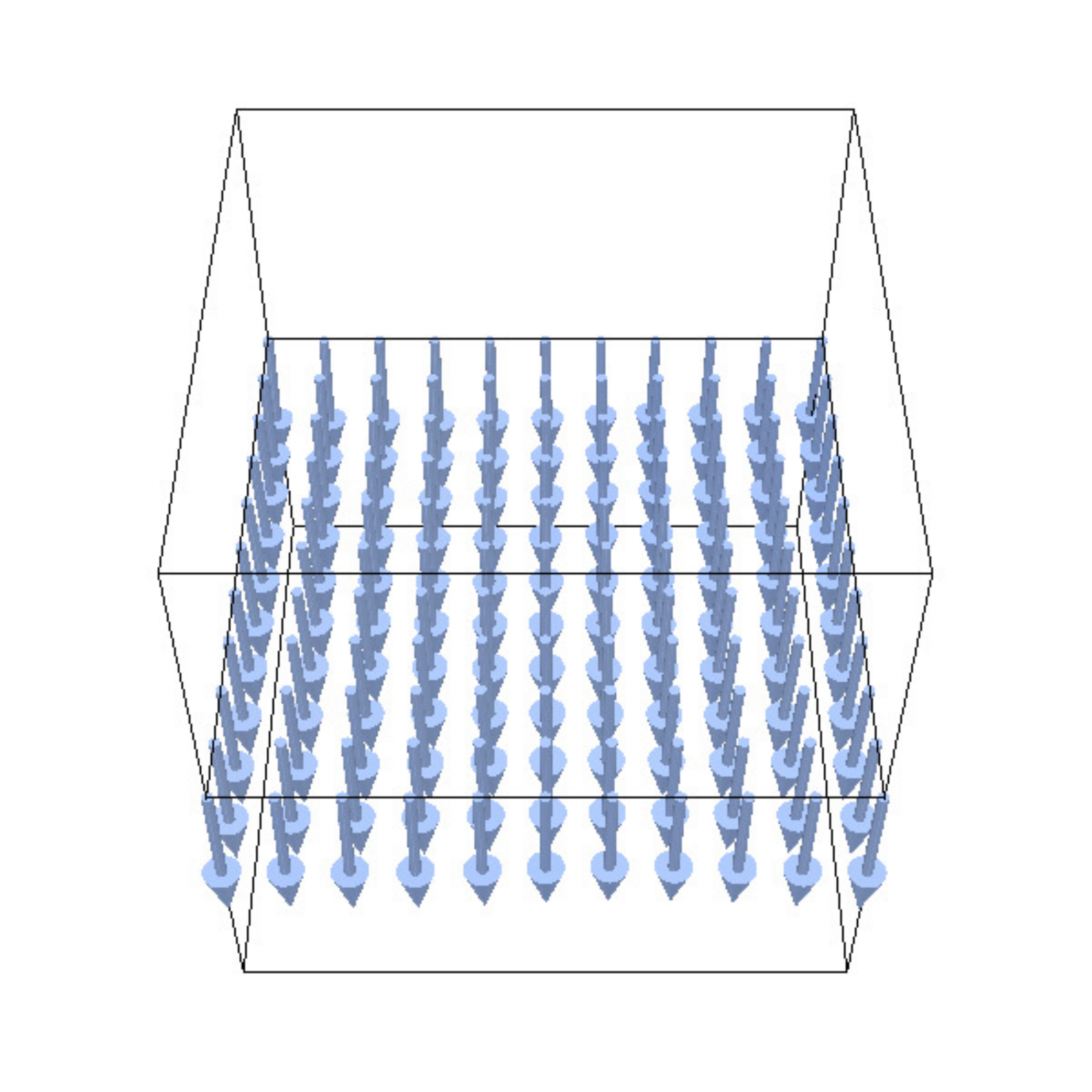}
\includegraphics[width=0.24\textwidth]{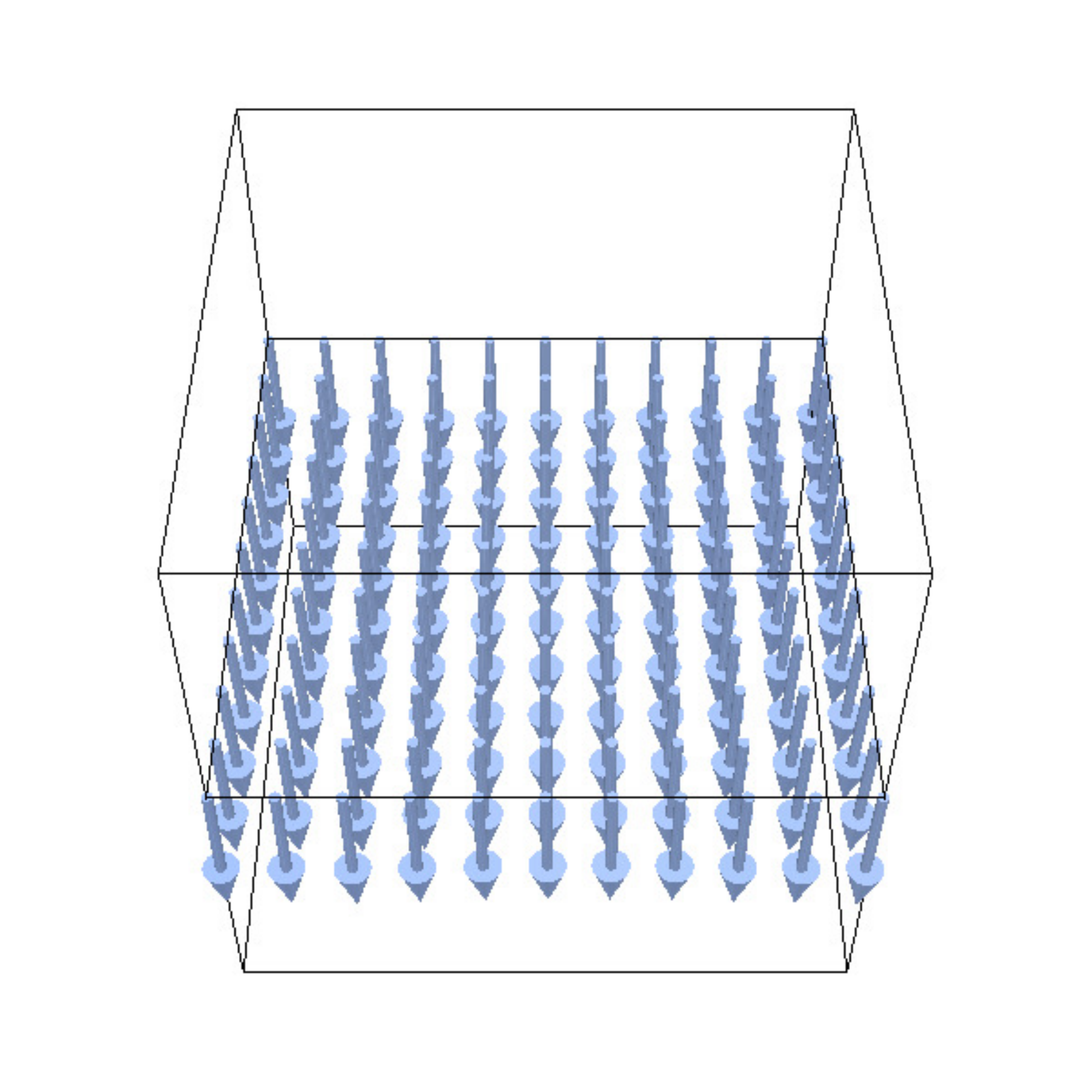}
\includegraphics[width=0.24\textwidth]{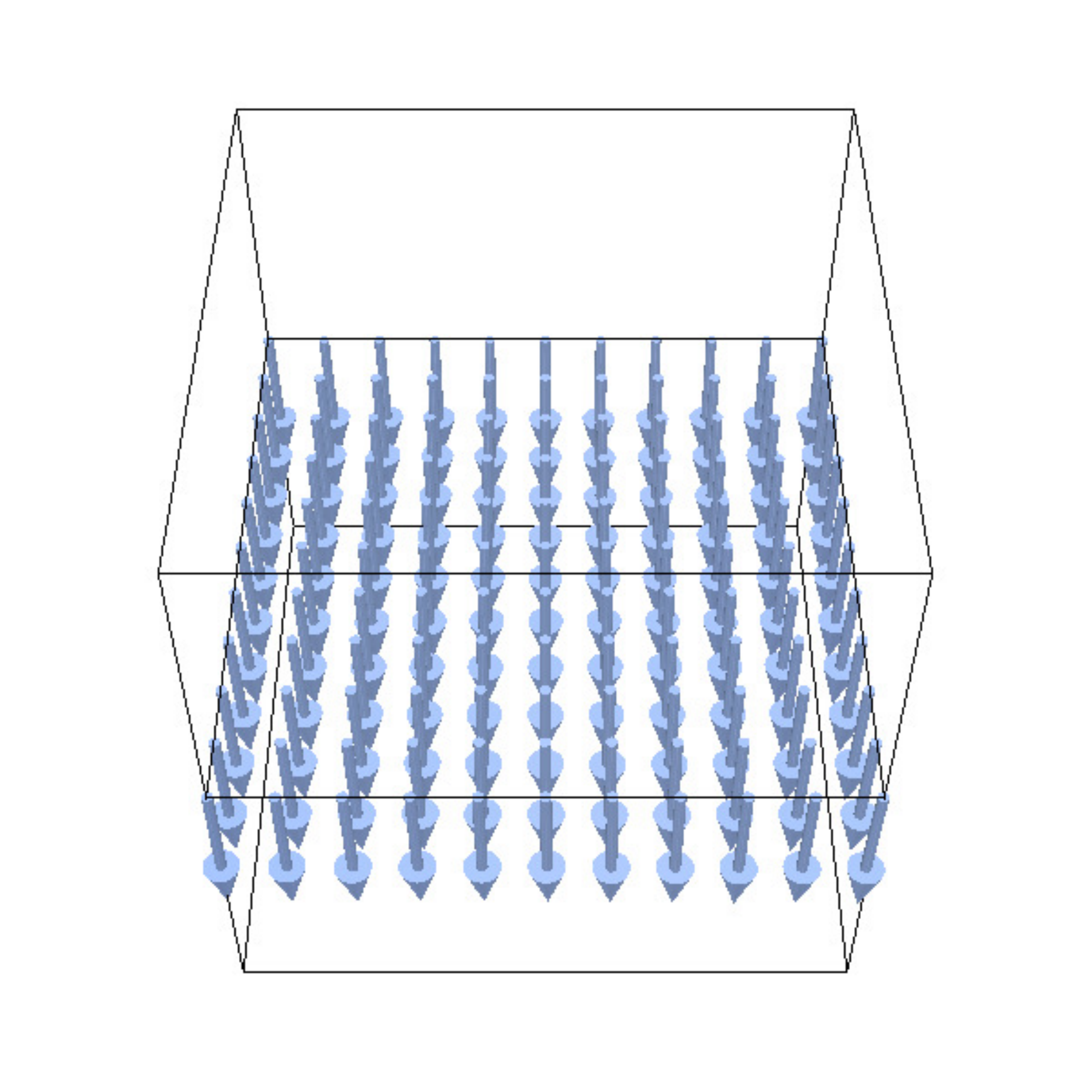}
\includegraphics[width=0.24\textwidth]{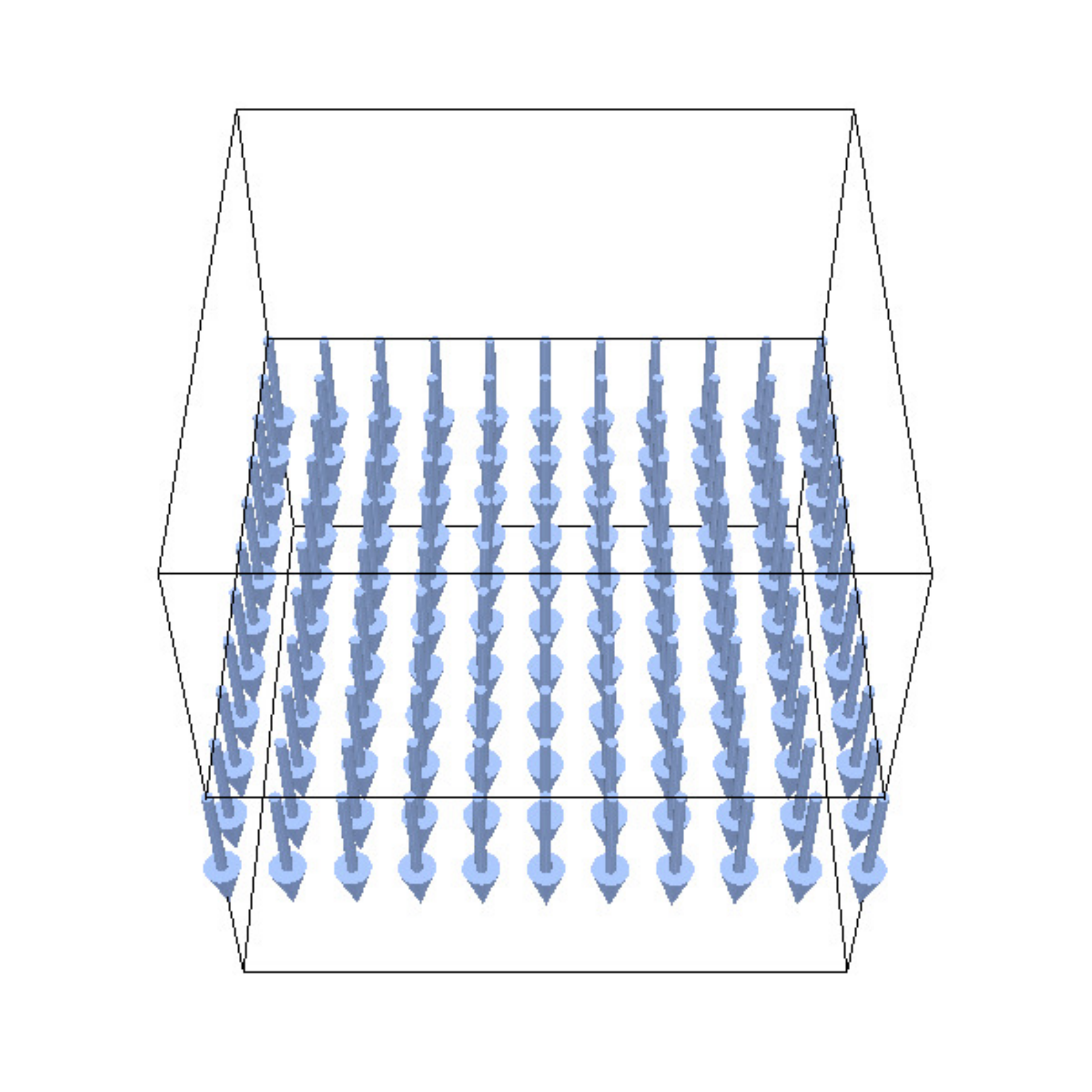}
\includegraphics[width=0.24\textwidth]{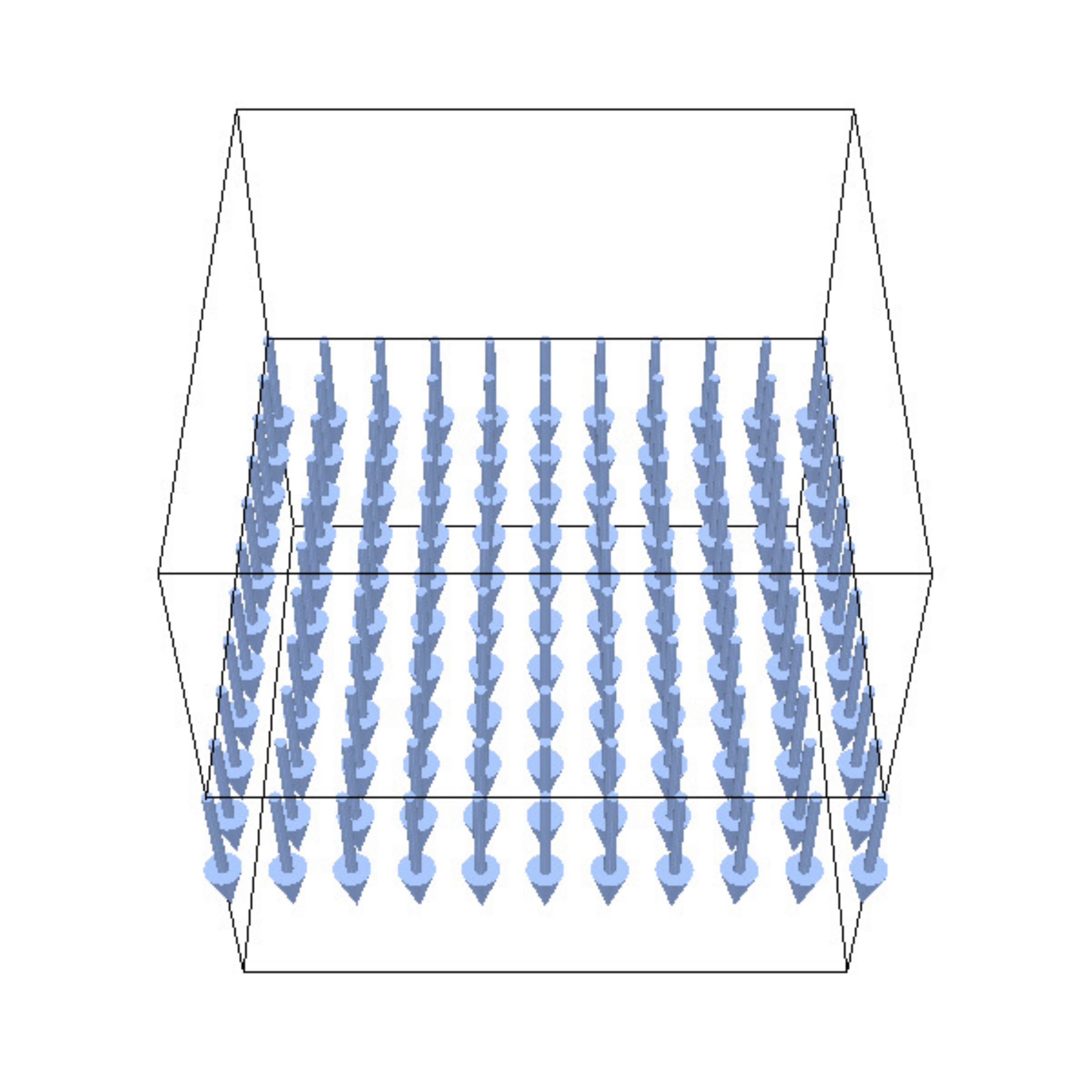}
\includegraphics[width=0.24\textwidth]{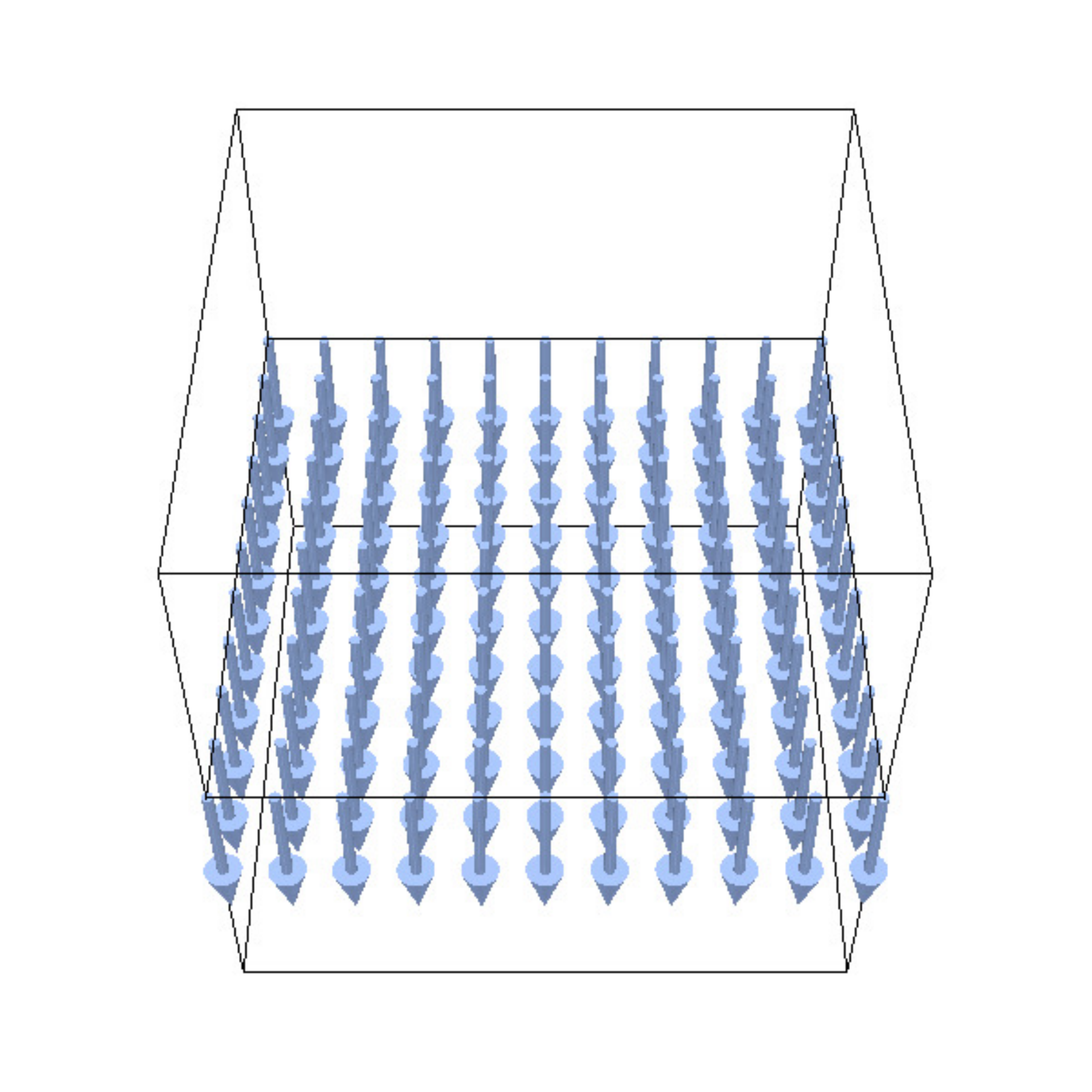}
\includegraphics[width=0.24\textwidth]{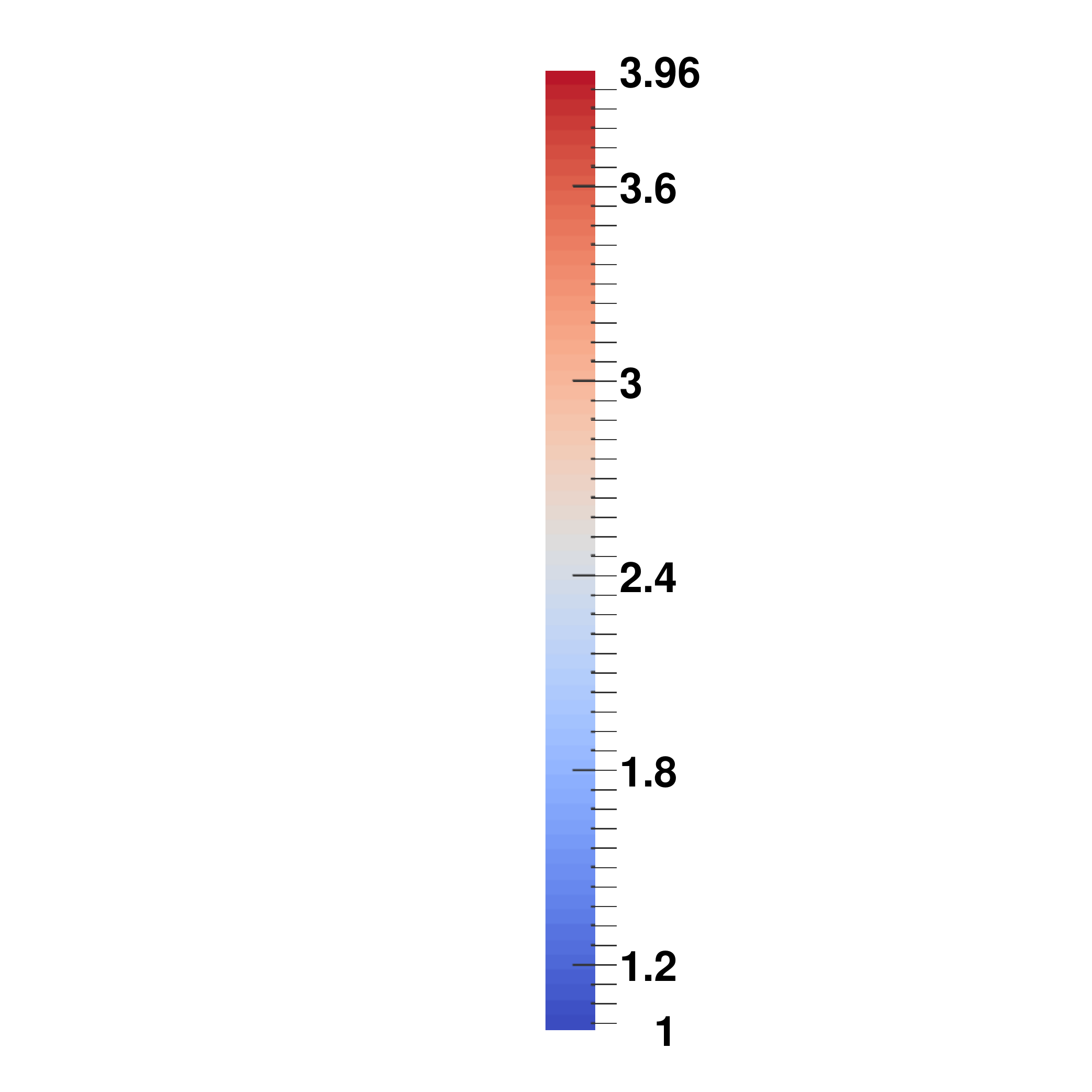}
\caption{Slice of the magnetic field $\bH_{hk}(t_i)$ at $[0,1]^2\times \{1/2\}$ for $i=0,\ldots,10$ with $t_i=0.2i$. The colour of the vectors represents the magnitude $|\bH_{hk}|$.
We observe only a slight movement in the middle of the cube combined with an overall reduction of field strength.}
\label{fig:h}
\end{figure}

\subsection{Example 2}
We use uniform time and space discretisation of the domain $D_T:= [0,0.1]\times D$ to partition $[0,0.1]$ into $1/k$ time intervals for $k\in\{0.001$, $0.002$, $0.004$, $0.008$, $0.016\}$
and $D$ into $\mathcal{O}(N^3)$ tetrahedra for $N\in \{5,10,15,20,25\}$.
\Cref{fig:MHNconv} shows convergence rates with respect to the space discretisation and \Cref{fig:MHkconv} with respect to the time discretization.
Since the exact solution is unknown, we use the finest computed approximation as a reference solution. The convergence plots reveal that the space discretization error dominates the time discretization error by far.
\comment{
The expected convergence rate~$\mathcal{O}(k)$ can be observed in
\Cref{fig:MHkconv} which underlines the theoretical results of
\Cref{thm:strongconvELLG}. It is less clear in \Cref{fig:MHNconv} if there is a
convergence of order~$\mathcal{O}(h)$. Preconditioners, a topic of further study, 
are required for implementation with larger values of~$N$. 
}

%However, we observe the expected convergence rates $\mathcal{O}(h)$ and $\mathcal{O}(k)$ for the computed approximations. This underlines the theoretical results of \Cref{thm:strongconvELLG}.

\begin{figure}
\centering
\psfrag{error}[cc][cc]{\tiny error}
\psfrag{N}[cc][cc]{\tiny number of elements per axis $\simeq 1/h$}
\psfrag{0.001}{\tiny $k=0.001$}
\psfrag{0.002}{\tiny $k=0.002$}
\psfrag{0.004}{\tiny $k=0.004$}
\psfrag{0.008}{\tiny $k=0.008$}
\psfrag{mNconv}[bc][tc]{\tiny Convergence of $\bm_{hk}$ with respect to $h$.}
\psfrag{HNconv}[bc][tc]{\tiny Convergence of $\bH_{hk}$ with respect to $h$.}
\includegraphics[width=0.4\textwidth]{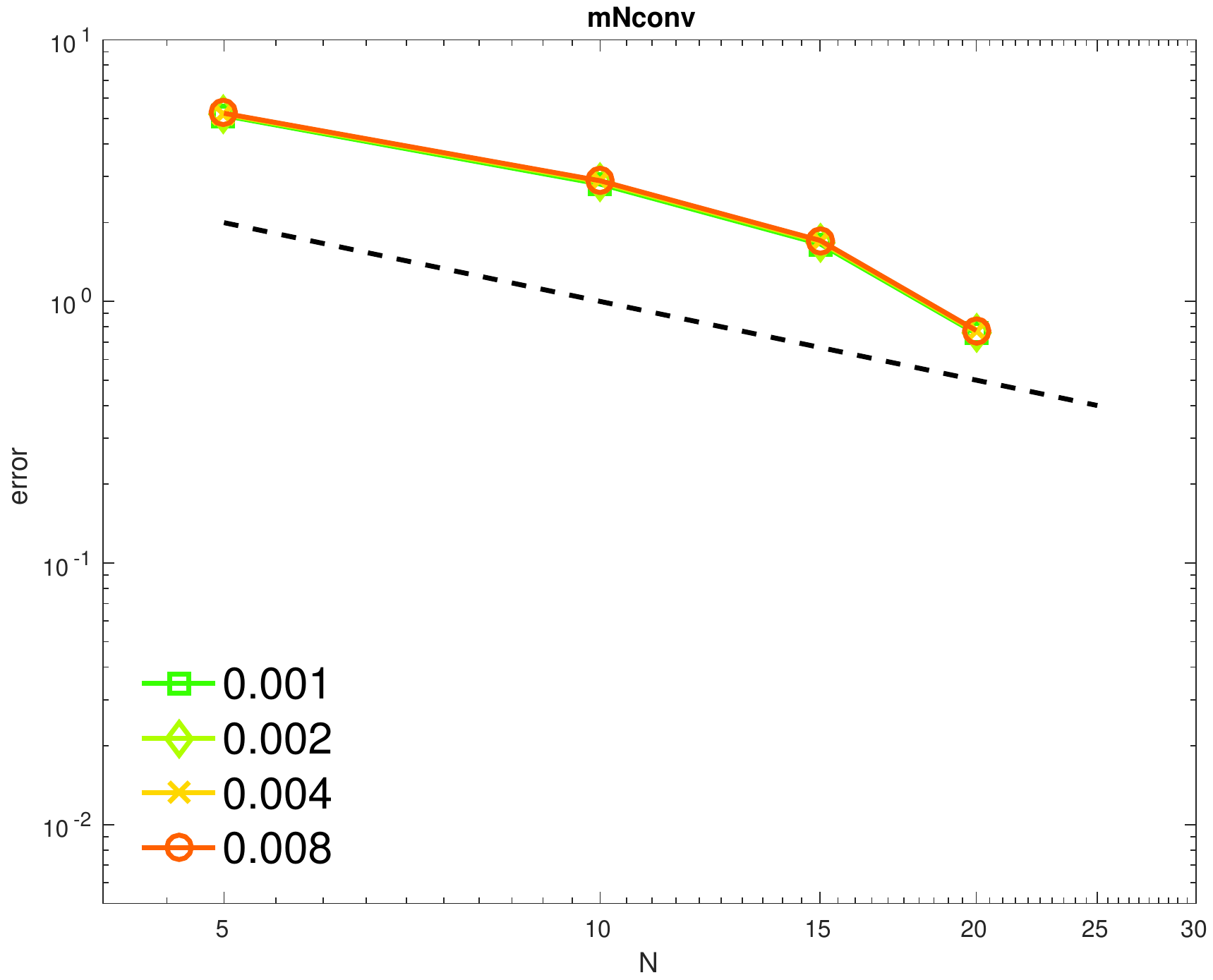}%
\qquad\includegraphics[width=0.4\textwidth]{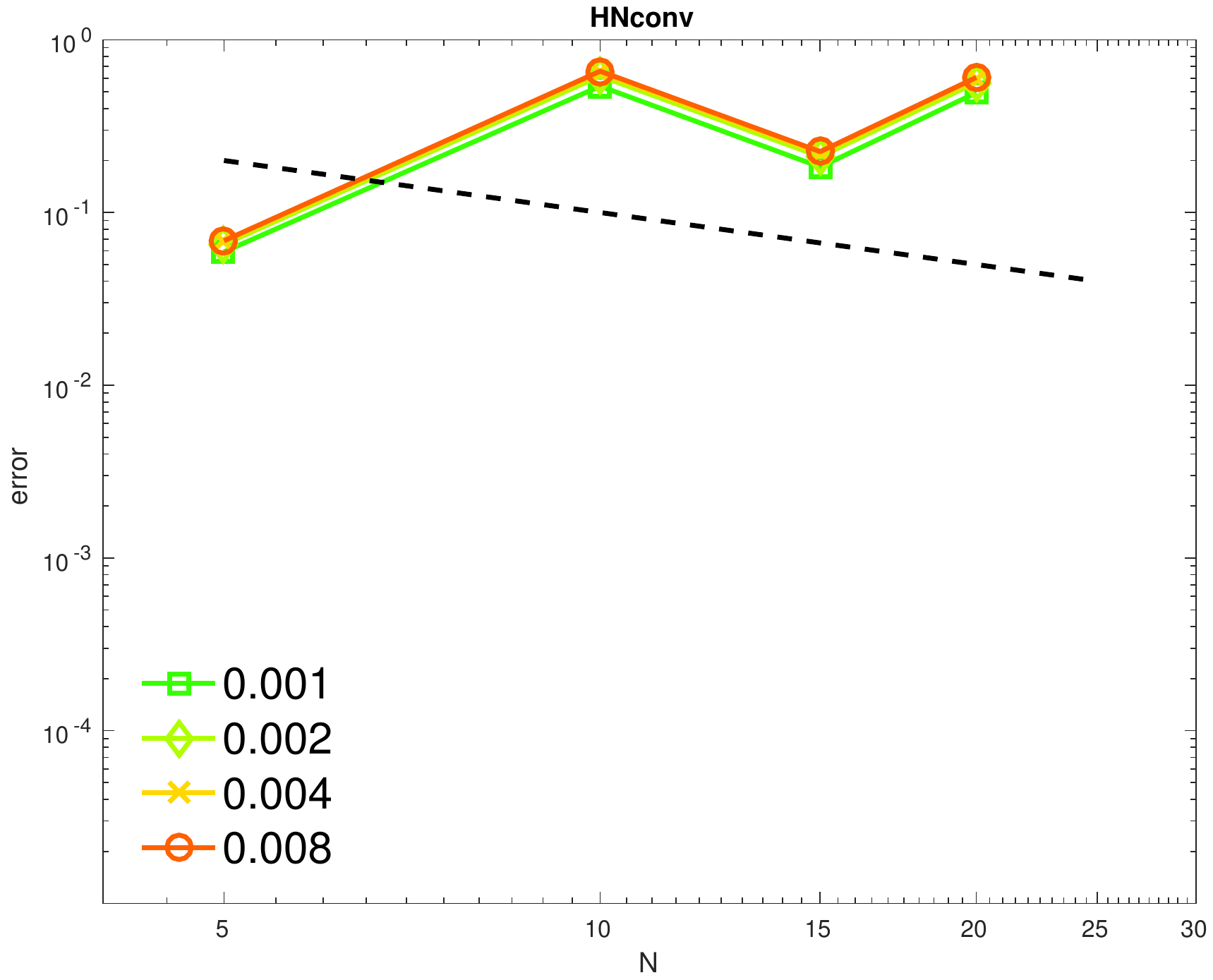}
\caption{The convergence rate of $\norm{\bm_\star-\bm_{hk}}{L^\infty(0,T;H^1(D))}$ (left)
and $\norm{\bH_\star-\bH_{hk}}{L^2(0,T;\Hcurl{D})}$ (right), where $\bm_\star=\bm_{hk}$ and $\bH_\star=\bH_{hk}$ for $h=1/25$ and $k=0.001$. The dashed line indicates $\mathcal{O}(h)$.}
\label{fig:MHNconv}
\end{figure}

\begin{figure}
\centering
\psfrag{error}[cc][cc]{\tiny error}
\psfrag{1/k}[cc][cc]{\tiny number time intervals $\simeq 1/k$}
\psfrag{m}{\tiny $\bm_{hk}$}
\psfrag{H}{\tiny $\bH_{hk}$}
\psfrag{mHKconv}[bc][tc]{\tiny Convergence  with respect to $k$.}
\psfrag{hKconv}[bc][tc]{\tiny Convergence of $\bH_{hk}$ with respect to $k$.}
\includegraphics[width=0.4\textwidth]{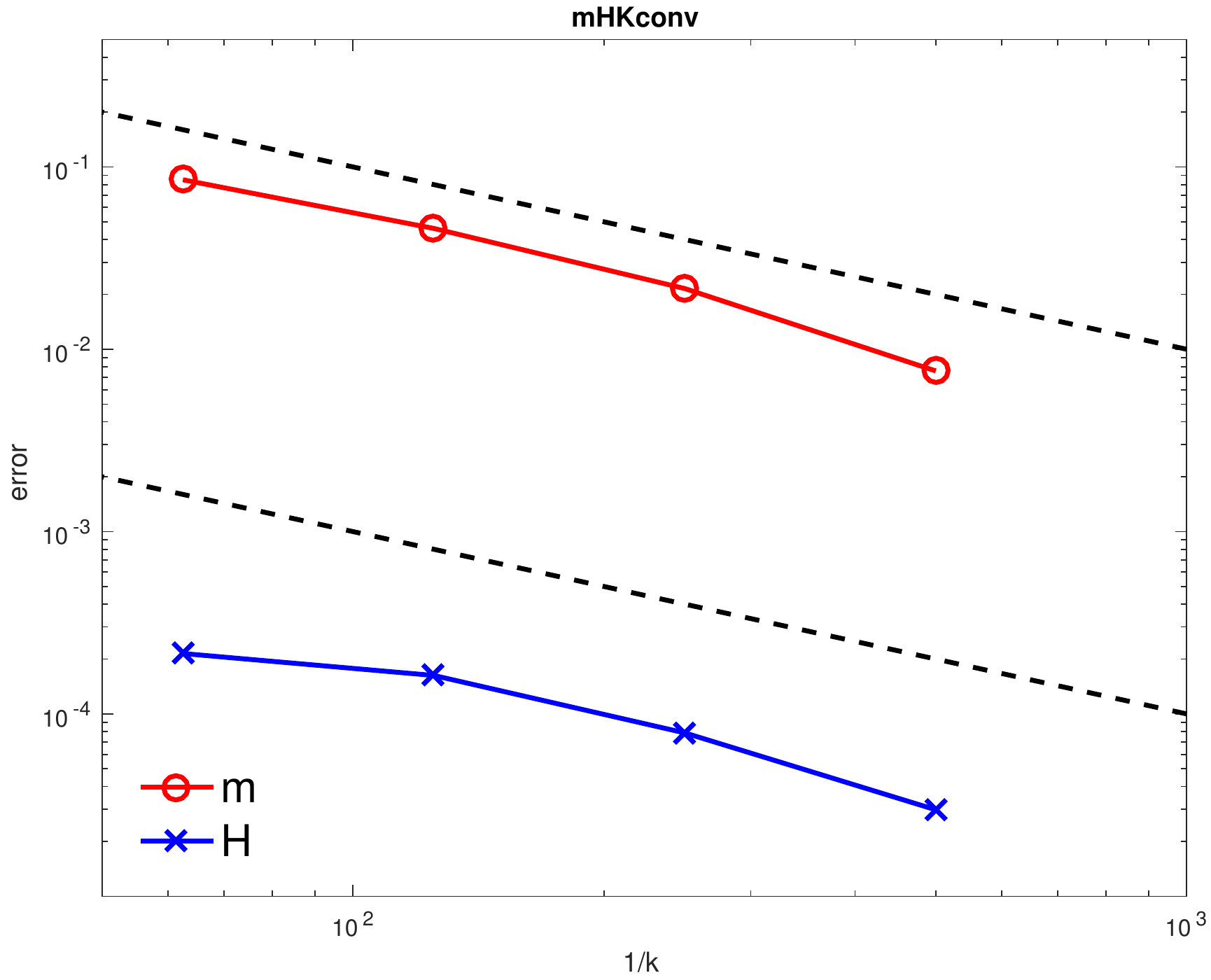}%
\caption{The convergence rate of $\norm{\bm_\star-\bm_{hk}}{L^\infty(0,T;H^1(D))}$ (red)
and $\norm{\bH_\star-\bH_{hk}}{L^2(0,T;\Hcurl{D})}$ (blue), where $\bm_\star=\bm_{h k_\star}$ and $\bH_\star=\bH_{hk_\star}$ for $h=1/25$ and $k_\star=0.001$. The dashed lines indicate $\mathcal{O}(k)$.}
\label{fig:MHkconv}
\end{figure}

\appendix
\section*{Appendix} Below, we state some well-known results.
\begin{lemma}\label{lem:h1seminorm}
Given $\TT_h$, there exists a constant $C_{\rm norm}>0$ which depends solely on the shape regularity of $\TT_h$ such that
\begin{align*}
C_{\rm norm}^{-1} \norm{\nabla \bw}{\Ltwo{D}}\leq \Big(h\sum_{T\in\TT_h}\sum_{z\in \NN_h\cap T} |\bw(z)-\bw(z_T)|^2\Big)^{1/2}\leq C_{\rm norm}\norm{\nabla\bw}{\Ltwo{D}}
\end{align*}
for all $\bw\in\SS^1(\TT_h)$ and some arbitrary choice of nodes $z_T\in \NN_h\cap T$ for all $T\in\TT_h$.
\end{lemma}
\begin{proof}
\revision{The proof follows from scaling arguments.}
% Let $\widehat T\subseteq \R^3$ be a reference tetrahedron and
% $\widehat\NN$ be the set of all its vertices. For a fixed $z_0\in
% \widehat\NN$ we define a seminorm
%  \begin{align*}
%   |\widehat\bw|_{z_0}^2:=\sum_{z\in \widehat\NN}
% |\widehat\bw(z)-\widehat\bw(z_0)|^2\quad\text{for all }\widehat\bw\in
% \PP^1(\widehat T).
%  \end{align*}
% %Then $|\widehat\bw|_{z_0}=0$ if and only if $\widehat\bw$ is constant
% %since the vectors~$z-z_0$ for $z\in\widehat\NN$ span~$\R^3$. 
% Hence, $\norm{\nabla(\cdot)}{\Ltwo{D}}$ and $|\cdot|_{z_0}$ 
% are norms on $\PP^1(\widehat T)/\R$ and thus they are
% equivalent. For the same reason, all four possible choices of $z_0$ yield equivalent seminorms $|\cdot|_{z_0}$. Given $T\in\TT_h$  let $\phi_T\colon \widehat T\to T$ denote the affine element map.  A scaling argument and the above considerations show
% \begin{align*}
% \norm{\nabla\bw}{\Ltwo{T}}^2
% \simeq 
% h
% \norm{\nabla(\bw\circ\phi_T)}{\Ltwo{\widehat{T}}}^2
% \simeq 
% h
% |\bw\circ\phi_T|_{\phi_T^{-1}(z_T)}^2
% = 
% h
% \sum_{z\in \NN_h\cap T} |\bw(z)-\bw(z_T)|^2,
% \end{align*}
% where $\phi_T^{-1}(z_T)\in\widehat\NN$.
% Summing over all elements $T\in\TT_h$ concludes the proof.
\end{proof}

\begin{lemma}\label{lem:dis Gro}
If $\{a_i\}$, $\{b_i\}$, $\{c_i\}$ are sequences of non-negative numbers
satisfying
\[
a_{i+1}\leq (1+b_i)a_i+c_i\quad\text{for all }i\in\N_0
\]
then for all $j\in\N_0$ there holds
\[
a_j
\leq 
\exp\big(\sum_{i=0}^{j-1}b_i\big)
\big(
a_0 + \sum_{i=0}^{j-1}c_i
\big).
\]
\end{lemma} 
\begin{proof}
The lemma can be easily shown by induction.
\end{proof}

\bibliographystyle{siamplain}
\bibliography{literature}
\end{document}